\documentclass[11pt,a4paper]{amsart} 
\usepackage{latexsym,color,amssymb,amscd,pinlabel,yhmath,graphicx,nicefrac}

\usepackage[all]{xy}
\SelectTips{cm}{}

\begin{document}
%
%
%

%
%

\theoremstyle{plain} 
\newtheorem{theorem}{Theorem}[section] 
\newtheorem{lemma}[theorem]{Lemma} 
\newtheorem{proposition}[theorem]{Proposition} 
\newtheorem{corollary}[theorem]{Corollary} 
\newtheorem{claim}[theorem]{Claim}

\theoremstyle{definition}
\newtheorem{definition}[theorem]{Definition} 
\newtheorem{example}[theorem]{Example}
\newtheorem{rem}[theorem]{Remark} 
\newtheorem{summary}[theorem]{Summary} 
\newtheorem{remark}[theorem]{Remark}
\newtheorem{convention}[theorem]{Convention}
\newtheorem*{question}{Question}
\newtheorem*{notation}{Notation}
\newtheorem*{acknowledgement}{Acknowledgements} 

\numberwithin{equation}{section}
\numberwithin{figure}{section}

%
%

\newcommand{\set}[1]{\lfloor #1\rceil}
\newcommand{\phin}[2]
{  \begin{array}{c} 
\labellist \small \hair 2pt 
\pinlabel {\scriptsize $#1$} [B] at 3 3
\pinlabel {\scriptsize $#2$} [B] at 37 3
\endlabellist
\includegraphics[scale=0.5]{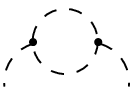}\\
\end{array}  }

\newcommand{\phio}[2]
{  \begin{array}{c} 
\labellist \small \hair 2pt 
\pinlabel {\scriptsize $#1$} [B] at 3 3
\pinlabel {\scriptsize $#2$} [B] at 37 3
\endlabellist
\includegraphics[scale=0.5]{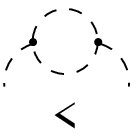}\\
\end{array}  }

\newcommand{\yn}[3]
{  \begin{array}{c} 
\labellist \small \hair 2pt 
\pinlabel {\scriptsize $#1$} [b] at 0 33
\pinlabel {\scriptsize $#2$} [b] at 27 0
\pinlabel {\scriptsize $#3$} [b] at 52 33
\endlabellist
\includegraphics[scale=0.5]{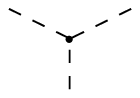}\\
\end{array}  }

\newcommand{\yo}[3]
{  \begin{array}{c} 
\labellist \small \hair 2pt 
\pinlabel {\scriptsize $#1$} [b] at 3 0
\pinlabel {\scriptsize $#2$} [b] at 27 0
\pinlabel {\scriptsize $#3$} [b] at 52 0
\endlabellist
\includegraphics[scale=0.5]{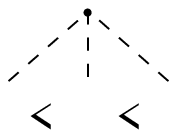}\\
\end{array}  }

\newcommand{\hn}[4]
{  \begin{array}{c} 
\labellist \small \hair 2pt 
\pinlabel {\scriptsize $#1$} [B] at 3 0
\pinlabel {\scriptsize $#2$} [B] at 27 0
\pinlabel {\scriptsize $#3$} [B] at 52 0
\pinlabel {\scriptsize $#4$} [B] at 78 0
\endlabellist
\includegraphics[scale=0.5]{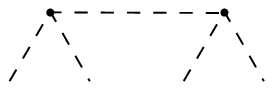}\\
\end{array}  }

\newcommand{\ho}[4]
{  \begin{array}{c} 
\labellist \small \hair 2pt 
\pinlabel {\scriptsize $#1$} [b] at 3 0
\pinlabel {\scriptsize $#2$} [b] at 27 0
\pinlabel {\scriptsize $#3$} [b] at 52 0
\pinlabel {\scriptsize $#4$} [b] at 78 0
\endlabellist
\includegraphics[scale=0.5]{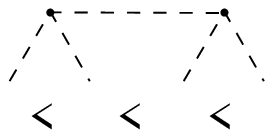}\\
\end{array}  }

\newcommand{\hob}[4]
{  \begin{array}{c} 
\labellist \small \hair 2pt 
\pinlabel {\scriptsize $#1$} [B] at 3 2
\pinlabel {\scriptsize $#2$} [B] at 29 2
\pinlabel {\scriptsize $#3$} [B] at 54 2
\pinlabel {\scriptsize $#4$} [B] at 80 2
\endlabellist
\includegraphics[scale=0.5]{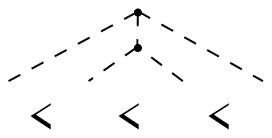}\\
\end{array}  }

\newcommand{\figtotext}[3]{\begin{array}{c}\includegraphics[width=#1pt,height=#2pt]{#3}\end{array}}

\newcommand{\thetagraph}{\hspace{-0.15cm} \figtotext{12}{12}{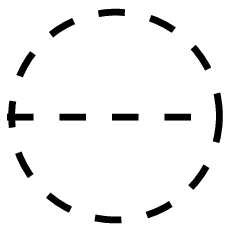} \hspace{-0.15cm}}

\newcommand{\ie}{i.e.\ }

\newcommand{\clocase}[1]{\wideparen{#1}} 

\newcommand{\fset}[1]{\lfloor #1\rceil}

%
%

\newcommand{\Q}{\mathbb{Q}}
\newcommand{\R}{\mathbb{R}}
\newcommand{\Z}{\mathbb{Z}}

\newcommand{\jacobi}{\mathcal{A}}    

\newcommand{\Torelli}{\mathcal{I}} 
\newcommand{\Johnson}{\mathcal{K}}  
\newcommand{\Next}{\mathcal{L}} 
\newcommand{\manifolds}{\mathcal{V}}
\newcommand{\mcg}{\mathcal{M}}          
\newcommand{\mcyl}{\mathbf{c}}
\newcommand{\Jcob}{\mathcal{KC}}
\newcommand{\cyl}{\mathcal{IC}}
\newcommand{\cob}{\mathcal{C}}

\newcommand{\Aut}{\operatorname{Aut}}
\newcommand{\dd}{\operatorname{d}}
\newcommand{\DD}{\operatorname{D}}
\newcommand{\cEul}{\operatorname{Eul}_{\operatorname{c}}}
\newcommand{\gEul}{\operatorname{Eul}_{\operatorname{g}}}
\newcommand{\Fr}{\operatorname{Fr}}
\newcommand{\Gr}{\operatorname{Gr}}
\newcommand{\Hom}{\operatorname{Hom}}
\newcommand{\Id}{\operatorname{Id}}
\newcommand{\Img}{\operatorname{Im}}
\newcommand{\incl}{\operatorname{incl}}
\newcommand{\interior}{\operatorname{int}}
\newcommand{\Ker}{\operatorname{Ker}}
\newcommand{\Lk}{\operatorname{Lk}}
\newcommand{\N}{\mathring{\operatorname{N}}}
\newcommand{\ord}{\operatorname{ord}}
\newcommand{\RT}{\operatorname{RT}}
\newcommand{\Sp}{\operatorname{Sp}}
\newcommand{\Spin}{\operatorname{Spin}}
\newcommand{\Tors}{\operatorname{Tors}}

\newcommand{\Lie}{\mathfrak{L}}

\title[]{Equivalence relations for homology cylinders\\ and the core of the Casson invariant}

\date{February 28, 2012}

\author[]{Gw\'ena\"el Massuyeau}
\address{Institut de Recherche Math\'ematique Avanc\'ee, Universit\'e de Strasbourg \& CNRS,
7 rue Ren\'e Descartes, 67084 Strasbourg, France}
\email{massuyeau@math.unistra.fr}

\author[]{Jean--Baptiste Meilhan}
\address{Institut Fourier, Universit\'e de Grenoble 1 \& CNRS, 
100 rue des Maths -- BP 74, 38402 Saint Martin d'H\`eres, France}
\email{jean-baptiste.meilhan@ujf-grenoble.fr}

\begin{abstract}
Let $\Sigma$ be a compact oriented surface of genus $g$ with one boundary component.
Homology cylinders over $\Sigma$ form a monoid $\cyl$ into which 
the Torelli group $\Torelli$ of $\Sigma$ embeds by the mapping cylinder construction.
Two homology cylinders $M$ and $M'$ are said to be \emph{$Y_k$-equivalent} 
if $M'$  is obtained from $M$ by ``twisting''  an  arbitrary surface $S\subset M$ 
with a homeomorphim belonging to the $k$-th term of the lower central series of the Torelli group of $S$.
The \emph{$J_k$-equivalence} relation on $\cyl$ is defined in a similar way using the $k$-th term of the Johnson filtration.
In this paper, we characterize the $Y_3$-equivalence with three classical invariants:
(1) the action on the third nilpotent quotient of the fundamental group of $\Sigma$, 
(2) the quadratic part of the relative Alexander polynomial, 
and 
(3) a by-product of the Casson invariant.
Similarly, we show that the $J_3$-equivalence is classified by (1) and (2).
We also prove that the core of the Casson invariant
(originally defined by Morita on the second term of the Johnson filtration of $\Torelli$) 
has a unique extension (to the corresponding submonoid of $\cyl$)
that is preserved by $Y_3$-equivalence and the mapping class group action.
\end{abstract}

\maketitle

\section{Introduction} \label{sec:intro}

Let $\Sigma$ be a compact connected oriented surface with one boundary component,
and let $g\geq 0$ be the genus of $\Sigma$.
A \emph{homology cobordism} of $\Sigma$ is a pair $(M,m)$ where  $M$ is a compact connected oriented $3$-manifold 
and $m: \partial\left(\Sigma \times [-1,1] \right) \to \partial M$ is an orientation-preserving homeomorphism such that 
the inclusions $m_\pm: \Sigma \to M$ defined by $x \mapsto m(x,\pm 1)$ induce isomorphisms $H_*(\Sigma;\Z) \to H_*(M;\Z)$.
Thus the $3$-manifold $M$ is a cobordism (with corners) between $\partial_+ M := m_+(\Sigma)$ and  $\partial_- M := m_-(\Sigma)$.
It is convenient to denote the cobordism $(M,m)$ simply by $M$,
the convention being that the boundary parametrization is always denoted by the lower-case letter $m$.
In particular, we shall denote the \emph{trivial cobordism} $(\Sigma \times [-1,1], \Id)$  simply by $\Sigma \times [-1,1]$.
The set of homeomorphism classes of homology cobordisms of $\Sigma$ is denoted by
$$
\cob := \cob(\Sigma),
$$
where two homology cobordisms $M,M'$ are considered \emph{homeomorphic}
if there is an orientation-preserving homeomorphism $f:M \to M'$ such that $f|_{\partial M} \circ m = m'$.
The \emph{composition} of two cobordisms $M$ and $M'$ is defined by ``stacking'' $M'$ on the top of $M$, 
\ie we define
\begin{displaymath}
M \circ M'  := M \cup_{m_+ \circ (m'_-)^{-1}} M'
\end{displaymath}
with $\partial(M \circ M')$ parametrized in the obvious way. So  there is a monoid structure on $\cob$.

The \emph{mapping class group} of $\Sigma$ is
the group of isotopy classes of self-homeomorphisms of $\Sigma$ that leave the boundary pointwise invariant. We shall denote it by
$$
\mcg := \mcg(\Sigma).
$$
The mapping cylinder construction $\mcyl: \mcg \to \cob$ is defined in the usual way by
\begin{equation}
\label{eq:mapping_cylinder}
\mcyl(s) := \big(\Sigma \times [-1,1],  (\Id \times (-1)) \cup (\partial \Sigma \times \Id) \cup (s \times 1)\big).
\end{equation}
Since the homomorphism $\mcyl$ is injective, we shall sometimes consider
the group $\mcg$ as a submonoid of $\cob$ and remove $\mcyl$ from our notation.
A base point  $\star$ being fixed on $\partial \Sigma$, 
the  mapping class group acts on the fundamental group $\pi := \pi_1(\Sigma,\star)$.
The resulting homomorphism
$
\rho: \mcg \to \Aut(\pi)
$
is known as the  \emph{Dehn--Nielsen representation} and is injective. For each $k\geq 0$, 
this representation induces a group homomorphism
\begin{equation}
\label{eq:rho_k_mcg}
\rho_k: \mcg \longrightarrow \Aut(\pi/\Gamma_{k+1} \pi)
\end{equation}
where $\pi = \Gamma_1 \pi \supset \Gamma_2 \pi \supset \Gamma_3 \pi \supset \cdots$
denotes the lower central series of $\pi$.
The \emph{Johnson filtration} of the mapping class group is  the decreasing sequence of subgroups
$$
\mcg = \mcg[0] \supset \mcg[1] \supset \mcg[2] \supset \mcg[3] \supset \cdots  
$$
where $\mcg[k]$ denotes the kernel of $\rho_k$ for all $k\geq 1$. The group $\pi$ being residually nilpotent,
the Johnson filtration has trivial intersection.
By virtue of Stallings' theorem \cite{Stallings}, 
each  homomorphism $\rho_k$ can be extended to the monoid $\cob$, so that we have a filtration 
$$
\cob = \cob[0] \supset \cob[1] \supset \cob[2] \supset \cob[3] \supset \cdots  
$$
of $\cob$  by submonoids \cite{GL_tree}.
But the intersection of this filtration is far from being trivial
since, for $g=0$,  we have $\cob[k] = \cob$ for all $k\geq 0$.

The first term $\mcg[1]$ of the Johnson filtration is known
as the \emph{Torelli group} of the surface $\Sigma$.
This is the subgroup of $\mcg$ acting trivially 
on the first homology group  $H:= H_1(\Sigma;\Z)$.
We shall denote it here by
$$
\Torelli := \Torelli(\Sigma).
$$
The study of this group, from a topological point of view, started with works of Birman \cite{Birman_Siegel}
and was followed  by Johnson in the eighties. 
The reader is referred to his survey \cite{Johnson_survey} for an overview of his work.
The Torelli group is residually nilpotent for the following reason:
the commutator of two subgroups $\mcg[k]$ and $\mcg[l]$ of the Johnson filtration is contained 
in $\mcg[k+l]$ for all $k,l\geq 1$, so that the lower central series of $\Torelli$ 
$$
\Torelli = \Gamma_1 \Torelli \supset \Gamma_2 \Torelli \supset \Gamma_3 \Torelli \supset \cdots
$$
is finer than the Johnson filtration (\ie we have $\Gamma_k \Torelli \subset \mcg[k]$). 
The graded Lie ring 
$$
\Gr^\Gamma \Torelli := \bigoplus_{i\geq 1} \frac{\Gamma_i  \Torelli}{\Gamma_{i+1}  \Torelli}
$$
has been computed explicitly in degree $i=1$ by Johnson \cite{Johnson_abelianization}
and in degree $i=2$ with rational coefficients by Morita \cite{Morita_Casson1, Morita_Casson2}.
Computations in degree $i=2$ with integral coefficients have also been done by Yokomizo \cite{yokomizo}.  
Johnson's computation of the abelianized Torelli group $\Torelli/\Gamma_2 \Torelli$ 
involves the action of $\Torelli$ on $\pi/\Gamma_3 \pi$
as well as the Rochlin invariant of homology $3$-spheres
in the form of some homomorphisms introduced by Birman and Craggs \cite{BC}.
Morita's computation of $(\Gamma_2\Torelli/\Gamma_3 \Torelli)\otimes \Q$ 
involves refinements of the latter invariants, 
namely the action of  $\Torelli$ on $\pi/\Gamma_4 \pi$ and the Casson invariant.
Besides, Hain found a presentation of the graded  Lie algebra 
$\Gr^\Gamma\! \Torelli \otimes \Q$  using mixed Hodge theory \cite{Hain}.

Thus $3$-dimensional invariants play an important role in Johnson and Morita's works on the Torelli group.
In the same perspective, let us consider the monoid 
$$
\cyl := \cyl(\Sigma)
$$
of \emph{homology cylinders} over $\Sigma$,
\ie homology cobordisms $M$ such that $(m_-)_*^{-1}\circ (m_+)_*$ is the identity of $H$.
The Torelli group $\Torelli$ embeds into the monoid $\cyl$ by the homomorphism $\mcyl$
and, for  $g=0$, the monoid $\cyl$ can be identified with the monoid of homology $3$-spheres.
The monoid of homology cylinders has been introduced in great generality by Goussarov \cite{Goussarov} and Habiro \cite{Habiro}
in connection with finite-type invariants of $3$-manifolds. 
Their approach  to the monoid $\cyl$ and, at the same time, to the group $\Torelli$
has been the subject of  several recent works:  see the survey \cite{HM_survey}.
As a substitute to the lower central series for the monoid $\cyl$, 
Goussarov and Habiro consider the filtration
$$
\cyl = Y_1 \cyl \supset  Y_2 \cyl \supset  Y_3 \cyl \supset \cdots
$$
where the submonoid $Y_i \cyl$ consists of the homology cylinders that are $Y_i$-equivalent to $\Sigma \times [-1,1]$.
Here two compact oriented $3$-manifolds  $M$ and $M'$ (with parametrized boundary if any)
are said to be \emph{$Y_k$-equivalent} 
if $M'$ can be obtained from $M$ by ``twisting'' an arbitrary embedded surface $E$ in the interior of $M$ 
with an element of the $k$-th term of the lower central series of the Torelli group $\Torelli(E)$ of $E$.
(Here the surface $E$ has an arbitrary position in $M$, but it is assumed to be compact connected oriented with one boundary component.)
The identity $\cyl = Y_1 \cyl$ is not trivial \cite{Habiro,Habegger}.
This means in genus $g=0$ that any  homology $3$-sphere is $Y_1$-equivalent to $S^3$ 
which, according to Hilden \cite{Hilden}, has first been observed by Birman.
(More generally, the $Y_1$-equivalence is characterized for closed oriented $3$-manifolds by Matveev in  \cite{Matveev}.)
The ``clasper calculus'' developed by Goussarov and Habiro \cite{Goussarov_graphs,Habiro,GGP}
offers the appropriate tools to study the $Y_k$-equivalence relation, 
since this relation is generated by ``clasper surgery'' along graphs with $k$ nodes.
With these methods, Goussarov \cite{Goussarov,Goussarov_graphs} and Habiro \cite{Habiro} proved 
among other things that each quotient monoid $Y_i \cyl/ Y_{i+1}$ is an abelian group and that
$$
\Gr^Y \cyl := \bigoplus_{i\geq 1} \frac{Y_i  \cyl}{Y_{i+1}}
$$
has a natural structure of graded Lie ring. 
This has been computed explicitly in degree $i=1$ by Habiro in  \cite{Habiro} 
and the authors in \cite{MM}, where the $Y_2$-equivalence on $\cyl$ 
is shown to be classified by the action on $\pi/\Gamma_3 \pi$ and  the Birman--Craggs homomorphism.
Thus the determination of $\cyl/Y_2$ goes parallel 
to Johnson's computation of $\Torelli/\Gamma_2 \Torelli$ and the two groups happen to be isomorphic via $\mcyl$ (for $g\geq 3$).
Besides, diagrammatic descriptions of the graded Lie algebra $\Gr^Y \cyl \otimes \Q$ are obtained in \cite{HM_SJD}
using clasper calculus and the LMO homomorphism
$$
Z: \cyl \longrightarrow \jacobi(H_\Q)
$$
which takes values in a graded algebra $\jacobi(H_\Q)$ of diagrams
``colored'' by the symplectic vector space $H_\Q := H_1(\Sigma;\Q)$. 
This invariant of homology cylinders is derived from a functorial extension \cite{CHM} of the LMO invariant \cite{LMO},
so that it is universal among rational-valued finite-type invariants.\\

In this paper, we shall classify the $Y_3$-equivalence relation on $\cyl$.
In addition to the action $\rho_3(M) \in \Aut(\pi/\Gamma_4\pi)$
and to the Casson invariant $\lambda(M) \in \Z$
-- which have been both used by Morita for the computation of  $(\Gamma_2 \Torelli/ \Gamma_3\Torelli)\otimes \Q$ --
we need the Alexander polynomial of homology cylinders $M$ relative 
to their ``bottom'' boundary $\partial_- M$.
More precisely, we define the \emph{relative Alexander polynomial} of $M\in \cyl$ as
the order of the relative homology group of $(M,\partial_-M)$
with coefficients twisted by $m_{\pm,*}^{-1}:H_1(M;\Z) \to H$:
$$
\Delta(M,\partial_-M) := \ord H_1(M, \partial_-M;  \Z[H]) \ \in \Z[H]/\pm H.
$$
The multiplicative indeterminacy in $\pm H$ can be fixed 
by considering a relative version of the Reidemeister--Turaev torsion 
introduced by Benedetti and Petronio \cite{BP,FJR}.
For this refinement of the relative Alexander polynomial, it is necessary to fix 
an Euler structure on $(M,\partial_-M)$,
\ie a homotopy class of vector fields on $M$ with prescribed behaviour on the boundary.
We shall see that homology cylinders $M$ have a preferred Euler structure $\xi_0$, 
so that the class $\Delta(M,\partial_-M)$ has a preferred representative
$$
\tau(M,\partial_-M;\xi_0) \in \Z[H].
$$
This invariant of homology cylinders features the same finiteness properties
as the Reidemeister--Turaev torsion of closed oriented $3$-manifolds \cite{Massuyeau_torsion}. 
More precisely, if we denote by $I$ the augmentation ideal of $\Z[H]$, 
then the reduction of $\tau(M,\partial_-M;\xi_0) $ modulo $I^{k+1}$ is  for every $k \geq 0$
a finite-type invariant of degree $k$ in the sense of Goussarov and Habiro.
In particular, a ``quadratic part'' 
$$
\alpha(M) \in I^2/I^3 \simeq S^2H
$$
can be extracted from $\tau(M,\partial_-M;\xi_0)$ and is a finite-type invariant of degree $2$.
Then our characterization of the $Y_3$-equivalence  for homology cylinders takes the following form.

\smallskip\smallskip
\noindent {\bf Theorem~A.} 
{\it Let $M$ and $M'$ be two homology cylinders over $\Sigma$.  
The following assertions are equivalent:
\begin{enumerate}
\item[(a)] $M$ and $M'$ are $Y_3$-equivalent;
\item[(b)] $M$ and $M'$ are not distinguished by any Goussarov--Habiro finite-type invariants of degree at most $2$;
\item[(c)] $M$ and $M'$ share the same invariants $\rho_3, \alpha$ and $\lambda$;
\item[(d)] The LMO homomorphism $Z$ agrees on $M$ and $M'$ up to degree $2$.
\end{enumerate} 
}
\smallskip

\noindent
In genus $g=0$, the theorem asserts that two homology $3$-spheres are $Y_3$-equivalent
if and only if they have the same Casson invariant, which is due to Habiro \cite{Habiro}.
The equivalence between conditions (a), (b) and (d) is based 
on the universality of the LMO homomorphism among $\Q$-valued finite-type invariants, 
its good behaviour under clasper surgery and the torsion-freeness of a certain space of diagrams.
Next the equivalence of condition (c) with the other three follows
by determining precisely how the invariants  $\rho_3, \alpha$ and $\lambda$ 
are diagrammatically encoded in the LMO homomorphism.
We emphasize that the Birman--Craggs homomorphism, 
which is needed to classify the $Y_2$-equivalence, do not appear explicitly 
in the above statement for the $Y_3$-equivalence:
indeed it is determined by the triplet $(\rho_3,\alpha,\lambda)$ as we shall see in detail.
Furthermore, we shall use the diagrammatic techniques of clasper calculus
to compute the group $\cyl/Y_3$ and show some of its properties.

We shall also be interested in the \emph{$J_k$-equivalence} relation among homology cylinders.
This relation is defined for every $k\geq 1$ 
in a way similar to the $Y_k$-equivalence but using the Johnson filtration of the Torelli group 
instead of its lower central series. It follows from these definitions that  $J_1 =Y_1$
and that $Y_k$ implies $J_k$ for every $k\geq 1$. 
But the converse is not true for $k\geq 2$ as illustrated by the following statements.

\smallskip\smallskip
\noindent {\bf Theorem~B.} 
{\it Two homology cylinders $M$ and $M'$ over $\Sigma$ are $J_2$-equivalent  
if and only if we have $\rho_2(M)=\rho_2(M')$.
}
\smallskip\smallskip

\noindent
In genus $g=0$, the theorem asserts that any homology $3$-sphere is $J_2$-equivalent to $S^3$. 
This is already noticed by Morita in \cite{Morita_Casson1}  
and follows from Casson's observation that any homology $3$-sphere is obtained 
from $S^3$ by a finite sequence of surgeries along $(\pm1)$-framed knots \cite{GM}.
(A generalization of this result to closed oriented $3$-manifolds is proved by Cochran, Gerges and Orr in \cite{CGO}.)
Although it does not seem to have been observed before, Theorem~B easily follows
from the computation of  $\cyl/Y_2$ done in \cite{MM}.

\smallskip\smallskip
\noindent {\bf Theorem~C.} 
{\it Two homology cylinders $M$ and $M'$ over $\Sigma$ are $J_3$-equivalent  
if and only if we have $\rho_3(M)=\rho_3(M')$ and $\alpha(M)=\alpha(M')$.
}
\smallskip\smallskip

\noindent
In genus $g=0$, we obtain that any homology $3$-sphere is $J_3$-equivalent
to $S^3$, which is due to Pitsch \cite{Pitsch}. 
The proof of Theorem~C makes use of Theorem~A.

Let us now come back to the Casson invariant  of homology cylinders which appears in Theorem~A.
In contrast to $\rho_3$ and $\alpha$, the invariant $\lambda$ is not  canonical. Indeed it depends on the choice of 
a Heegaard embedding $j: \Sigma \hookrightarrow S^3$
and $\lambda(M) := \lambda_j(M)$ is defined as the Casson invariant of the homology $3$-sphere 
obtained by inserting $M$ in place of a regular neighborhood of $j(\Sigma)\subset S^3$.
In the case of the Torelli group, Morita considered the restriction of $\lambda$ to the \emph{Johnson subgroup}
$$
\Johnson := \Johnson(\Sigma)
$$
which coincides with the second term $\mcg[2]$ of the Johnson filtration.
He proved that $\lambda|_{\Johnson}$ is a group homomorphism which can be written as 
the sum of two homomorphisms $\Johnson \to \Q$: 
one of them is not canonical  and is determined by $\rho_3$,
while the other one is (up to a $1/24$ factor) a canonical homomorphism $d:\Johnson \to 8\Z$ 
and is called the \emph{core of the Casson invariant}.
In the case of homology cylinders, the second term $\cob[2]$ of the Johnson filtration of $\cob$ is denoted by
$$
\Jcob := \Jcob(\Sigma).
$$

\smallskip\smallskip
\noindent {\bf Theorem~D.} 
{\it 
Assume that $g\geq 3$.
Then there is a unique extension of the core of the Casson invariant to the monoid $\Jcob$\\[-0.5cm]
$$
\xymatrix{
\Johnson \ar[r]^-d \ar[d]_-\mcyl & 8\Z \\
\Jcob \ar@{-->}[ru]_-d & 
}
$$
that is invariant under $Y_3$-equivalence and under the action of the mapping class group by conjugation.
Moreover, the monoid homomorphism $d:\Jcob \to 8\Z $ 
can be written explicitly in terms of $\rho_3$, $\alpha$ and $\lambda$.
}
\smallskip\smallskip

\noindent
When $M\in \Jcob$ belongs to the Johnson subgroup $\Johnson$,
we have $\alpha(M)=0$ and our formula for $d(M)$ 
coincides with Morita's formula \cite{Morita_Casson1,Morita_Casson2}.
We shall also see that the assumption $g\geq 3$ in Theorem~D 
can be removed by stabilization of the surface $\Sigma$.

In another paper \cite{Morita_Casson3}, Morita gave 
a topological interpretation of $d(h)$ for $h\in \Johnson$ 
as the signature defect of the mapping torus of $h$ equipped with a certain $2$-framing.
It would be very interesting to generalize this intrinsic description of
$d: \Johnson \to 8\Z$ to the monoid $\Jcob$. See \cite{DP} in this connection.\\[-0.5cm]

The paper is organized as follows:
{\small \tableofcontents}

\newpage

In the rest of the paper, we shall use the following conventions.
We denote by $I$ the interval $[-1,1]$.
An abelian group $G$, or its action on a set, is written additively, except when it 
is seen as a subgroup of the group of units of $\Z[G]$.
Besides, (co)homology groups are taken with $\Z$ coefficients if no coefficients are  specified. 

\begin{acknowledgement}
The authors would like to thank Kazuo Habiro for stimulating discussions
about the diagrammatic description of the group $\cyl/Y_3$,
and for his comments on the first version of this manuscript.
They are also grateful to Takuya Sakasai
and the referee for their remarks. The first author was partially supported 
by the French ANR research project ANR-08-JCJC-0114-01.
The second author is supported by the French ANR research project ANR-11-JS01-00201.
\end{acknowledgement}

\section{Preliminaries on the equivalence relations} \label{sec:preliminaries}

In this section, we give the precise definitions of the $Y_k$-equivalence and the $J_k$-equivalence relations on $3$-manifolds,
and we recall their relationship with the Goussarov--Habiro theory of finite-type invariants.

\subsection{Definition of the equivalence relations} \label{subsec:definition_relations}

Let $R$ be a closed oriented  surface, which may be empty or disconnected.
We consider compact connected oriented $3$-manifolds $M$ whose boundary is \emph{parametrized} 
by $R$, \ie $M$ comes with an orientation-preserving homeomorphism $R \to \partial M$
which is denoted by the lower-case letter $m$.
Two such manifolds with parametrized boundary $M$ and $M'$ are considered \emph{homeomorphic}
if there is an orientation-preserving homeomorphism $f:M\to M'$ such that $f \circ m =m'$.
We denote by $\manifolds(R)$ the set of homeomorphism classes of compact connected
oriented $3$-manifolds with boundary parametrized by $R$.

One way to modify an $M \in \manifolds(R)$ is to choose
a  compact oriented connected surface $S \subset \interior(M)$ with one boundary component,
and an element  $s\in \Torelli(S)$ of the Torelli group of $S$. We then define
$$
M_s := \big(M \setminus \interior(S\times [-1,1])\big) \cup \mcyl(s)
$$
where $S\times [-1,1]$ denotes a regular neighborhood of $S$ in $M$ 
and $\mcyl(s)$ is the mapping cylinder of $s$ defined by (\ref{eq:mapping_cylinder}).
The boundary parametrization of $M_S$ is defined from $m$ in the obvious way.
The move $M \leadsto M_s$  in $\manifolds(R)$ is called a \emph{Torelli surgery}.

Let $k\geq 1$ be an integer,
and consider two compact connected oriented $3$-manifolds $M$ and $M'$ with boundary parametrized by $R$.
We say that $M$ is \emph{$Y_k$-equivalent} to $M'$
if there is a Torelli surgery $M \leadsto M_s$ such that $M_s=M' \in \manifolds(R)$ 
and the gluing  homeomorphism $s$ belongs to the $k$-th term $\Gamma_k \Torelli(S)$ of the lower central series of $\Torelli(S)$.
(Recall that the \emph{lower central series} of a group $G$ is
defined inductively by $\Gamma_1 G:=G$ and $\Gamma_{k+1} G :=[G,\Gamma_{k}G]$ for all $k\geq 1$.) 
It is easily checked that the $Y_k$-equivalence is an equivalence relation on the set $\manifolds(R)$.
The \emph{$J_k$-equivalence} relation on  $\manifolds(R)$ is defined in a similar way 
using the $k$-th term of the Johnson filtration instead of the lower central series.
Thus we have defined an infinity of equivalence relations, which are organized as follows: 
$$
\begin{array}{cccccccccccc}
Y_1 & \Longleftarrow & Y_2 & \Longleftarrow & Y_3 & \Longleftarrow & \cdots & Y_k &  \Longleftarrow & Y_{k+1} & \Longleftarrow & \cdots  \\
\parallel & & \Downarrow & & \Downarrow && &  \Downarrow & & \Downarrow && \\
J_1 & \Longleftarrow & J_2 & \Longleftarrow & J_3 & \Longleftarrow & \cdots & J_k &  \Longleftarrow & J_{k+1} & \Longleftarrow & \cdots 
\end{array}
$$
The weakest of these relations, namely the $Y_1$-equivalence, is already non-trivial
since a Torelli surgery  $M \leadsto M_s$  comes with a canonical isomorphism 
\begin{equation}
\label{eq:iso_homology}
\xymatrix{
H_1(M) \ar@{-->}[rr]^-{ \Phi_s}_-\simeq & & H_1(M_s),\\
&H_1 \big(M \setminus \interior(S\times [-1,1])\big) \ar@{->>}[lu]^-{\incl_*}  \ar@{->>}[ru]_-{\incl_*} & 
}
\end{equation}
whose existence is easily deduced from the Mayer--Vietoris theorem. 

In this paper, we shall restrict our study to the class of homology cylinders over $\Sigma$, 
as defined in the introduction, 
\ie to the subset 
$$
\cyl(\Sigma) \subset \manifolds\big(\partial(\Sigma\times [-1,1])\big).
$$
The $Y_k$-equivalence  and the $J_k$-equivalence relations are preserved 
by \emph{stabilization} of the surface $\Sigma$.
More precisely, assume that $\Sigma^s$ is the boundary connected sum of $\Sigma$ 
with another compact connected oriented surface $\Sigma'$ 
having one boundary component as shown in Figure \ref{fig:stabilization}.
Then, there is a canonical injection
\begin{equation}
\label{eq:stabilization}
\cyl(\Sigma) \longrightarrow \cyl(\Sigma^s), \ M \longmapsto M^s
\end{equation}
obtained by gluing to any homology cylinder $M$ over $\Sigma$ the trivial cylinder $\Sigma'\times I$ 
along the square $(\Sigma\cap \Sigma')\times I$.
Then the implications $M \stackrel{Y_k}{\sim} N \Rightarrow M^s \stackrel{Y_k}{\sim} N^s$
and $M \stackrel{J_k}{\sim} N \Rightarrow M^s \stackrel{J_k}{\sim} N^s$ obviously hold true for any $k\geq 1$.

\begin{figure}[h]
\begin{center}
{\labellist \small \hair 0pt 
\pinlabel {$\Sigma^s$} [t] at 572 -4 
\pinlabel {$\Sigma$} [lb] at 40 50
\pinlabel {$\Sigma'$} [lb] at 610 50
\endlabellist}
\includegraphics[scale=0.3]{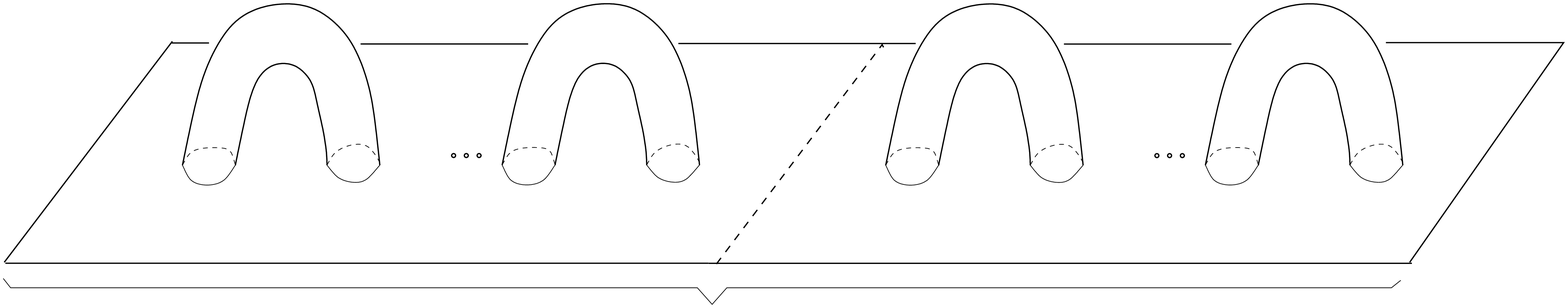}
\end{center}
\caption{A stabilization $\Sigma^s$ of the surface $\Sigma$.}
\label{fig:stabilization}
\end{figure}

The $Y_k$-equivalence  and the $J_k$-equivalence relations can also be defined in terms of Heegaard splittings.
Let us formulate this in the case of homology cylinders.
A homology cylinder $M$ over $\Sigma$ has a ``bottom surface'' $\partial_- M = m_-(\Sigma)$
and a ``top surface'' $\partial_+ M = m_+(\Sigma)$. Some collar neighborhoods of them
are suggestively denoted by $\partial_- M \times [-1,0]$ and $\partial_+ M \times [0,1]$.
A \emph{Heegaard splitting} of $M$  of \emph{genus} $r$ is a decomposition 
$$
M = M_- \cup M_+,
$$
where $M_-$ is obtained from $\partial_- M \times [-1,0]$ by adding $r$ $1$-handles along $\partial_- M \times \{0\}$, 
$M_+$ is obtained from $\partial_+ M \times [0,1]$ by adding $r$ $1$-handles along $\partial_+ M \times \{0\}$,
and $M_-\cap M_+ = \partial M_- \cap \partial M_+$ is called the \emph{Heegaard surface}.
Note that the Heegaard surface is a $2$-sided compact connected surface of genus $g+r$
with one boundary component which is properly embedded in $M$;
we give it the orientation inherited from $M_-$. 
Any homology cylinder $M$ has a Heegaard splitting,
since the cobordism $M$ can be obtained from the trivial cobordism
$\Sigma \times[-1,1]$ by attaching simultaneously some $1$-handles along the ``top surface'' $\Sigma \times \{1\}$ 
and, next, some $2$-handles along the new ``top surface''.
This fact follows from elementary Morse theory and is true for any $3$-dimensional cobordism  \cite{Milnor}.

\begin{lemma}
Two homology cylinders $M,M'$ over $\Sigma$ are $Y_k$-equivalent (respectively $J_k$-equivalent) 
if and only if there is a Heegaard splitting  $M=M_- \cup M_+$ with Heegaard surface $S$ 
and an  $s \in \Gamma_{k}\Torelli(S)$ (respectively an $s \in \mcg(S)[k]$) such that $M' = M_- \cup_s M_+$.
\end{lemma}

\begin{proof}
Assume that $M'=M_e \in \cyl$ where $M_e$ is the result of a Torelli surgery along a compact connected oriented surface 
$E \subset \interior(M)$ with one boundary component. We consider a regular neighborhood $E \times [-1,1]$ of $E$ in $M$
that does not meet the collar neighborhood $\partial_-M \times [-1,0]$, and where $E\times\{0\}$ is the surface $E$ itself.
Next we connect $E \times [-1,0]$ to $\partial_-M \times [-1,0]$ by a solid tube $T$:
more precisely,  $T$ meets $E \times [-1,0]$ along a disk of $E \times \{-1\}$
and it meets $\partial_-M \times [-1,0]$ along a disk of $\partial_-M \times \{0\}$. 
Thus the union 
$$
L_- := \left(\partial_-M \times [-1,0]\right) \cup T \cup \left(E \times [-1,0]\right)
$$
is obtained from $\partial_-M \times [-1,0]$ by attaching some $1$-handles (twice the genus of $E$).
Let $L_+$ be the closure in $M$ of $M\setminus L_-$,
which we regard  as a cobordism with corners from $R:=L_-\cap L_+$ to $\partial_+ M$.
This cobordism has a handle decomposition
\begin{equation}\label{eq:L+}
L_+=\big(\left(R\times [0,1]\right) \cup \hbox{$1$-handles}\big)\cup  \hbox{$2$-handles}.
\end{equation}
Note that the surface $R \subset M$ contains $E$
and, after an isotopy, we can assume that the $1$-handles in (\ref{eq:L+}) 
are attached to $R\times \{1\}$ outside the surface $E\times \{1\}$.
Now, we can ``stretch'' each of these $1$-handles towards $R \times \{0\}$, 
where by ``stretching'' a $1$-handle $[-1,1]\times D^2$ we mean replacing it with 
$[-1-\varepsilon,1+\varepsilon]\times D^2$ for some positive $\varepsilon$, so that 
the ``stretched'' $1$-handles are now attached to $R\times \{0\}$ outside the surface $E\times \{0\}$.
Furthermore, we can ``contract'' the resulting $1$-handles so that they are all disjoint from $\partial_+M$. 
Here, by ``contracting'' a $1$-handle $D^1\times D^2$ we mean replacing it with 
$D^1\times (\varepsilon D^2)$ for some  $\varepsilon\in\, ]0,1[$. 
We now define $M_-$ to be the union of $L_-$ with these ``stretched'' and ``contracted'' $1$-handles, 
and we define $M_+$ as the exterior in $L_+$ of those $1$-handles. 
Thus we have found a Heegaard splitting $M_-\cup M_+$ of $M$,
whose Heegaard surface $S:= M_-\cap M_+$ contains $E$ as a subsurface. The conclusion easily follows.
\end{proof}

\subsection{Description by clasper surgery} \label{subsec:clasper}

Generators for the $Y_k$-equivalence relations are known, 
which makes these easier to study than the $J_k$-equivalence relations.
Indeed the $Y_k$-equivalence is generated by surgery along graph claspers of degree $k$.
This viewpoint, which we shall briefly recall, 
has been developed by Goussarov \cite{Goussarov,Goussarov_graphs} and Habiro \cite{Habiro}:
the $Y_k$-equivalence relation is the same as the ``$(k-1)$-equivalence'' in \cite{Goussarov}
or the ``$A_k$-equivalence'' in \cite{Habiro}.

Let $M$ be a compact oriented $3$-manifold. In the terminology of \cite{Habiro},
a {\em graph clasper}  in $M$ is a compact, connected surface $G$ embedded in  $\interior(M)$, 
which comes decomposed between {\em leaves}, {\em nodes} and {\em edges}.  Leaves are annuli and nodes are discs. 
Edges are $1$-handles connecting leaves and nodes,
so that each edge has two ``ends" (the attaching locus of the $1$-handle).
Each leaf should have exactly one end of an edge attached to it,
while each node should have exactly three ends of edges attached to it.
See Figure \ref{fig:graph_clasper} for an example of a graph clasper.

\begin{figure}[h]
\begin{center}
{\labellist \small \hair 0pt 
\pinlabel {an edge} [l] at 94 21
\pinlabel {a node} [l] at 190 333
\pinlabel {a leaf} [l] at 608 165
\pinlabel {$=$} at 809 90
\endlabellist}
\includegraphics[scale=0.3]{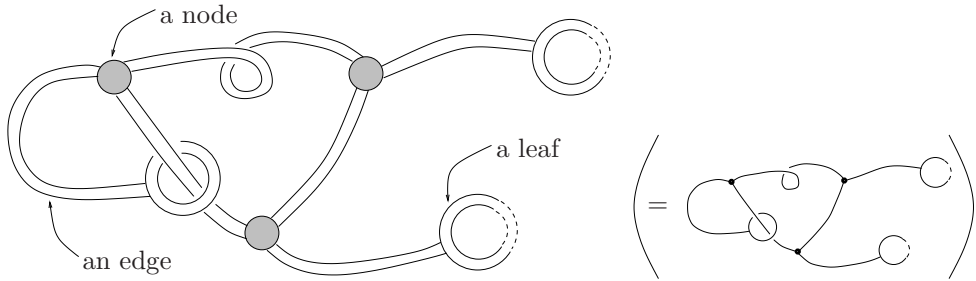}
\end{center}
\caption{An example of graph clasper with $3$ nodes, $3$ leaves and $6$ edges.
(And how it is drawn with the blackboard framing convention.)}
\label{fig:graph_clasper}
\end{figure}

\noindent
Given a graph clasper $G\subset M$,  one can forget the leaves of $G$ and
collapse the rest to a one-dimensional graph.
This finite graph, which has vertices of valency $1$ or $3$, is called the \emph{shape} of $G$.
Then one can be interested in graph claspers of a specific shape. 
For example, a \emph{$Y$-graph} is a graph clasper with shape $\textsf{Y}$,
and a graph clasper is said to be \emph{looped} if its shape contains a loop (as is the case in Figure \ref{fig:graph_clasper}).

The {\em degree} of a graph clasper $G$ is  the number of nodes contained in $G$. 
Graph claspers of degree $0$ are called \emph{basic claspers}  in \cite{Habiro} and consist of only one edge and two leaves:
$$
\includegraphics[scale=0.3]{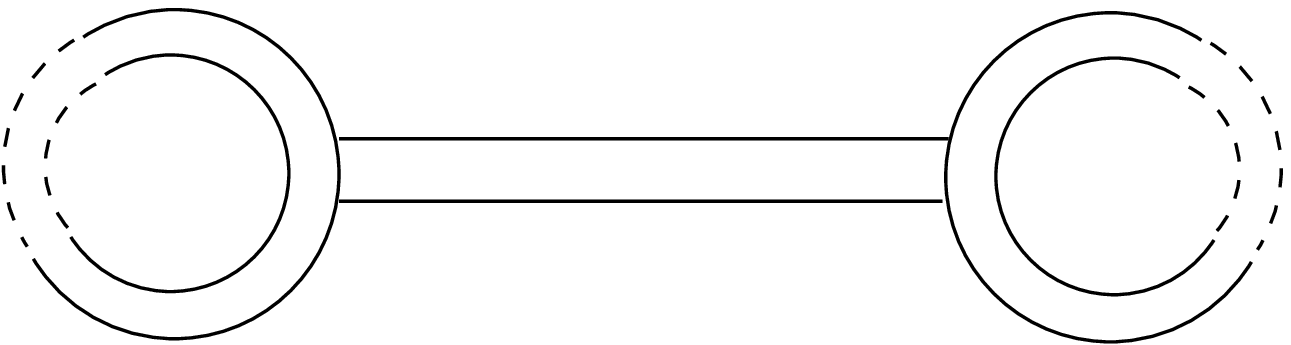}
$$
Surgery along a graph clasper $G\subset M$ is defined as follows.  
We first replace each node with three leaves linking like the Borromean rings in the following way: 
\begin{center}
{\labellist \small \hair 0pt 
\pinlabel {$\longrightarrow$} at 409 143
\endlabellist}
\includegraphics[scale=0.3]{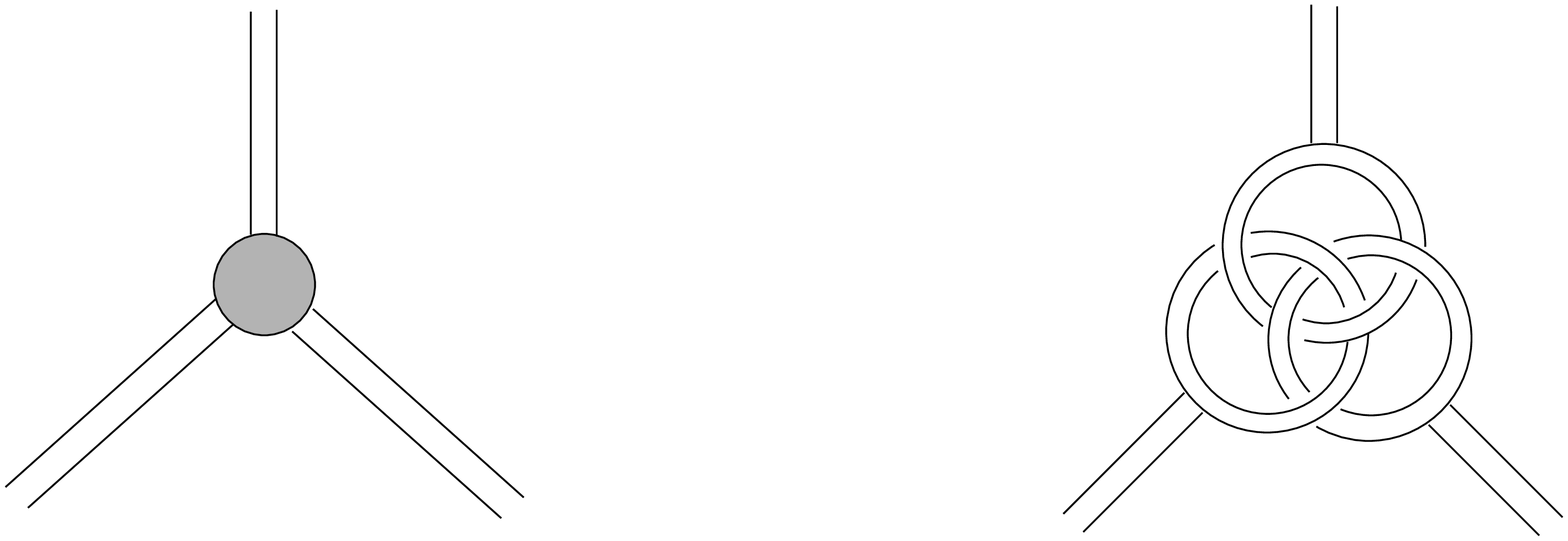}
\end{center}
Thus, we obtain a disjoint union of basic claspers.    
Next, we replace each basic clasper with a $2$-component framed link as follows:
\begin{center}
{\labellist \small \hair 0pt 
\pinlabel {$\longrightarrow$} at 460 50
\endlabellist}
\includegraphics[scale=0.3]{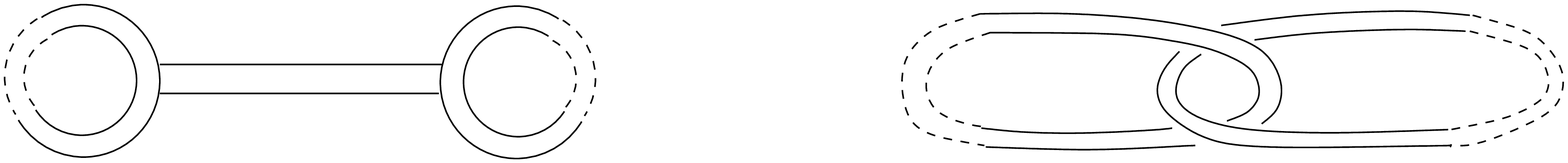}
\end{center}
Then, \emph{surgery} along the graph clasper $G$ is defined to be the surgery along the framed link thus obtained in $M$.
The resulting $3$-manifold is denoted by $M_G$.

\begin{proposition}[Habiro \cite{Habiro}] \label{prop:Y_clasper}
For any integer $k\geq 1$, the $Y_k$-equivalence relation is generated by surgeries along graph claspers of degree $k$.
\end{proposition}

\noindent
See the appendix of \cite{Massuyeau_DSP} for a proof. 

A ``calculus of claspers'' is developed in \cite{Goussarov_graphs,Habiro,GGP},
in the sense that some specific ``moves'' between graph claspers are shown to produce  by surgery homeomorphic $3$-manifolds.
This calculus can be regarded as a topological analogue of the commutator calculus in groups \cite{Habiro}.
Thanks to Proposition \ref{prop:Y_clasper}, this calculus has proved to be a very efficient tool for the study of the $Y_k$-equivalence relations 
and, here again, we shall use it in a crucial way.
For the reader's convenience, we have collected all the technical results on claspers that we shall need in  Appendix \ref{sec:calc_clasper}.  
In particular, we need a number of relations satisfied by graph claspers $G$ with \emph{special} leaves, 
\ie leaves which bound disks disjoint from $G$ and which are $(-1)$-framed.

\subsection{Relationship with finite-type invariants}

The $Y_k$-equivalence relations are closely  connected to finite-type invariants.
Here, we are referring to the  Goussarov--Habiro theory of finite-type invariants 
for compact oriented $3$-manifolds \cite{Goussarov,GGP,Habiro}, 
which essentially generalizes Ohtsuki's theory \cite{Ohtsuki} for  homology $3$-spheres
but differs from the Cochran--Melvin theory \cite{CM}.

We fix as in \S \ref{subsec:definition_relations} a closed oriented surface $R$.
We also consider a $Y_1$-equivalence class $\mathcal{Y} \subset \manifolds(R)$.
An invariant $f: \mathcal{Y} \to A$ of  manifolds in this class with values in an abelian group $A$
is said to be a \emph{finite-type invariant} of \emph{degree} at most $d$ if we have
$$
\sum_{P \subset \{0,\dots,d\}}(-1)^{|P|} \cdot f(M_{P}) = 0 \ \in A
$$\\[-0.2cm]
for any $M \in \mathcal{Y}$, for any pairwise-disjoint surfaces $S_0 \sqcup \cdots \sqcup S_{d} \subset \interior(M)$ 
and for any $s_0 \in \Torelli(S_0),\dots, s_d \in \Torelli(S_d)$, 
where $M_{P}\in \manifolds(R)$ is obtained from $M$ by simultaneous Torelli surgeries  along the surfaces $S_p$ for which $p\in P$.
In other words, $f$ behaves like a  ``polynomial'' map of degree at most $d$ with respect to Torelli surgeries.
The original definition by Goussarov \cite{Goussarov} and Habiro \cite{Habiro} involves
clasper surgery instead of Torelli surgery, 
but it follows from Proposition \ref{prop:Y_clasper} that the two definitions are equivalent.

\begin{lemma}[Goussarov \cite{Goussarov}, Habiro \cite{Habiro}] \label{lem:Y_FTI}
Let $M,M' \in \mathcal{Y}$. If $M$ and $M'$ are $Y_{d+1}$-equivalent, 
then we have $f(M)=f(M')$ for any finite-type invariant $f: \mathcal{Y} \to A$ 
of degree at most $d$.
\end{lemma}

\begin{proof}
Let $S \subset \interior(M)$ be the surface along which the Torelli surgery $M \leadsto M_s$ yields $M'$ for some $s \in \Gamma_{d+1} \Torelli(S)$.
Let $\Z\! \cdot\! \mathcal{Y}$ be the abelian group freely generated by the set $\mathcal{Y}$,
to which $f$ extends by linearity. We denote by $\Z[\Torelli(S)]$ the group ring of $\Torelli(S)$ with augmentation ideal $I$. 
The map $\Torelli(S) \to  \mathcal{Y}$ defined by $r \mapsto M_r$
extends by linearity to a map $\zeta: \Z[\Torelli(S)] \to \Z\! \cdot\! \mathcal{Y}$. 
It follows from the definition of a finite-type invariant that
$$
f\circ \zeta\left(I^{d+1}\right) =0.
$$
By assumption, $s-1 \in \Z[\Torelli(S)]$ belongs to $I^{d+1}$, so that we  have $f(M_s-M)=0\in A$.
\end{proof}

According to \cite{Habiro,Habegger}, any homology cylinder over $\Sigma$ is $Y_1$-equivalent to the trivial cylinder, 
so that  the above applies to the class $\mathcal{Y}:=\cyl(\Sigma)$. 
Goussarov and Habiro have conjectured the converse of Lemma \ref{lem:Y_FTI} to be true 
for homology cylinders, and they proved this converse in genus $g=0$ \cite{Goussarov_graphs,Habiro}.
The conjecture is also known to be true 
in degree $d=2$ (as will follow from Theorem A too), and in some weaker forms \cite{Massuyeau_DSP}.

\section{Some classical invariants of homology cylinders}\label{sec:invariants}

In this section, we define the topological invariants of homology cylinders 
that are needed to characterize the $Y_k$-equivalence and the $J_k$-equivalence relations for $k=2$ and $k=3$.
We also describe some of their relationship, and their variation under surgery along a graph clasper. 

\subsection{Johnson homomorphisms} \label{subsec:Johnson}

The Johnson homomorphisms have been introduced and studied  
by Johnson \cite{Johnson_first_homomorphism,Johnson_survey} and Morita \cite{Morita_abelian} on the Torelli group, 
and by Garoufalidis--Levine \cite{GL_tree} and Habegger \cite{Habegger} on the monoid of homology cylinders.

First of all, we recall how the Johnson filtration is defined.
Using the same notation as in the introduction, we  set $\pi:=\pi_1(\Sigma,\star)$ for the fundamental group of $\Sigma$ 
(with base point $\star$ on the boundary), 
and we denote by 
$
\pi = \Gamma_1 \pi \supset \Gamma_2 \pi \supset \Gamma_3 \pi \supset \cdots
$
the lower central series of $\pi$.
Let $(M,m)$ be a homology cobordism of $\Sigma$. 
According to Stallings \cite{Stallings}, the map $m_{\pm}$ induces an isomorphism $m_{\pm,*}$
at the level of the $k$-th nilpotent quotient $\pi_1(\cdot)/\Gamma_{k+1} \pi_1(\cdot)$ 
of the fundamental group, 
so that the composition ${m_{-,*}}^{-1}\circ m_{+,*}$ defines an element of $\Aut\left(\pi/\Gamma_{k+1}\pi \right)$.  
(Here, the base point of $M$ is $m(\star,0)$ and is connected to $m_\pm(\star)=m(\star,\pm1)$ 
through the segment $m(\{\star\} \times I)\subset \partial M$.)
So for each $k\ge 0$ we get a monoid homomorphism 
$$ 
\rho_k: \cob \longrightarrow \Aut\left(\pi/\Gamma_{k+1}\pi \right), \ M \longmapsto {m_{-,*}}^{-1}\circ m_{+,*}
$$
which is the group homomorphism (\ref{eq:rho_k_mcg}) on the mapping class group.
Thus, $\rho_k$ should be thought of as the ``$k$-th nilpotent approximation'' of the  Dehn--Nielsen representation.  
The descending sequence of submonoids
$$ 
\cob =\cob[0] \supset \cob[1] \supset \cob[2]\supset \cdots,
$$
where $\cob[k]$ is the kernel of $\rho_k$, is called the \emph{Johnson filtration} of $\cob$. 
We are particularly interested in the monoids $\cob[1]=\cyl$ and $\cob[2]=\Jcob$.

The \emph{$k$-th Johnson homomorphism} $\tau_k$ is then defined in the following way.  
An element $f\in \Hom\left(H,\Gamma_{k+1}\pi/\Gamma_{k+2}\pi \right)$ defines an automorphism of $\pi/\Gamma_{k+2} \pi$ 
by sending any $\{x\} \in \pi/\Gamma_{k+2} \pi$ to $f(\{x\})\cdot \{x^{-1}\}$.
Thus we obtain an  exact sequence of groups
$$
1 \rightarrow \Hom\left(H,\Gamma_{k+1}\pi/\Gamma_{k+2}\pi\right) \rightarrow 
\Aut\left(\pi/\Gamma_{k+2}\pi\right) \rightarrow \Aut\left(\pi/\Gamma_{k+1}\pi\right)
$$
and the restriction of $\rho_{k+1}$ to the submonoid $\cob[k]$ yields a monoid homomorphism
$$
\tau_k: \cob[k] \longrightarrow \Hom\left(H,\Gamma_{k+1}\pi/\Gamma_{k+2}\pi\right)
\simeq H^* \otimes \Gamma_{k+1}\pi/\Gamma_{k+2}\pi \simeq H  \otimes \Gamma_{k+1}\pi/\Gamma_{k+2}\pi.
$$ 
Here, the group $H$ is identified with $H^*=\Hom(H,\Z)$ by the map $h \mapsto \omega(h,\cdot)$ 
using the intersection form of $\Sigma$
$$
\omega: H \times H \longrightarrow \Z.
$$
One usually restricts the target of the $k$-th Johnson homomorphism  in the following way.
We denote by $\Lie(H)$ the graded Lie ring freely generated by $H$ in degree $1$.
There is a canonical isomorphism between $\Lie(H)$ 
and the graded Lie ring $\Gr^\Gamma \pi= \bigoplus_{k\geq1} \Gamma_k\pi/\Gamma_{k+1}\pi$ 
associated to the lower central series of $\pi$ \cite{Bourbaki}. 
Therefore, $\tau_k$ can be seen with values in $H\otimes \Lie_{k+1}(H)$.
It turns out that $\tau_k$ takes values in the kernel of the Lie bracket map
$$
\DD_{k}(H) := 
\Ker\left([\cdot,\cdot]: H \otimes \Lie_{k+1}(H) \longrightarrow \Lie_{k+2}(H)\right),
$$
see \cite{Morita_abelian,GL_tree}.
Here are a few properties of the  $k$-th Johnson homomorphism $\tau_k$:
\begin{itemize}
\item $\tau_k: \cob[k]\rightarrow \DD_{k}(H)$ is surjective \cite{GL_tree,Habegger};
\item $\tau_k$ is $Y_{k+1}$-invariant 
(since the map $\rho_{k+1}$ is invariant under $Y_{k+1}$-equivalence as follows, for example, from Lemma \ref{lem:rho_k} below);
\item $\tau_k$ is invariant under stabilization in the sense that, if $\Sigma$ is stabilized to a surface $\Sigma^s$
as shown in Figure \ref{fig:stabilization} so that $H$ embeds into $H^s := H_1(\Sigma^s)$,
then the following diagram is commutative:
$$
\xymatrix{
\cob(\Sigma)[k]\ar[d]_-{\tau_k}  \ar@{->}[r] & \cob(\Sigma^s)[k] \ar[d]^-{\tau_k} \\
\DD_{k}(H) \ar@{->}[r] & \DD_{k}(H^s).
}
$$
\end{itemize}

We now specialize to the cases  $k=1$ and $k=2$ which will be enough for our purposes.
Then the free abelian group $\DD_{k}(H)$ has the following description.
For $k=1$, there is an isomorphism
\begin{equation}\label{eq:D_3}
\Lambda^3 H \stackrel{\simeq}{\longrightarrow} \DD_{1}(H)
\end{equation}
given by $a \wedge b \wedge c \mapsto a\otimes [b,c]+ c\otimes [a,b] + b \otimes [c,a]$,
see \cite{Johnson_first_homomorphism}.
For $k=2$, the description of $\DD_{k}(H)$ is more delicate and involves 
$\left(\Lambda^2 H \otimes \Lambda^2H\right)^{\mathfrak{S}_2}$, 
\ie the symmetric part of the second tensor power of $\Lambda^2 H$.
This free abelian group contains an isomorphic image of the degree $2$ part $S^2 \Lambda^2 H$
of the symmetric algebra over $\Lambda^2 H$, 
which we regard as a quotient of the tensor algebra over $\Lambda^2H$.
Indeed, the map $S^2\Lambda^2 H \to \left(\Lambda^2 H \otimes \Lambda^2H\right)^{\mathfrak{S}_2}$
sending $x\cdot y$ to  $(x \leftrightarrow y) := x \otimes y + y \otimes x$ is injective,
and we denote its image by $\Lambda^2 H \leftrightarrow \Lambda^2H$.
Note that  we have an isomorphism
$$
\frac{\Lambda^2 H}{2\cdot\Lambda^2 H } \stackrel{\simeq}{\longrightarrow}
\frac{\left(\Lambda^2 H \otimes \Lambda^2H\right)^{\mathfrak{S}_2}}{\Lambda^2 H \leftrightarrow \Lambda^2H}, \quad
\{a \wedge b\}\longmapsto \{ (a\wedge b) \otimes (a\wedge b) \},
$$
hence a short exact sequence of abelian groups:
\begin{equation}
\label{eq:sym_to_sym}
\xymatrix{
0 \ar[r] & S^2 \Lambda^2H \ar[r]^-{\leftrightarrow} & \left(\Lambda^2 H \otimes \Lambda^2H\right)^{\mathfrak{S}_2} \ar[r] 
& \frac{\Lambda^2 H}{2\cdot\Lambda^2 H } \ar[r] & 0
} 
\end{equation}
where the map on the right side sends tensors of the form $(a\wedge b) \otimes (a\wedge b)$ to $\{a \wedge b\}$.
Finally, note that $\Lambda^4H$ can be embedded in 
$\left(\Lambda^2 H \leftrightarrow \Lambda^2H\right) \subset \left(\Lambda^2 H \otimes \Lambda^2H\right)^{\mathfrak{S}_2}$ 
by sending any $4$-vector  $a\wedge b\wedge c\wedge d$ to the sum
$$
(a\wedge b)\leftrightarrow (c\wedge d) - (a\wedge c)\leftrightarrow (b\wedge d) + (a\wedge d)\leftrightarrow (b\wedge c).
$$

\begin{proposition}[Morita, Levine]\label{prop:Morita-Levine}
There is a unique isomorphism 
\begin{equation}\label{eq:D_2}
\frac{\left(\Lambda^2 H \otimes \Lambda^2H\right)^{\mathfrak{S}_2}}{\Lambda^4H} 
\stackrel{\simeq}{\longrightarrow} \DD_{2}(H)
\end{equation}
that is defined by
\begin{equation}\label{eq:D_2_formula}
\big((a\wedge b)\leftrightarrow (c\wedge d)\big)  \mapsto 
a \otimes [b,[c,d]] + b \otimes [[c,d],a]+ c \otimes [d,[a,b]] + d \otimes [[a,b],c].
\end{equation}
\end{proposition}

\begin{proof}[Sketch of the proof]
The map (\ref{eq:D_2}) is defined by Morita in \cite{Morita_Casson1,Morita_Casson2}, 
where the abelian group ${\left(\Lambda^2 H \otimes \Lambda^2H\right)^{\mathfrak{S}_2}}/{\Lambda^4H}$ is denoted by $\overline{T}$.  
There, Morita states that the map is injective,
and the bijectivity of  (\ref{eq:D_2}) is essentially proved by Levine in \cite{Levine_quasi-Lie}.
To show that the map (\ref{eq:D_2}) is uniquely defined, we consider the map
$$
\eta: (\Lambda^2 H \leftrightarrow \Lambda^2H) \longrightarrow \DD_{2}(H)
$$
defined by formula (\ref{eq:D_2_formula}). Taking rational coefficients, we get a homomorphism $\eta \otimes \Q$
from $\left(\Lambda^2 H \otimes \Lambda^2H\right)^{\mathfrak{S}_2} \otimes \Q$ 
$= (\Lambda^2 H \leftrightarrow \Lambda^2H)\otimes \Q$ 
to $\DD_{2}(H)\otimes \Q$. For all $a,b\in H$,
\begin{eqnarray*}
(\eta \otimes \Q)\big((a\wedge b)\otimes (a\wedge b)\big)
&= &\frac{1}{2} (\eta \otimes \Q)\big((a\wedge b)\leftrightarrow (a\wedge b)\big)\\
&=& \frac{1}{2} \left(2 \cdot a \otimes [b,[a,b]] + 2 \cdot b \otimes [[a,b],a]\right)
\end{eqnarray*}
belongs to $\DD_2(H) \subset \DD_2(H)\otimes \Q$.
Thus, the restriction of $\eta \otimes \Q$ to $\left(\Lambda^2 H \otimes \Lambda^2H\right)^{\mathfrak{S}_2}$
takes values in $\DD_2(H)$, and a simple computation shows that it vanishes
on the image of $\Lambda^4 H$. This discussion shows that the homomorphism (\ref{eq:D_2})
is well-defined and is determined by the formula (\ref{eq:D_2_formula}).

Following Levine \cite{Levine_addendum,Levine_quasi-Lie}, 
we consider the  \emph{quasi-Lie ring} $\Lie'(H)$ freely generated by $H$ and,
similarly to $\DD_k(H)$, we define
$$
\DD'_{k}(H) := 
\Ker\left([\cdot,\cdot]: H \otimes \Lie_{k+1}'(H) \longrightarrow \Lie_{k+2}'(H)\right).
$$
The natural group homomorphism $\Lie'(H) \to \Lie(H)$
induces a group homomorphism $\DD'(H) \to \DD(H)$ which, in degree $2$,
happens to be injective but not surjective \cite{Levine_addendum}:
\begin{equation}
\label{eq:D_to_D}
\xymatrix{
0 \ar[r] & \DD'_2(H) \ar[r] & \DD_2(H) \ar[r]^-L & (\Lambda^2 H)\otimes \Z_2 \ar[r] & 0.
} 
\end{equation}
Here the map $L$ is defined by an application of the ``snake lemma''. 
Levine also considers the map
$\eta': \frac{S^2 \Lambda^2H}{\Lambda^4H} \longrightarrow D'_2(H)$
defined by
$$
\left\{ (a\wedge b) \cdot (c\wedge d)\right\}  \stackrel{\eta'}{\longmapsto}
a \otimes [b,[c,d]] + b \otimes [[c,d],a]+ c \otimes [d,[a,b]] + d \otimes [[a,b],c].
$$
(This map $\eta'$ is actually the degree $2$ case of a more general construction,
which transforms tree Jacobi diagrams to elements of $\DD'(H)$: see \S \ref{sec:diagrams} in this connection.)
From the definition of Levine's map $L$, we see that the following diagram is commutative:
$$
\xymatrix{
0 \ar[r] & \frac{S^2 \Lambda^2H}{\Lambda^4H} \ar[r]^-{\leftrightarrow} \ar[d]^-{\eta'} & 
\frac{\left(\Lambda^2 H \otimes \Lambda^2H\right)^{\mathfrak{S}_2}}{\Lambda^4H} \ar[r] \ar[d]^-{\eta\otimes \Q}
& \frac{\Lambda^2 H}{2\cdot\Lambda^2 H }  \ar[r] \ar[d]^-\simeq & 0\\
0 \ar[r] & \DD'_2(H) \ar[r] & \DD_2(H) \ar[r]^-L & (\Lambda^2 H)\otimes \Z_2 \ar[r] & 0
} 
$$
The map $\eta'$ is bijective in degree $2$ \cite{Levine_quasi-Lie}.
We conclude that (\ref{eq:D_2}) is an isomorphism.
\end{proof}

In the sequel, the identifications (\ref{eq:D_3})  and (\ref{eq:D_2}) will be implicit.
A formula for the variation of $\tau_1$ under surgery along a $Y$-graph  is given in \cite{MM}.  
Strictly similar arguments give the following formula for $\tau_2$ and graph claspers of degree $2$. 

\begin{lemma}\label{lem:tau2}
Let $H$ be a degree $2$ graph clasper in a homology cylinder $M$
with $4$ leaves $f_1, \dots, f_4$ which are oriented as shown below:\\[-0.4cm]
$$
\labellist
\small\hair 2pt
 \pinlabel {$f_4$} [l] at 70 9
 \pinlabel {$f_3$} [l] at 69 43
 \pinlabel {$f_2$} [r] at 0 41
 \pinlabel {$f_1$} [r] at 0 9
\endlabellist
\includegraphics[scale=1.3]{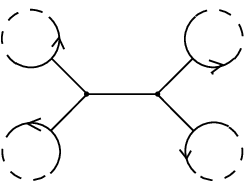}\\[-0.3cm]
$$
Then we have
$$
\tau_2\left( M_H\right)-\tau_2(M) = \left\{ (h_1\wedge h_2) \leftrightarrow (h_3 \wedge h_4) \right\}
$$
where $h_1,\dots,h_4\in H$ denote the homology classes of $f_1,\dots,f_4$ respectively.   
\end{lemma}

\noindent 
By clasper calculus, 
we can always assume that a degree $2$ graph clasper  has four leaves.  
In particular, this lemma implies that surgery along a looped graph clasper $L$ of degree $2$ does not modify $\tau_2$.
Combining this with Lemma \ref{lem:doubling} also gives  the following.

\begin{lemma}\label{lem:tau2_2}
Let $Y$ be a $Y$-graph in a homology cylinder $M$ with one special leaf and two oriented leaves $f,f'$ as shown below:\\[-0.4cm]
$$
\labellist
\small\hair 2pt
 \pinlabel {$f$} [r] at 0 71
 \pinlabel {$f'$} [l] at 78 72
\endlabellist
\includegraphics[scale=0.8]{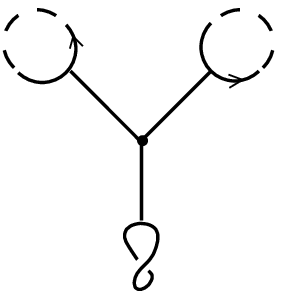}\\[-0.3cm]
$$
Then we have
$$
\tau_2\left(M_Y\right) -\tau_2(M) = \left\{(h \wedge h') \otimes (h\wedge h')   \right\}
$$
where $h,h'\in H$ denote the homology classes of $f,f'$ respectively.
\end{lemma}

\subsection{The Alexander polynomial and the Reidemeister--Turaev torsion}\label{subsec:Alexander}

There is a relative version of the Alexander polynomial for homology cylinders \cite{Sakasai}.
The \emph{relative Alexander polynomial} of an $M\in \cyl$ is
the order of the relative homology group $H_1(M, \partial_-M;  \Z[H])$
whose coefficients are twisted by the ring homomorphism
$$
\xymatrix{
\Z[\pi_1(M)] \ar@{->>}[r] & \Z[H_1(M)] \ar[rr]^-{(m_{\pm,*})^{-1}}_-\simeq & &  \Z[H].
}
$$
This order is defined up to multiplication by a unit of the ring $\Z[H]$, \ie an element of the form $\pm h$ for some $h\in H$.
We denote it by
$$
\Delta(M,\partial_-M) := \ord H_1(M, \partial_-M;  \Z[H]) \ \in \Z[H]/\pm H.
$$

\begin{lemma}
For all $M \in \cyl$, we have $\Delta(M,\partial_-M) \neq 0$.
\end{lemma}

\begin{proof}
We have the following general fact, essentially proved in \cite[Proposition 2.1]{KLW}.\\[-0.3cm]

\begin{quote}
\textbf{Fact.}
\emph{Let $(X,Y)$ be a connected
CW pair such that the inclusion $Y\subset X$  induces an isomorphism in homology. 
Then, for all injective ring homomorphism $\varphi: \Z[H_1(X)] \to \mathbb{F}$ with values in a commutative field $\mathbb{F}$, 
the homology group $H_*(X,Y;\mathbb{F})$ with coefficients twisted by $\varphi$ is trivial.}\\[-0.3cm]
\end{quote}

\noindent
Thus, an application of the universal coefficient theorem gives
$$
H_1(M, \partial_-M;  \Z[H]) \otimes_{\Z[H]} Q(\Z[H])
\simeq H_1\big(M, \partial_-M;  Q(\Z[H])\big) =0
$$
where $Q(\Z[H])$ denotes the fraction field of the domain $\Z[H]$.
We deduce that the $\Z[H]$-module $H_1(M, \partial_-M;  \Z[H])$ is torsion.
\end{proof}

\noindent

It has been shown by Milnor for link complements \cite{Milnor_duality} and by Turaev for closed manifolds \cite{Turaev_Alexander}
that the Alexander polynomial can be interpreted in dimension $3$ as a kind of Reidemeister torsion.
The reader is referred to Turaev's book \cite{Turaev_book} for an introduction to the theory of Reidemeister torsions.
The same interpretation holds for the relative Alexander polynomial of homology cylinders.

\begin{proposition}\label{prop:Alexander_torsion}
Let $M\in \cyl$ and let $(X,Y)$ be a cell decomposition of $(M,\partial_-M)$. 
We denote by $\mu: \Z[\pi_1(X)] \to Q(\Z[H])$ the ring map induced by the isomorphism
$(m_{\pm,*})^{-1}: H_1(X)\simeq H_1(M) \to H$ and we denote by $\tau^\mu(X,Y)\in Q(\Z[H])/\pm H$
the relative Reidemeister torsion with abelian coefficients given by $\mu$.
Then we have
$$
\tau^\mu (X,Y) = \Delta(M,\partial_-M) \ \in \Z[H]/\pm H.
$$
\end{proposition}

\noindent
This follows from \cite[Lemma 3.6]{FJR} for instance.
The main argument is that $M$ collapses relatively to $\partial_- M$ onto a cell complex having only $1$-cells and $2$-cells in an equal number.
(This fact follows from the existence of a Heegaard splitting for $M$, as discussed in \S \ref{subsec:definition_relations}.)
Thus, the computations of $\tau^\mu (X,Y)$ and $\Delta(M,\partial_-M)$ reduces to a single determinant.

Thanks to Proposition \ref{prop:Alexander_torsion}, one can use Turaev's refinement of the Reidemeister torsion \cite{Turaev_Euler}
to fix the ambiguity in $\pm H$ in the definition of the relative Alexander polynomial.
Let $M\in \cyl$ and let $(X,Y)$ be a cell decomposition of $(M,\partial_-M)$.
An \emph{Euler chain}  in $X$ relative to $Y$ is a singular $1$-chain $c$ in $X$ with boundary
$$
\partial c = \sum_{\sigma} (-1)^{\dim(\sigma)} \cdot (\hbox{center of } \sigma)
$$
where the sum is indexed by the cells $\sigma$ of $X \setminus Y$.
Such chains exist since the relative Euler characteristic of the pair $(X,Y)$ is zero.
Two Euler chains $c$ and $c'$ are \emph{homologous}
if the $1$-cycle $c-c'$ is null-homologous.
An \emph{Euler structure} on $X$ \emph{relative} to $Y$ is a homology class of Euler chains. The set
$$
\cEul(X,Y)
$$
of Euler structures on $X$ relative to $Y$ is an $H_1(X)$-affine space.
Turaev associates in \cite{Turaev_Euler} to each $\theta \in \cEul(X,Y)$ a representative 
$$
\tau^\mu(X,Y;\theta) \in Q(\Z[H])
$$
of the relative Reidemeister torsion $\tau^\mu(X,Y)$ in such a way that
\begin{equation}
\label{eq:torsion_affine}
\forall h \in H \simeq H_1(X), \quad \tau^\mu\left(X,Y;\theta + \overrightarrow{h}\right) = h \cdot \tau^\mu(X,Y;\theta).
\end{equation}
We call $\tau^\mu(X,Y;\theta)$ the \emph{Reidemeister--Turaev torsion} (or, in short, \emph{RT torsion}) 
of the CW pair $(X,Y)$ equipped with $\theta$.
The ambiguity in $H$ is fixed by lifting an Euler chain in the class $\theta$ to the maximal abelian cover of $X$,
which gives a preferred lift for each cell of $X\setminus Y$. 
The sign ambiguity is fixed thanks to a correcting multiplicative term: in general, one has to choose
an orientation of the $\R$-vector space $H_*(X,Y;\R)$ but, in the situation of homology cylinders, this space is trivial.
Observe also that, in this situation,  $\tau^\mu(X,Y;\theta)$ belongs to $\Z[H]$ by Proposition \ref{prop:Alexander_torsion}.

The Euler structures that are defined in the previous paragraph are called \emph{combinatorial}
since they are defined for pairs of CW complexes $(X,Y)$. 
There is also a \emph{geometric} version of Euler structures 
which are defined  in \cite{Turaev_Euler} for pairs of smooth manifolds $(U,V)$:
the submanifold $V$ is then assumed to be a union of connected components of $\partial U$.
Turaev's correspondence between the two notions of Euler structures  involves  smooth triangulations.
If the manifold $U$ is three-dimensional, Benedetti and Petronio define in \cite{BP} a relative version
of the Reidemeister--Turaev torsion for quite general submanifolds $V$ of $\partial U$.
This invariant was rediscovered by Friedl, Juh\'asz and Rasmussen in the context of sutured Heegaard--Floer homology \cite{FJR}.
The correspondence between combinatorial Euler structures and geometric Euler structures is proved
in \cite{BP} using the theory of branched standard spines and in \cite{FJR} using Morse theory.
In the sequel, we adopt the latter viewpoint which is better suited to our purposes.

The constructions of \cite{FJR} apply to any homology cylinder $M$ over $\Sigma$ in the following way.
In order to consider homology cylinders in the smooth category, we smooth the corners of $(\Sigma \times I)$
and we denote the smooth trivial cylinder by $(\Sigma \times I)^{\operatorname{sc}}$:
$$
\labellist
\small\hair 2pt
 \pinlabel {$\Sigma \times I$}  at  77 31
  \pinlabel {$\leadsto$}  at  212 31
 \pinlabel {$\left(\Sigma \times I\right)^{\operatorname{sc}}$}  at  335 31
\endlabellist
\includegraphics[scale=0.5]{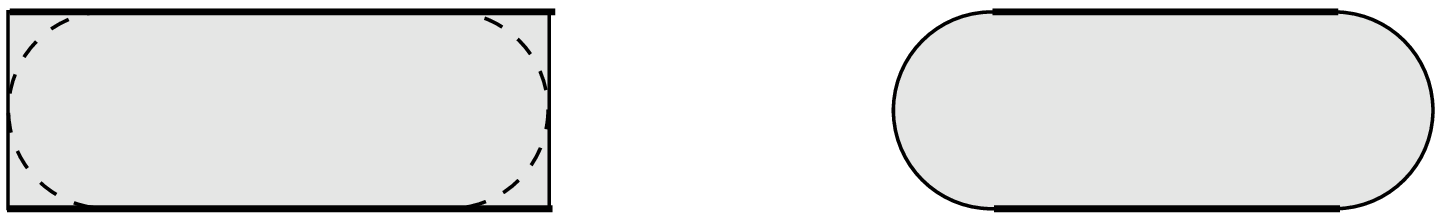}
$$
The inclusion $(\Sigma \times I)^{\operatorname{sc}} \subset (\Sigma \times I)$
identifies the tangent bundle $T(\Sigma \times I)^{\operatorname{sc}}$
with the restriction of $T \Sigma \times TI$ to $(\Sigma \times I)^{\operatorname{sc}}$.
We give the $3$-manifold  $M$ a smooth structure 
and we assume that its boundary is parametrized by a diffeomorphism
$m:\partial\left(\Sigma \times I\right)^{\operatorname{sc}} \to \partial M$.
This diffeomorphism induces  an isomorphism of  vector bundles
\begin{equation} \label{eq:m*}
\R\oplus m_*: T\left(\Sigma \times I\right)^{\operatorname{sc}}|_{\partial \left(\Sigma \times I\right)^{\operatorname{sc}}}
\simeq \R \oplus T\partial \left(\Sigma \times I\right)^{\operatorname{sc}}
\longrightarrow \R \oplus T\partial M \simeq TM|_{\partial M}
\end{equation}
between the tangent bundles restricted to the boundaries.
We denote by $t$ the coordinate along $I$ in $(\Sigma\times I)$ 
and by $\frac{\partial}{\partial t}$ the corresponding vector field,
which is a non-singular section of $T \Sigma \times TI$.
The image by (\ref{eq:m*}) of its restriction to $(\Sigma \times I)^{\operatorname{sc}}$ 
defines on $\partial M$ a non-singular vector field $v_0$ of $M$ 
which points outside $M$ on $\partial_+M$, points inside $M$ on $\partial_-M$ 
and is tangent to $\partial M$ along the circle $m(\partial \Sigma \times \{0\})$:
see Figure \ref{fig:v_0}.
An \emph{Euler structure} on $M$ \emph{relative} to $\partial_-M$ 
is an equivalence class of non-singular vectors fields $v$ on $M$ such that $v|_{\partial M}=v_0$.
Here two such vector fields $v,v'$ are considered \emph{equivalent} if there is an open $3$-ball $B\subset \interior(M)$
such that $v|_{M\setminus B}$ and $v'|_{M\setminus B}$ are homotopic relatively to $\partial M$.
Obstruction theory tells us that the set of Euler structures on $M$ relative to $\partial_- M$
$$
\gEul(M,\partial_-M)
$$
is an $H_1(M)$-affine space since we have $H_1(M)\simeq H^2(M,\partial M)$ by Poincar\'e duality.
Let $(X,Y)$ be a cell decomposition of $(M,\partial_- M)$
arising from a handle decomposition of $M$ relative to $\partial_- M$.
Then, an $H_1(M)$-equivariant correspondence between combinatorial and geometric Euler structures
\begin{equation}
\label{eq:comb_to_geom}
\xymatrix{
\cEul(X,Y) \ar[r]^-\simeq & \gEul(M,\partial_-M)
}
\end{equation}
is defined in \cite[\S 3]{FJR} by desingularizing a gradient-like vector field of a Morse function that induces the given handle decomposition.
This bijection is similar to the formulation in the closed case of Turaev's correspondence  \cite{Turaev_Euler} 
in terms of Morse theory \cite{HL,Massuyeau_torsion}.
Therefore, the \emph{relative RT torsion} of $M$ equipped with $\xi\in \gEul(M,\partial_-M)$ is defined as
$$
\tau(M,\partial_-M;\xi) := \tau^\mu(X,Y;\theta) \in Q(\Z[H])
$$
where $\theta\in \cEul(X,Y)$ corresponds to $\xi$ by the correspondence (\ref{eq:comb_to_geom}).
\begin{figure}[h]
\begin{center}
\labellist
\small\hair 2pt
 \pinlabel {$M$}  at  109 35
 \pinlabel {$\partial_+M$} [r] at 0 59
 \pinlabel {$\partial_-M$} [r] at 0 2
 \pinlabel {$m\left((\partial \Sigma \times I)^{\operatorname{sc}}\right)$} [l] at 208 30
 \pinlabel {$v_0$} [bl] at 168 75
\endlabellist
\includegraphics[scale=0.8]{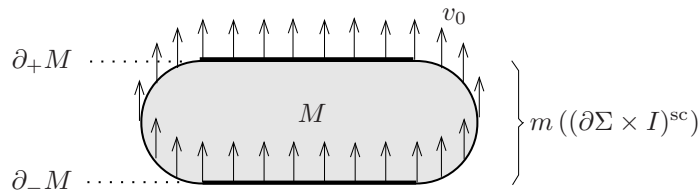}
\end{center}
\caption{The non-singular vector field $v_0$ of $M$ defined on $\partial M$.}
\label{fig:v_0}
\end{figure}

We shall need further constructions for relative Euler structures,
which are already available in the literature for closed oriented $3$-manifolds.
Thus an efficient way to extend them to the case of homology cylinders is to ``close'' any $M \in \cyl$ as follows.
First of all, we  add a $2$-handle $D$ to the surface $\Sigma$  to obtain a closed connected oriented surface $\clocase{\Sigma}$ of genus $g$.
Next we  glue a $2$-handle $D \times I$ along $M$ to obtain a cobordism $\clocase{M}$ between two copies of $\clocase{\Sigma}$.
Finally we obtain a closed connected oriented $3$-manifold $M^\sharp$
by identifying the bottom surface of $\clocase{M}$ with its top surface using the homeomorphism
$$ 
\xymatrix{
\partial_-  \clocase{M} & &  \ar[ll]_-{m_-\cup \Id_D}^-\cong  \clocase{\Sigma} \ar[rr]^-{m_+\cup \Id_D}_-\cong & & \partial_+  \clocase{M}.
} 
$$
Observe that we have $M  \subset M^\sharp$ and there is an isomorphism 
$$
H_1(M) \oplus \Z \overset{\simeq}{\longrightarrow} H_1(M^\sharp), \ 
(h,x)\longmapsto \incl_*(h) + x\cdot \left[0\times S^1\right], 
$$
where the circle $0\times S^1$ denotes the co-core $0\times I$ of the $2$-handle $D\times I$
with the two ends identified.

A non-singular vector field $v$ on $M$ which coincides with $v_0$ on $\partial M$
can be extended to a non-singular vector field $v^\sharp$ on $M^\sharp$ 
by gluing it to the vector field $\frac{\partial}{\partial t}$ on $D\times I$.
Let $\gEul(M^\sharp)$ be the space of  geometric Euler structures on the closed $3$-manifold $M^\sharp$,
\ie the space of non-singular vector fields on $M^\sharp$ up to homotopy on $M^\sharp$ deprived of an open $3$-ball \cite{Turaev_Euler}.
Then a \emph{closure} map for  Euler structures  
\begin{equation}
\label{eq:Euler_closure}
\xymatrix{
\gEul(M,\partial_-M) \ar[r] & \gEul(M^\sharp), \quad \xi \ar@{|->}[r] & \xi^\sharp
}
\end{equation}
is defined by  $\xi^\sharp:= [v^\sharp]$ if $\xi=[v]$. This map is affine over $\incl_*:H_1(M) \to H_1(M^\sharp)$.

Recall that a ``Chern class'' map $c:\gEul(M^\sharp) \to H_1(M^\sharp)$ 
is defined in \cite{Turaev_Euler} for the closed oriented $3$-manifold $M^\sharp$.
The Chern class $c([u]) \in H^2(M^\sharp) \simeq H_1(M^\sharp)$ of a $[u]\in \gEul(M^\sharp)$ 
is the obstruction to find a non-singular vector field on $M^\sharp$ linearly independent with $u$.
In other words, it is defined as the Euler class (or, equivalently, first Chern class) 
of a $2$-dimensional oriented vector bundle:
$$
c([u]) =e\left(TM^\sharp/\langle u\rangle\right) \ \in H^2(M^\sharp).
$$
We define the \emph{relative Chern class} map for the homology cylinder $M$ by the  diagram
$$
\xymatrix{
    \gEul(M,\partial_- M) \ar@{.>}[d]_c \ar[r]^-{(\ref{eq:Euler_closure})} & \gEul(M^\sharp) \ar[d]^{c}\\
     H_1(M) & \ar@{->>}[l]^-p  H_1(M)\oplus \Z \simeq H_1(M^\sharp) \quad \quad \quad \quad \quad
    }
$$
where the map $p$ denotes the natural projection.
The relative Chern class $c(\xi)\in H_1(M) \simeq H^2(M,\partial M)$ of a $\xi \in \gEul(M,\partial_- M)$ 
can be described without making reference to $M^\sharp$ as follows.
Let $w$ be a non-singular vector field on the surface $\Sigma$,
and let $w'$ be the image by the isomorphism (\ref{eq:m*})
of the vector field $w\times I$ on $(\Sigma\times I)^{\operatorname{sc}} \subset \Sigma \times I$:
thus, $w'$ is a non-singular vector field of $M$ defined on $\partial M$ and linearly independent with $v_0$.
Then, for any non-singular vector field $v$ on $M$ representing $\xi$,
$c(\xi)$ is the obstruction to extend  $w'$ to a non-singular vector field on $M$ linearly independent with $v$:
\begin{equation}
\label{eq:Chern_as_obstruction}
c(\xi) =e(TM/\langle v\rangle, w') \ \in H^2(M,\partial M).
\end{equation}

\begin{example} \label{ex:Chillingworth}
For the mapping cylinder $\mcyl(h)$ of an element $h$ of the Torelli group $\Torelli$,
 the relative Chern class map also has the following description.
Recall that the \emph{Chillingworth homomorphism} \cite{Chillingworth,Johnson_abelianization}
$$
t: \Torelli\longrightarrow H^1(\Sigma)
$$
maps $h\in \Torelli$ to the obstruction $t(h)$ to find a homotopy 
between a non-singular vector field $w$ on $\Sigma$ and its image under $h^{-1}$.  
Under the isomorphisms $H^1(\Sigma)\simeq H_1(\Sigma,\partial \Sigma)\simeq H_1(\Sigma)$, 
the Chillingworth homomorphism takes values in $H$. 
Let $\left[\frac{\partial}{\partial t}\right]\in \gEul(\mcyl(h),\partial_-\mcyl(h))$ 
be the  Euler structure represented by the ``upward'' vector field.  
Then we deduce from (\ref{eq:Chern_as_obstruction}) that
$c\left(\left[\frac{\partial}{\partial t}\right]\right)$ is equal to $- t(h)\in H_1(\Sigma\times I)\simeq H.$
\end{example}

We now give some basic properties of the relative Chern class map.

\begin{lemma}\label{lem:c}
Let $M\in \cyl$.
The relative Chern class map $c: \gEul(M,\partial_- M) \to H_1(M)$ is affine 
over the multiplication by $2$ map $H_1(M) \to H_1(M)$,
and its image is $2H_1(M)$.
\end{lemma}

\begin{proof}
In the closed case, the Chern class map $\gEul(M^\sharp) \to H_1(M^\sharp)$ is known to be affine over the multiplication by $2$ \cite{Turaev_Euler}.
Moreover, the closure map (\ref{eq:Euler_closure}) is affine over $\incl_*: H_1(M) \to H_1(M^\sharp)$.
We deduce that, for any $\xi \in \gEul(M,\partial_-M)$ and $h\in H_1(M)$,
$$
c\left(\xi + \overrightarrow{h}\right) = pc\left(\left(\xi + \overrightarrow{h}\right)^\sharp\right) =
pc(\xi^\sharp) + \overrightarrow{2 p \incl_*(h)}
= c(\xi) + 2 \overrightarrow{h}
$$
which proves the first statement.

To prove the second statement, it is enough to show that for any $\xi \in \gEul(M,\partial_-M)$ the mod $2$ reduction of $c(\xi)$ is trivial.
Let $w$  be a non-singular vector field on $\Sigma$,  
and let $w'$ be the image of $w\times I$ by (\ref{eq:m*}).
We deduce from (\ref{eq:Chern_as_obstruction}) that the mod $2$ reduction of $c(\xi)$ is the (primary) obstruction
to extend the parallelization $(v_0,w',v_0\wedge w')$ of $M$ on $\partial M$ to the whole of $M$.
Since $TM|_{\partial M}$ is isomorphic to $\R \oplus T\partial M$ (using the normal vector field),
$(v_0,w',v_0\wedge w')$ defines a spin structure $\sigma$ on $\partial M$,
and the latter obstruction is a relative second Stiefel--Whitney class:
$$
c(\xi) \! \! \! \mod 2 = w_2(M,\sigma) \in H^2(M,\partial M;\Z_2).  
$$
Let $(\rho_1,\dots,\rho_{2g})$ be a system of simple oriented closed curves on the surface $\Sigma$ which generates $H$.
Since $M$ is a homology cylinder over $\Sigma$, we can find for every $i=1,\dots,2g$
a compact connected oriented surface $R_i$ properly embedded in $M$ 
such that  $\partial R_i= m_+(\rho_i)- m_-(\rho_i)$. Then the quantity
$$
\left \langle w_2(M,\sigma), [R_i] \right\rangle \in \Z_2
$$
is the obstruction to extend the spin structure $\sigma|_{\partial R_i}$ to $R_i$ 
and, so, it vanishes since the latter restricts to the same spin structure of $\rho_i$ on each connected component of $\partial R_i$.
(Here we are using the fact that the $1$-dimensional spin cobordism group is  $\Z_2$.)
Since $[R_1],\dots, [R_{2g}]$ generate $H_2(M,\partial M)$, we conclude that  $w_2(M,\sigma)$ is trivial.
\end{proof}

The following is justified by  Lemma \ref{lem:c}.
\begin{definition}
Let $M\in \cyl$.
The \emph{preferred} relative Euler structure of $M$ 
is the unique $\xi_0\in \gEul(M,\partial_-M)$ satisfying $c(\xi_0)=0$. 
\end{definition}
\noindent
Thus the polynomial 
$$
 \tau(M,\partial_-M;\xi_0) \in \Z[H] \subset Q(\Z[H])
$$
is a topological invariant of homology cylinders which, by Proposition \ref{prop:Alexander_torsion},
can be regarded as a \emph{normalized} version of $\Delta(M,\partial_-M)\in \Z[H]/\!\pm H$.

\begin{example}
If $M$ is the mapping cylinder of an  $h \in \Torelli$,
we deduce from Lemma \ref{lem:c} and Example \ref{ex:Chillingworth} that
$\xi_0$ is $\left[\frac{\partial}{\partial t}\right] + \overrightarrow{t(h)/2}$. 
So formula (\ref{eq:torsion_affine}) gives
\begin{equation}
\label{eq:torsion_mapping_cylinder}
\tau\big(\mcyl(h),\partial_-\mcyl(h);\xi_0\big)= t(h)^{\nicefrac{1}{2}} \ \in H \subset \Z[H].
\end{equation}
(Recall that $H$ inside $\Z[H]$ is denoted multiplicatively.) 
This example shows that
$\tau(M,\partial_-M;\xi_0)$ \emph{does} depend on the boundary parametrization $m$ of $M$
although the class $\Delta(M,\partial_-M)$ only depends on $m$ through the isomorphism $m_{\pm,*}:   H \to H_1(M)$.
\end{example}

We shall now prove several  properties for the relative RT torsion of homology cylinders.
A first property is its invariance under stabilization. 
Indeed, if $\Sigma^s$ is a stabilization of the surface $\Sigma$ as shown in Figure \ref{fig:stabilization}
so that $H$ embeds into $H^s := H_1(\Sigma^s)$, then the following diagram is commutative:
\begin{equation}
\label{eq:stabilization_RT}
\xymatrix{
\cyl(\Sigma) \ar[r] \ar[d]_-{\tau(\, \cdot\, ,\partial_-\, \cdot\, ;\xi_0)}
& \cyl(\Sigma^s) \ar[d]^-{\tau(\, \cdot\, ,\partial_-\, \cdot\, ;\xi_0)}\\
\Z[H] \ar[r] & \Z[H^s].
}
\end{equation}
Next, the relative RT torsion is a limit of infinitely many finite-type invariants.
To show this, we need beforehand to understand how Torelli surgeries ``transport'' Euler structures.

\begin{lemma}\label{lem:Euler_surgery}
Let $M\in \cyl$, let $S \subset \interior(M)$ be a compact connected oriented surface with one boundary component and let $s\in \Torelli(S)$.
Then the Torelli surgery $M \leadsto M_s$ induces a  canonical bijection $\Omega_s: \gEul(M,\partial_- M)\rightarrow \gEul(M_s,\partial_- M_s)$, 
which is affine over the isomorphism $\Phi_s$ defined at (\ref{eq:iso_homology}) and which fits into the following commutative diagram:
$$ 
\xymatrix{
\gEul(M,\partial_-M) \ar[d]_c \ar[r]^{\Omega_s}_-\simeq & \gEul(M_s,\partial_- M_s) \ar[d]^{c}\\
H_1(M) \ar[r]^{\simeq}_-{\Phi_s} & H_1(M_s).
}
$$
\end{lemma}

\begin{proof}
We know from \cite{Massuyeau_torsion}, which deals with the closed case,
 that the Torelli surgery $M^\sharp \leadsto {M_s}^\sharp$ induces a canonical bijection 
$$
\Omega_s: \gEul(M^\sharp) \stackrel{\simeq}{\longrightarrow} \gEul\left({M_s}^\sharp\right).
$$
(The notion of ``Torelli surgery'' is defined  in \cite{Massuyeau_torsion} along the boundary 
of an embedded handlebody instead of an embedded surface $S$ with one boundary component. 
But the two notions are equivalent since a regular neighborhood of $S$ is a handlebody.)
The Euler structure $\Omega_s(\rho)$ associated to a  $\rho \in \gEul(M^\sharp)$ can be described as follows.
Let $u\in \rho$ be a non-singular vector field on $M^\sharp$ which is normal to $S$ and is positive with respect to $S$.
Then there is a regular neigborhood $S \times [-1,1]$ of $S$ in $\interior(M) \subset M^\sharp$ 
such that $u$ is  the ``upward'' vector field $\frac{\partial}{\partial t}$ on it.
Provided the smooth structure on ${M_s}^\sharp$ is appropriately chosen with respect to that of $M^\sharp$,
there is a unique vector field $u_s$ on ${M_s}^\sharp$ which coincides with $u$ on $M^\sharp\setminus \interior(S \times [-1,1])$ 
and with $\frac{\partial}{\partial t}$ on $\mcyl(s)$. Then, \cite[Lemma 3.5]{Massuyeau_torsion} tells us that
\begin{equation}
\label{eq:concrete_description}
\Omega_s(\rho)  = [u_s] -  \overrightarrow{\incl_*(t(s)/2)}
\end{equation}
where $t: \Torelli(S) \to H^1(S) \simeq H_1(S,\partial S)\simeq H_1(S)$ is the Chillingworth homomorphism.
In particular, $\Omega_s(\rho)$ is obtained by modifying the vector field $u_s$ 
by some ``Reeb turbulentization'' \cite{Turaev_Euler} which is supported in a neighborhood of $S$.
This fact implies that, for all $\xi \in \gEul(M)$, $\Omega_s(\xi^\sharp)$ belongs to the image of $\gEul(M_s)$. 
So we can define the map $\Omega_s$ in the case of homology cylinders by the following commutative diagram:
$$
\xymatrix{
\gEul(M,\partial_-M) \ar@{>->}[r]\ar@{-->}[d]_{\exists! \Omega_s}\ & \gEul(M^\sharp) \ar[d]_-\simeq^-{\Omega_s}\\
\gEul(M_s,\partial_-M_s) \ar@{>->}[r] & \gEul({M_s}^\sharp)
}
$$
In the closed case, the map  $\Omega_s$ is affine over the Mayer--Vietoris isomorphism $\Phi_s$
and it is compatible with the Chern class map \cite{Massuyeau_torsion}.
We easily deduce from the definitions that the map $\Omega_s$ has the same properties in the case of homology cylinders.
\end{proof}

In the sequel we denote by $I$ the augmentation ideal of $\Z[H]$.
 
\begin{theorem}\label{thm:torsion_finiteness}
Let $d\ge 1$. Then the $d$-th $I$-adic reduction of the relative Reidemeister--Turaev torsion of homology cylinders
$$
\tau(M,\partial_-M;\xi_0) \in \Z[H]/I^d
$$
is a finite-type invariant of degree at most $d-1$. 
\end{theorem}

\begin{proof}[Sketch of the proof]
The analogous statement for closed oriented $3$-manifolds $N$ has been proved in \cite{Massuyeau_torsion} using Heegaard splittings. 
To understand how the RT torsion changes $\tau(N,\rho) \leadsto \tau(N_s,\Omega_s(\rho))$
when a Torelli surgery $N \leadsto N_s$ is performed, one needs two technical ingredients: 
\begin{itemize}
\item[(i)]  a description following  \cite{HL} of Turaev's correspondence between combinatorial 
and geometric Euler structures in terms of Morse theory \cite[\S 2.3]{Massuyeau_torsion},
\item[(ii)] an explicit formula following  \cite{Turaev_spinc} which computes the RT torsion of $N$ 
from a Heegaard splitting by means of Fox's free derivatives \cite[\S 4.1]{Massuyeau_torsion}.
\end{itemize}
This proof can be adapted  in a straightforward way to the case of homology cylinders 
using the notion of Heegaard splitting defined in \S \ref{subsec:definition_relations} 
and some technical results from \cite{FJR}.
To be more specific,
the analogue of (i) can be found in \cite[\S 3.5]{FJR} and the analogue of (ii) is done in \cite[\S 4]{FJR}.
The final observation is that a Torelli surgery $M \leadsto M_s$ between homology cylinders 
preserves the preferred Euler structure $\xi_0$, as follows from Lemma \ref{lem:Euler_surgery}.
\end{proof}

We shall now identify the ``leading term'' of the relative RT torsion with respect to the $I$-adic filtration of $\Z[H]$.
According to Johnson \cite{Johnson_abelianization}, 
the Chillingworth homomorphism can be recovered from the first Johnson homomorphism by the formula
$t(h)= \operatorname{cont} \circ\, \tau_1 (h)$, for all $h\in \Torelli$,
where ``$\operatorname{cont}$'' is the contraction homomorphism $\Lambda^3H \to H$ defined by 
$$
\operatorname{cont} (a\wedge b\wedge c):=2\cdot \left(\omega(a,b)c+\omega(b,c)a+\omega(c,a)b\right).
$$
Thus  the map $t$  can be extended to a monoid homomorphism
$$
t: \cyl \longrightarrow H
$$ 
simply by setting $t:= \operatorname{cont} \circ\, \tau_1$.

\begin{lemma} \label{lem:I2}
For all $M\in \cyl$, we have $\tau(M,\partial_-M;\xi_0) = t(M)^{\nicefrac{1}{2}} \mod I^2$.  
\end{lemma}

\begin{proof}
The relative RT torsion is invariant by stabilization in the sense of (\ref{eq:stabilization_RT}).
Therefore we can assume that $g\geq 3$, 
and $M$ is then $Y_2$-equivalent to a mapping cylinder $\mcyl(h)$ for some $h\in \Torelli$ \cite{MM}.
By Theorem \ref{thm:torsion_finiteness} and Lemma \ref{lem:Y_FTI},
we have
$$
\tau(M,\partial_-M;\xi_0) = \tau(\mcyl(h),\partial_-\mcyl(h);\xi_0) \mod I^2.
$$
We conclude thanks to (\ref{eq:torsion_mapping_cylinder})
and the fact that $\tau_1$ is invariant under $Y_2$-equivalence.
\end{proof}

By Lemma \ref{lem:I2},  we can associate to any $M \in \cyl$ the symmetric tensor
\begin{equation}\label{def:a}
  \alpha(M):= \left\{ \tau(M,\partial_-M;\xi_0) - t(M)^{\nicefrac{1}{2}}\right\}  \ \in \frac{I^2}{I^3}\simeq S^2H.
\end{equation}
(Here $S^2H$ is identified with $I^2/I^3$ in the usual way, namely $h\cdot h' \mapsto \left\{(h-1)(h'-1)\right\}$.)
We shall refer to $\alpha(M)$ as the \emph{quadratic part} of the relative RT torsion. It has the following properties.

\begin{proposition}\label{prop:alpha_properties}
The map $\alpha: \cyl \to S^2H$ is an additive finite-type invariant of degree $2$, which satisfies 
\begin{equation}\label{eq:alexcyl}
\forall f \in \Torelli, \quad \alpha(\mcyl(f))=0.
\end{equation}
Next, if $G$ is a looped graph clasper of degree $2$ in a homology cylinder $M$
whose leaves $f$ and $f'$ are oriented as follows
$$
\labellist
\small\hair 2pt
 \pinlabel {$f$} [r] at 0 16
 \pinlabel {$f'$} [l] at 155 16
\endlabellist
\includegraphics[scale=0.8]{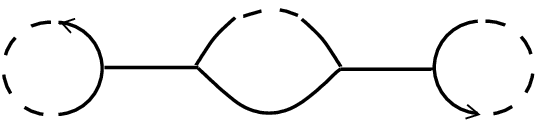}
$$
and have homology classes $h\in H$ and $h'\in H$ respectively,
then we have
$$
\alpha(M_G) - \alpha(M) = -2 \cdot h h'.
$$
Finally, if $Y$ is a $Y$-graph in a homology cylinder $M$ with two special leaves 
and one arbitrary leaf $f$  which is oriented in an arbitrary way
$$
\labellist
\small\hair 2pt
 \pinlabel {$f$} [r] at 22 20
\endlabellist
\includegraphics[scale=0.8]{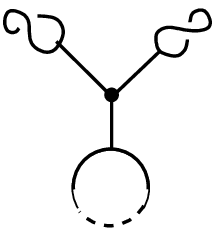}\\[-0.3cm]
$$
and has homology class $h\in H$, then we have
$$
\alpha(M_Y)-\alpha(M) =  h^2.
$$
\end{proposition}

In order to prove this proposition,
we shall need two other properties of the relative RT torsion of homology cylinders.

\begin{lemma}\label{lem:multiplicativity}
For all $M,M' \in \cyl$, we have
$$
\tau\big(M\circ M', \partial_- (M\circ M'); \xi_0\big) 
=  \tau(M,\partial_-M;\xi_0) \cdot \tau(M',\partial_-M';\xi_0) \in \Z[H].
$$
\end{lemma}

\begin{lemma}\label{lem:loop}
Let $G$ be a looped graph clasper of degree $d\geq 1$ in a homology cylinder $M$,
whose leaves $f_1,\dots,f_d$ and loop $\ell$ of edges are oriented as shown below:\\[0.1cm]
$$
\labellist
\small\hair 2pt
\pinlabel {$\varepsilon_1$}  at 42.5 37
\pinlabel {$\varepsilon_2$}  at 73 37 
\pinlabel {$\varepsilon_d$}  at 138.5 37
\pinlabel {$f_1$} [r] at 16 13
\pinlabel {$f_2$} [r] at 50 13
\pinlabel {$f_{d}$} [r] at 115 13
\pinlabel {$\ell$}  [b] at 2 75
\pinlabel {where} [l] at 196 49
\pinlabel {$0$} at 288 65
\pinlabel {$1$} at 288 33
\pinlabel {$=$} at 308 64
\pinlabel {$=$} at 308 32
\pinlabel {\scriptsize\!\!\! (a half-twist)} [l] at 318 21
\endlabellist
\!\!\! \includegraphics[scale=1.1]{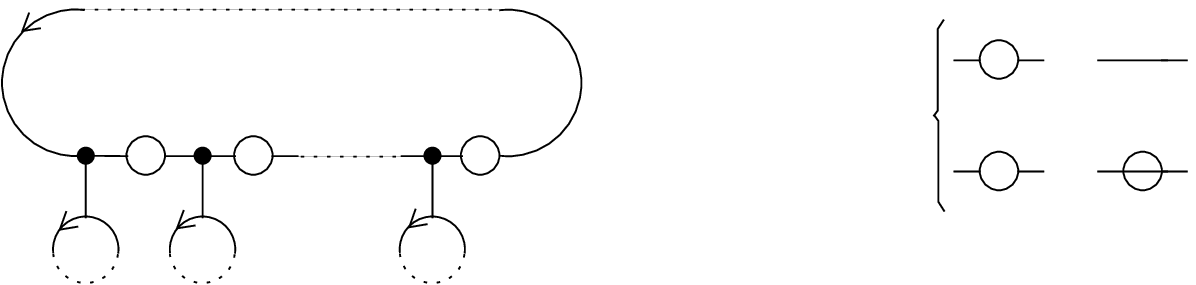}
$$
We denote by $h_1,\dots,h_d\in H$ and $b\in H$ the homology classes of $f_1,\dots,f_d$ and $\ell$, respectively,
and we set $\varepsilon:= \varepsilon_1 + \cdots + \varepsilon_d\in \Z_2$. Then we have
$$
\tau(M_G,\partial_-M_G;\xi_0) = P_{d,\varepsilon}(b^{-1},h_1,\dots,h_d) 
\cdot  P_{d,\varepsilon}(b,h_1^{-1},\dots,h_d^{-1}) \cdot \tau(M,\partial_-M;\xi_0)
$$
where we denote 
$\displaystyle{P_{d,\varepsilon}(Y,X_1,\dots,X_d) := Y + (-1)^{\varepsilon+1} \prod_{i=1}^d (1-X_i)} 
\in\Z[Y,X_1,\dots,X_d]$.
\end{lemma}

The proofs are as follows.

\begin{proof}[Proof of Lemma \ref{lem:multiplicativity}]
The identity in $Q(\Z[H])/\pm H$ easily follows from the multiplicativity of the Reidemeister torsion of acyclic chain complexes
with respect to direct sums \cite{Turaev_book}. Thus there are some unique $\varepsilon \in \{+1,-1\}$ and $h \in H$ such that
$$
\tau\big(M\circ M', \partial_- (M\circ M'); \xi_0\big) 
=\varepsilon h \cdot \tau(M,\partial_-M;\xi_0) \cdot \tau(M',\partial_-M';\xi_0).
$$
Setting $x(M) := \tau(M,\partial_-M;\xi_0)-t(M)^{\nicefrac{1}{2}}$, 
which belongs to  $I^2$ by Lemma \ref{lem:I2}, this identity reads
\begin{equation}
\label{eq:multiplicativity}
t(M\circ M')^{\nicefrac{1}{2}} + x(M\circ M')= 
\varepsilon h \cdot \left(t(M)^{\nicefrac{1}{2}}+x(M) \right) \cdot \left(t(M')^{\nicefrac{1}{2}}+x(M') \right).
\end{equation}
By reducing (\ref{eq:multiplicativity}) modulo $I$, we obtain $\varepsilon =1$.   
By reducing (\ref{eq:multiplicativity}) modulo $I^2$, we get 
$$
\left\{t(M\circ M')^{\nicefrac{1}{2}} - 1 \right\}= \left\{ h\cdot t(M)^{\nicefrac{1}{2}} \cdot t(M')^{\nicefrac{1}{2}} - 1 \right\} \ \in I/I^2 \simeq H
$$
which implies $h=1\in H$ (written multiplicatively).
\end{proof}

\begin{proof}[Proof of Lemma \ref{lem:loop}]
Analogues of this formula have already been proved in two contexts:
for the Alexander polynomial of knots in \cite{GL_loop}
and for the RT torsion of closed $3$-manifolds with Euler structure in \cite{Massuyeau_these}.
(It should be observed that, in both situations, the orientation conventions for claspers differ from ours.)
Since in the situation of homology cylinders, the Reidemeister torsion can be interpreted
as an Alexander polynomial (Proposition \ref{prop:Alexander_torsion}), the proof given by Garoufalidis and Levine
in \cite{GL_loop} can be adapted in a straightforward way to obtain that
\begin{equation}
\label{eq:looped_graph}
\tau(M_G,\partial_-M_G;\xi_0) =  \eta \cdot k \cdot P_{d,\varepsilon}(b^{-1},h_1,\dots,h_d) 
\cdot  P_{d,\varepsilon}(b,h_1^{-1},\dots,h_d^{-1}) \cdot \tau(M,\partial_-M;\xi_0)
\end{equation}
where $k\in H$ and $\eta=\pm1$ are unknown.
(We leave the details of the computations to the interested reader.)
To fix the indeterminacy in $\pm H$, it is enough to reduce (\ref{eq:looped_graph}) modulo $I^2$. 
We deduce from Lemma \ref{lem:I2} that
$$
t(M_G)^{\nicefrac{1}{2}} = \eta k \cdot  t(M)^{\nicefrac{1}{2}} \in \Z[H]/I^2.
$$
We conclude that $\eta=+1$ and $k=1\in H$ (written multiplicatively).
\end{proof}

\begin{proof}[Proof of Proposition \ref{prop:alpha_properties}]
Assertion (\ref{eq:alexcyl}) follows from (\ref{eq:torsion_mapping_cylinder}).
To show the additivity of $\alpha$, consider some $M,M'\in \cyl$. We abbreviate
$$
x := \tau(M,\partial_-M;\xi_0)-t(M)^{\nicefrac{1}{2}} \quad \hbox{and} \quad x' := \tau(M',\partial_-M';\xi_0)-t(M')^{\nicefrac{1}{2}}.
$$
By Proposition \ref{lem:multiplicativity}, 
$\tau\big(M\circ M', \partial_- (M\circ M'); \xi_0\big)$ is equal to
$$
\underbrace{t(M)^{\nicefrac{1}{2}}\cdot t(M')^{\nicefrac{1}{2}}}_{= t(M\circ M')^{\nicefrac{1}{2}}} + (x + x') + 
\underbrace{\left(t(M)^{\nicefrac{1}{2}}-1\right)x'+ \left(t(M')^{\nicefrac{1}{2}}-1\right)x+xx'}_{\in I^3}
$$
and we deduce that  $\alpha(M\circ M') =\alpha(M) + \alpha(M')$.

Thanks to Theorem \ref{thm:torsion_finiteness}, showing that $\alpha$ is a finite-type invariant of degree $\leq 2$ is equivalent to proving that
$\iota \circ t$ has the same property, where $\iota: H \to \Z[H]/I^3$ is the canonical map.
Let $M\in \cyl$ and let $S_0\sqcup S_1 \sqcup S_2$ be three pairwise-disjoint surfaces in $\interior(M)$,
together with some elements $s_0\in \Torelli(S_0), s_1\in \Torelli(S_1), s_2\in \Torelli(S_2)$. 
For $i=0,1,2$ we also set $t_i := t\left(M_{s_i}\right) -t(M) \in H$ (written additively).
For each $P\subset \{0,1,2\}$, let $M_P$ be the result of the Torelli surgeries along the surfaces $S_p$ for which $p \in P$.
Since $\tau_1$ is a finite-type invariant of degree $1$,  
$t\left(M_P\right)-t(M)$ is the sum of the $t_p$ for which $p\in P$. Therefore the alternate sum in $\Z[H]$
$$
\sum_{P\subset \{0,1,2\}} \!\! (-1)^{|P|} t\left(M_P\right)
= t(M) \!\!  \sum_{P\subset \{0,1,2\}} \!\! (-1)^{|P|} t\left(M_P\right)t(M)^{-1}
=  t(M) \!\!  \sum_{P\subset \{0,1,2\}} \!\! (-1)^{|P|} \prod_{p\in P} t_p
$$
is equal to $t(M)\cdot(1-t_0)(1-t_1)(1-t_2) \in I^3$. 
This shows that $\iota \circ t$, and consequently $\alpha$,  are finite-type invariants of degree at most $2$.

We now prove the surgery formulas.
In the case of the looped graph clasper $G$ of degree $2$,
 we deduce from Lemma \ref{lem:loop} that
\begin{eqnarray*}
&&\tau(M_G,\partial_-M_G;\xi_0) \\
&=&\left(1 - (1-h) (1-h') \cdot b -b^{-1} \cdot (1-h^{-1}) (1-h'^{-1}) \right) 
 \cdot \tau(M,\partial_-M;\xi_0) \mod I^3\\
&=& \tau(M,\partial_-M;\xi_0)  - (1-h) (1-h') \cdot b -b^{-1} \cdot (1-h^{-1}) (1-h'^{-1})  \mod I^3\\
&=& \tau(M,\partial_-M;\xi_0)  - (h-1) (h'-1) -  (1-h^{-1}) (1-h'^{-1})  \mod I^3.
\end{eqnarray*}
Since we have $t(M)=t(M_G)$, we conclude that $\alpha(M_G)=\alpha(M)- 2 hh'$.
(It also follows that the degree of  the finite-type invariant $\alpha$ is precisely $2$.)
The case of the $Y$-graph with two special leaves is deduced from Corollary \ref{cor:2spe}
and the fact that $\alpha$ is invariant under $Y_3$-equivalence (by Lemma \ref{lem:Y_FTI}).
\end{proof}

\subsection{The Casson invariant}\label{subsec:casson}

\label{subsec:Casson}

One can produce from the Casson invariant of homology $3$-spheres
an invariant of homology cylinders over $\Sigma$.
For this it is necessary to \emph{choose} an embedding of the surface $\Sigma$ in $S^3$ as follows.
Let $F_g\subset S^3$ be the surface obtained from the genus $g$ Heegaard surface of $S^3$
by removing a small open disk. We fix an orientation on $F_g$: 
the handlebody of the genus $g$ Heegaard splitting  that induces this orientation on $F_g$
is called \emph{lower} while the other one is called \emph{upper}.
An embedding $j:\Sigma \to S^3$ is called a \emph{Heegaard embedding} 
if its image is $F_g$ and if $j: \Sigma \to F_g$ is orientation-preserving.

Let $M$ be a homology cylinder over $\Sigma$.  
Denote by $S^3(M,j)$ the homology $3$-sphere obtained 
by ``cutting'' $S^3$ along $j(\Sigma)=F_g$  and by ``inserting'' $M$ at this place:
more precisely we define
\begin{equation}\label{eq:twisting_S3}
S^3(M,j) := \left(S^3 \setminus (j(\Sigma) \times [-1,1])\right) \cup_{j' \circ m^{-1}}  M
\end{equation}
where $j(\Sigma) \times [-1,1]$ denotes a closed regular neighborhood of $j(\Sigma)$ in $S^3$
and $j'$ is the restriction to the boundary of the homeomorphism 
$j \times \Id: \Sigma  \times [-1,1] \to j(\Sigma) \times [-1,1]$.
Evaluating the Casson invariant $\lambda$ on this homology $3$-sphere yields a map
$$
\lambda_j: \cyl \longrightarrow \Z, \ M \longmapsto \lambda\left(S^3(M,j)\right)
$$
which strongly depends on the embedding $j$.  
We sometimes abbreviate  $\lambda := \lambda_j$ and the dependence on $j$ is discussed in \S \ref{sec:core_Casson}.
 
It has been proved by Ohtsuki that the Casson invariant of homology $3$-spheres is a finite-type invariant \cite{Ohtsuki}.
More generally, the ``sum formulas'' of Morita \cite{Morita_Casson2} or Lescop \cite{Lescop} for the Casson invariant  
imply  that $\lambda_j: \cyl \rightarrow \Z$ is a finite-type invariant of degree 2.
The same formulas show that $\lambda_j: \cyl \rightarrow \Z$ is \emph{not} additive:
for any $M,M'\in \cyl$, we actually have 
\begin{equation}
\label{eq:sum_formula}
\lambda_j(M\circ M')= \lambda_j(M) + \lambda_j(M') +2 \cdot  \tau_1(M) \star_j  \tau_1(M')
\end{equation}
where $\star_j: \Lambda^3 H \times \Lambda^3 H \to \Z$ 
is a certain non-trivial bilinear pairing whose definition depends on $j$ \cite{Morita_Casson2,Lescop,CHM}.

Finally, let us observe that the function $\lambda_j$ is preserved by stabilization.
More precisely, if the surface $\Sigma$ is stabilized to a surface $\Sigma^s$ of genus $g^s$ as shown in Figure \ref{fig:stabilization}
and if the Heegaard embedding $j^s:  \Sigma^s \to F_{g^s}$ extends $j: \Sigma \to F_g$, 
then we have $\lambda_{j^s}(M^s) = \lambda_j(M)$ for all $M \in \cyl$.

\subsection{The Birman--Craggs homomorphism} \label{subsec:Birman-Craggs}

The Birman--Craggs homomorphism is a representation of the Torelli group derived 
from the Rochlin invariant of spin closed $3$-manifolds \cite{BC,Johnson_BC}.
This representation extends in a direct way to the monoid of homology cylinders \cite{Levine_enlargement,MM}.

To briefly recall its definition, we need the set $\Spin(\Sigma)$ of spin structures on $\Sigma$,
which is an affine space over the $\Z_2$-vector space $H^1(\Sigma;\Z_2)$.
As shown in  \cite{Johnson_quadratic},
the space of spin structures on $\Sigma$ can be identified with the space
$$
\left\{H\otimes \Z_2 \stackrel{q}{\longrightarrow} \Z_2: \forall x,y \in H\otimes \Z_2, \ 
q(x+y) - q(x) - q(y) = \omega(x, y)\  \operatorname{mod}\ 2 \right\}
$$
of \emph{quadratic forms} whose polar form is the intersection pairing  mod $2$,
and this identification will be tacit in the sequel.
We shall denote by
$$
B:= \operatorname{Map}(\Spin(\Sigma), \Z_2)
$$
the space of boolean functions on $\Spin(\Sigma)$.
For every $n\geq 0$, let $B_{\leq n}$ denote the subspace of $B$ consisting
of polynomial functions of degree at most $n$, \ie sums of products of $n$ affine functions.
In particular, $B_{\leq 1}$ is the space of affine functions  and includes the following:
\begin{equation}
\label{eq:affine_functions}
\left\{\begin{array}{rcl}
\Spin(\Sigma) & \overset{\overline{1}}{\longrightarrow} & \Z_2\\
q & \longmapsto & 1
\end{array}\right.
\quad \quad \quad
\left\{\begin{array}{rcl}
\Spin(\Sigma) & \overset{\overline{h}}{\longrightarrow} & \Z_2\\
q & \longmapsto & q(h)
\end{array}\right. \quad \hbox{where $h\in H$.}
\end{equation}
The \emph{$n$-th derivative} of  a boolean function  $f: \Spin(\Sigma) \to \Z_2$ at $\sigma \in \Spin(\Sigma)$
is the  map $\dd_\sigma^{n}f: H^1(\Sigma;\Z_2)^n \to \Z_2$ defined by
\begin{equation}
\label{eq:derivative}
\dd_\sigma^{n}f(y_1,\dots,y_n) := \sum_{P \subset \{1,\dots,n\}}(-1)^{|P|}\cdot f\left(\sigma + \overrightarrow{y_P}\right)
\end{equation}
where $y_P$ is the sum of the $y_p$'s for which $p\in P$. 
As a general fact, a map $f$ is polynomial of degree $\leq n$ if and only if $\dd_\sigma^{n+1}\! f$ vanishes at some point $\sigma$
and, in this case, $\dd_\sigma^{n}f$ does not depend on $\sigma$ and is multilinear. 
Since the ground field is here $\Z_2$, the signs do not count in (\ref{eq:derivative})
and the $n$-th derivative $\dd^{n}_\sigma f$ of a  function $f$ vanishes when two arguments are repeated.
Thus, we have a canonical isomorphism
\begin{equation}\label{eq:iso_Boolean}
\dd^{n}:B_{\leq n} /B_{\leq n-1} \stackrel{\simeq}{\longrightarrow} 
\Hom\left(\Lambda^n H^1(\Sigma;\Z_2), \Z_2 \right) \simeq \Lambda^n H_{(2)}
\end{equation}
where $H_{(2)}:= H_1(\Sigma;\Z_2) = H \otimes \Z_2$ 
is identified with $\Hom(H^1(\Sigma;\Z_2),\Z_2)$.

Now, let $j:\Sigma \hookrightarrow S^3$ be any embedding.  
Pulling back the (unique) spin structure $\sigma_0$ of $S^3$ by $j$ gives an element $j^\ast\sigma_0\in \Spin(\Sigma)$,
and any element of $\Spin(\Sigma)$ can be realized in this way.  
For a homology cylinder $M$ over $\Sigma$, denote by $S^3(M,j)$ the  homology $3$-sphere defined at (\ref{eq:twisting_S3}).
Evaluating  Rochlin's  $\mu$ invariant on this homology sphere yields a monoid homomorphism
$ \cyl \to \Z_{2}$ which only depends on $j^*\sigma_0\in \Spin(\Sigma)$.  
By making the spin structure vary on $\Sigma$, ones gets the \emph{Birman--Craggs homomorphism}
$$ 
\beta: \cyl \longrightarrow  B_{\leq 3}, \
M\longmapsto \left(  j^\ast\sigma_0\mapsto \mu\left(S^3(M,j)\right)\right).
$$
The map $\beta$ is an additive finite-type invariant of degree $1$,
which is preserved by stabilization of the surface $\Sigma$.
Let us recall how $\beta$ changes under surgery along a $Y$-graph.

\begin{lemma}\label{lem:Y_beta}
Let $Y$ be a $Y$-graph in a homology cylinder $M$, whose leaves are ordered and oriented in an arbitrary way.
We denote by $h_1,h_2,h_3\in H$ their homology classes and 
by $f_1,f_2,f_3 \in \Z$ their framing numbers in $M$
(as defined in  Appendix \ref{subsec:framing_numbers}).
Then, using the notation (\ref{eq:affine_functions}), we have
$$
\beta\left(M_Y\right) - \beta(M) = \prod_{i=1}^3 \left(\overline{h_i}+ f_i \cdot \overline{1} \right) \ \in B_{\leq 3}.
$$
\end{lemma}

\begin{proof}
This formula is essentially contained in \cite{MM}. 
Let $q\in \Spin(\Sigma)$ (which we think of as a quadratic form $H_{(2)} \to \Z_2$ with polar form $\omega\otimes \Z_2$)
and let $\sigma$ be the unique spin structure on $M$  
whose pull-back by $m_+:\Sigma \to M$ (or, equivalently, by $m_-$) gives $q$.
We denote by $FM$ the bundle of oriented frames of $M$ with fiber $\operatorname{GL}_+(3;\R)$,
and we think of $\sigma$ as an element of $H^1(FM;\Z_2)$ which is not zero on the fiber.
We deduce from \cite[Lemma 3.14]{MM}  that
$$
\langle \beta\left(M_Y\right) - \beta(M), q \rangle = \prod_{i=1}^3 \langle \sigma, t_{L_i}  \rangle \ \in \Z_2
$$
where $L_1, L_2,L_3$ denote the leaves of $Y$
and, for any framed oriented knot $K$ in $M$, $t_K \in H_1(FM)$ denotes the homology class
of the oriented curve in $FM$ obtained by lifting $K$ with an extra $(+1)$-twist  to $FM$.
Thus it is enough to check the following.

\begin{claim}
For any oriented framed knot $K$ in $M$, we have $q([K]) = \langle \sigma, t_K \rangle + \Fr(K)$
where $[K]\in H$ denotes the homology class of $K$ and $\Fr(K) \in \Z$ its framing number.
\end{claim}

\noindent
We set $q'(K) :=  \langle \sigma, t_K \rangle + \Fr(K) \in \Z_2$ and we first check that $q'(K)$ only depends on the homology class of $K$.
For this, let $K_1$ and $K_2$ be two oriented framed knots such that $[K_1]=[K_2] \in H$. 
Then we can find a compact oriented surface $S \subset M$ such that $\partial S = K_1 \sqcup (-K_2)$
and  $K_1$ is $0$-framed with respect to this surface $S$. Let $n$ be the framing of $K_2$ with respect to $S$,
and let $K_2'$ be the oriented framed knot obtained from $K_2$ 
by  an extra $(-n)$-twist so that $K'_2$ is $0$-framed with respect to $S$.
It follows from \cite[Lemma 2.7]{MM} that $t_{K_1}=t_{K_2'}$
and, using Lemma \ref{lem:framed_connect}, we obtain that $\Fr(K_1)=\Fr(K_2')$.
We  deduce that
$$
\langle \sigma, t_{K_1} \rangle + \Fr(K_1) = \langle \sigma, t_{K_2'} \rangle + \Fr(K_2') = 
\langle \sigma, t_{K_2}  \rangle +n +  \Fr(K'_2) = \langle \sigma, t_{K_2} \rangle + \Fr(K_2) \in \Z_2. 
$$
Thus, we get a map $q':H \to \Z_2$. Moreover this map is quadratic with polar form $\omega\otimes \Z_2$ since,
for any oriented framed knots $K$ and $L$ in $M$, we have
\begin{eqnarray*}
q'([K]+[L]) &= &q'([K \sharp L]) \\
 &= &\langle \sigma, t_{K} + t_{L} \rangle + \Fr(K) + \Fr(L) + 2 \Lk(K,L)\\
 & =& q'([K])+ q'([L]) + \omega([K],[L]).
\end{eqnarray*}
(Here we have used  \cite[Lemma 2.7]{MM} and Appendix B again.) 
In particular, it follows that $q'$ factorizes to a quadratic form $q':H_{(2)} \to \Z_2$. 
To conclude that $q=q'$, it is enough to check 
that  $q([\alpha])=q'([\alpha])$ for any oriented simple closed curve $\alpha$ on $\Sigma$.
Let $\alpha_+$  be the oriented framed knot obtained by pushing the curve $m_+(\alpha)$,
framed along $m_+(\Sigma)$, in the interior of $M$. 
The way a spin structure on $\Sigma$ is identified with a quadratic form in \cite{Johnson_quadratic} implies that
$q([\alpha])= \langle \sigma, t_{\alpha_+}\rangle$ 
and, since we have $\Fr(\alpha_+)=0$ in this case, we conclude that $q([\alpha])=q'([\alpha])$.
\end{proof}

An  important property of $\beta$ is that, for any $M \in \cyl$, 
the third derivative of the cubic function $\beta(M)$ is the mod $2$ reduction of $\tau_1(M)$:
\begin{equation}
\label{eq:d^3beta}
\xymatrix{
\cyl \ar[r]^-\beta \ar[d]_-{\tau_1} & B_{\leq 3} \ar[r]^-{\dd^3} & \Lambda^3 H_{(2)} \ar@{=}[d]\\
\Lambda^3 H \ar@{->>}[rr]_-{\mod  2 } && \Lambda^3 H_{(2)}
}
\end{equation}
This relation is due to Johnson \cite{Johnson_abelianization} (in the case of the Torelli group), 
and it can be proved  by comparing how $\beta$ and $\tau_1$ change under surgery along a $Y$-graph \cite{MM}.
The following lemmas give the next derivatives of $\beta$.

\begin{lemma}\label{lem:d^2beta}
The following diagram is commutative:
$$
\xymatrix{
\Jcob \ar[d]_{\tau_2} \ar[r]^{\beta} & B_{\leq 2} \ar[r]^-{\dd^2} & \Lambda^2H_{(2)} \\
\frac{\left(\Lambda^2 H \otimes \Lambda^2 H\right)^{\mathfrak{S}_2}}{\Lambda^4H} 
\ar@{->>}[rr]_-{L} && \frac{\Lambda^2H}{2 \cdot \Lambda^2H} \ar[u]_-\simeq
} 
$$
Here $\Jcob = \cob[2]$ denotes the second term of the Johnson filtration and
$L$ is the homomorphism appearing in the short exact sequence (\ref{eq:sym_to_sym}).
\end{lemma}

\begin{proof}
It is a consequence of (\ref{eq:d^3beta}) that the restriction of $\beta$ to $\Jcob$ takes its values in $B_{\leq 2}$.  
Now let $M\in \Jcob$. 
Using Lemma \ref{lem:Y_beta}, we  can find some $Y$-graphs with one special leaf 
$G_1, \dots, G_m$ in $(\Sigma \times I)$ such that 
$$
\beta(M) = \sum_{i=1}^m \beta \left( (\Sigma\times I)_{G_i} \right) \ \in B_{\leq 2}.
$$
Since the $Y_2$-equivalence is classified by the pair $(\tau_1,\beta)$ \cite{MM}, we deduce that
$$
M\stackrel{Y_2}{\sim}  \prod_{i=1}^m (\Sigma\times I)_{G_i}.
$$
Therefore, by clasper calculus, we can find 
graph claspers $H_1,\dots, H_n$ 
of degree $2$ in $(\Sigma\times I)$ such that
$$
M\stackrel{Y_3}{\sim}  \prod_{j=1}^n (\Sigma \times I)_{H_j} \circ \prod_{i=1}^m (\Sigma\times I)_{G_i}.
$$ 
Since $\tau_2$ is invariant under $Y_3$-equivalence, we have  
\begin{equation}
\tau_2(M) = \sum_{j=1}^n \tau_2\left((\Sigma\times I)_{H_j} \right) +  \sum_{i=1}^m \tau_2\left((\Sigma\times I)_{G_i} \right). 
\end{equation}
We deduce from Lemma \ref{lem:tau2} and Lemma \ref{lem:tau2_2} that 
$$
L\tau_2(M) = 0 + \sum_{i=1}^m l_i \wedge r_i \ \in \Lambda^2 H_{(2)}
$$
where $l_i,r_i \in H_{(2)}$ denote the mod $2$ homology classes of the non-special leaves of $G_i$.
We conclude thanks to Lemma \ref{lem:Y_beta}.
\end{proof}

\begin{rem}
The fact that the second Johnson homomorphism $\tau_2$ is related to the Birman--Craggs homomorphism $\beta$
has already been observed by Yokomizo for  the Johnson subgroup of the Torelli group \cite{yokomizo}.  
\end{rem}

\begin{lemma}\label{lem:d^1beta}
The following diagram is commutative:
$$
\xymatrix{
\cob[3] \ar[d]_{\alpha} \ar[r]^{\beta} &   B_{\leq 1} \ar[r]^{\dd^1} & H_{(2)} \ar@{>->}[d]^-s\\
S^2 H  \ar@{->>}[rr]_-{\mod 2} & &S^2H_{(2)} }
$$
Here $s$ is the square map defined by $x \mapsto x^2$.  
\end{lemma}

\begin{proof}
It is a consequence of Lemma \ref{lem:d^2beta} that the restriction of $\beta$ to $\cob[3]$ takes its values in $B_{\leq 1}$.
Let  now $M\in \cob[3]$.
Using the same arguments as in the proof of  Lemma \ref{lem:d^2beta},
we can find some some $Y$-graphs $G_1,\dots ,G_m$ in $(\Sigma \times I)$ with \emph{two} special leaves
and graph claspers $H_1,\dots,H_n$ 
of degree $2$ in $(\Sigma \times I)$ such that
$$
M\stackrel{Y_3}{\sim}  \prod_{j=1}^n (\Sigma \times I)_{H_j} \circ \prod_{i=1}^m (\Sigma\times I)_{G_i}.
$$
Since we have $\tau_2(M)=0$ by assumption
and $\tau_2\left((\Sigma\times I)_{G_i}\right)=0$ by Lemma \ref{lem:tau2_2},
we deduce that 
$\sum_j \tau_2\left((\Sigma \times I)_{H_j}\right)=0$.
Then a more delicate use  of clasper calculus 
shows that we can assume each graph clasper 
$H_j$ to be looped.
(This will be proved in Lemma \ref{lem:tau2_loop} below.)
We deduce from Proposition \ref{prop:alpha_properties} that
$$
\alpha(M) = \sum_{j=1}^n\underbrace{\alpha\left((\Sigma \times I)_{H_j} \right)}_{\in 2\cdot S^2H}
+ \sum_{i=1}^m \underbrace{\alpha\left((\Sigma\times I)_{G_i}\right)}_{=  g_i^2}
$$
where $g_i \in H$ is the homology class of the non-special leaf of $G_i$ (which is oriented in an arbitrary way). 
Therefore, we have
$$
\alpha(M)\!\! \mod 2 = \sum_{i=1}^m  g_i^2 \in S^2 H_{(2)}.
$$
Again, we conclude thanks to Lemma \ref{lem:Y_beta}. 
\end{proof}

\begin{remark}
Since $\alpha$ is trivial on the Torelli group, 
we deduce from the previous lemmas that the  function $\beta(f):\Spin(\Sigma)\to \Z_2$ is constant for any $f\in \mcg[3]$.
This phenomenon has already been observed by Johnson in  \cite[p.178]{Johnson_survey}.
\end{remark}

The next statement is deduced from the previous lemmas, and it will be used in the proof of Theorem A.

\begin{lemma}\label{lem:beta}
If two homology cylinders $M$ and $M'$ over $\Sigma$ have the same invariants 
$\rho_3$, $\alpha$ and $\lambda_j$ (for some Heegaard embedding $j$ of $\Sigma$ in $S^3$),
then the  Birman--Craggs homomorphism $\beta$ does not distinguish $M$ from $M'$.  
\end{lemma}

\begin{proof} 
Since the quotient monoid $\cyl/Y_3$ is a group according to Goussarov and Habiro \cite{Goussarov,Habiro}, 
there exists an $N \in \cyl$ such that $M\circ N$ is $Y_3$-equivalent to $M'$.
The fact that $\rho_3(M)=\rho_3(M')$ implies that $\rho_3(N)=1$, so that $N$ belongs to $\cob[3]$.
By (\ref{eq:d^3beta}) and Lemma \ref{lem:d^2beta} we obtain that $\beta(N)\in B_{\leq 1}$.  
Next,  the fact that $\alpha(M)=\alpha(M')$ implies that $\alpha(N)=0$, and  
we deduce from Lemma \ref{lem:d^1beta} that the boolean function $\beta(N): \Spin(\Sigma) \to \Z_2$ is constant.  
Finally, $M$ and $M'$ having the same Casson invariant $\lambda_j$, 
we deduce from formula (\ref{eq:sum_formula}) that $\lambda_j(N)=0$. 
Therefore, we have
$$
\beta(N)(j^*\sigma_0) = \mu\left(S^3(N,j)\right)
= \lambda\left(S^3(N,j)\right) \! \! \!\! \mod 2 = \lambda_j(N) \! \! \!\! \mod 2 = 0
$$
which shows that $\beta(M')-\beta(M) = \beta(N)$ is the trivial map.
\end{proof}

\section{Some diagrammatic invariants of homology cylinders} \label{sec:diagrams}

In this section, we briefly review the LMO homomorphism 
(which is a diagrammatic representation of the monoid $\cyl$ introduced in \cite{CHM,HM_SJD})
and its connection with clasper surgery.
We recall how the invariants $\tau_1$, $\tau_2$ and $\lambda$ introduced in \S \ref{sec:invariants}
can be  extracted from the LMO homomorphism \cite{CHM}, 
and we give a similar result for the quadratic part $\alpha$ of the relative RT torsion.

\subsection{Jacobi diagrams} \label{subsec:Jacobi}

We start by defining the diagrammatic spaces that we shall need.
A \emph{Jacobi diagram} is a finite graph whose vertices  have valency $1$ or $3$:
univalent vertices are called \emph{external} vertices,
while trivalent vertices are called \emph{internal} vertices and are assumed to be oriented. 
(An \emph{orientation} of a vertex is a cyclic ordering of its  incident half-edges.)
The \emph{internal degree} (or simply \emph{degree}) of a Jacobi diagram is its number of internal vertices,
and the \emph{loop degree}  is its first Betti number.
We call a connected Jacobi diagram of internal degree $0$ a {\it strut},
and we call a connected Jacobi diagram of internal degree $1$ and loop degree $1$ a {\it lasso}.
A Jacobi diagram is \emph{colored} by a set $S$ if a map from the set of its external vertices to $S$ is specified.
As usual \cite{BN} for figures, we use dashed lines to depict Jacobi diagrams, 
and we take the cyclic ordering at a trivalent vertex given by the counter-clockwise orientation: 
see Figure \ref{fig:diagrams_examples} for some examples.

\begin{figure}[h]
\begin{center}
\includegraphics[scale=0.45]{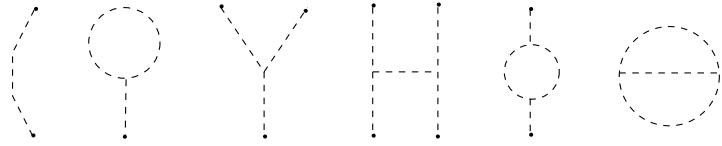}
\end{center}
\caption{Some examples of Jacobi diagrams: the strut, the lasso,
the Y graph, the H graph, the $\Phi$ graph and the $\Theta$ graph.}
\label{fig:diagrams_examples}
\end{figure}

As in the previous sections, we denote $H:=H_1(\Sigma)$ and $H_\Q := H\otimes \Q$.
We consider the following abelian group:
$$
\jacobi(H) := 
\frac{\Z\cdot \left\{ \begin{array}{c} \hbox{Jacobi diagrams without strut component}\\
\hbox{and with external vertices colored by } H  \end{array} \right\}}
{\hbox{AS, IHX, loop, multilinearity}}.
$$
The ``AS'' and ``IHX'' relations are diagrammatic analogues of the antisymmetry and Jacobi identities  in Lie algebras:\\

\begin{center}
\labellist \small \hair 2pt
\pinlabel {AS} [t] at 102 -5
\pinlabel {IHX} [t] at 543 -5
\pinlabel {$= \ -$}  at 102 46
\pinlabel {$-$} at 484 46
\pinlabel {$+$} at 606 46
\pinlabel {$=0$} at 721 46 
\endlabellist
\centering
\includegraphics[scale=0.4]{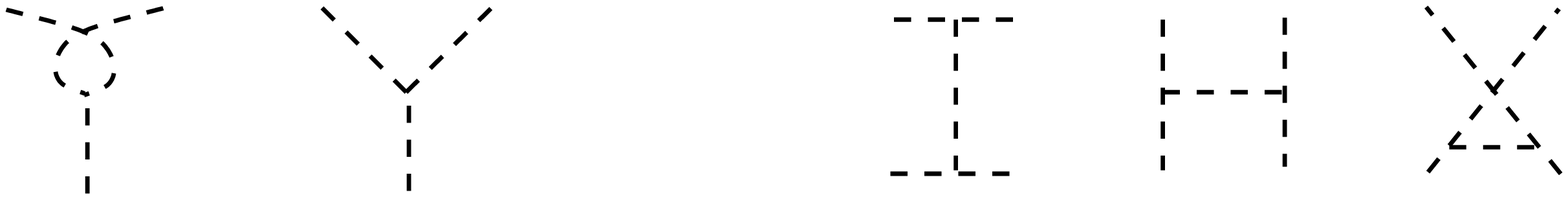}
\end{center}
\vspace{0.5cm}

\noindent
The ``loop'' relation says that a Jacobi diagram with a looped edge (e.g$.$ a lasso) is trivial,
and follows from the IHX relation in internal degree $\geq 2$.
The ``multilinearity'' relation states that a Jacobi diagram $D$ having one external vertex $v$ colored
by $n_1\cdot h_1 + n_2\cdot h_2$ (with $n_1,n_2 \in \Z$ and $h_1,h_2 \in H$)
is equivalent to the  linear combination $n_1\cdot D_1 + n_2 \cdot D_2$
where $D_i$ is the Jacobi diagram $D$ with the vertex $v$ colored by $h_i$. 
The abelian group $\jacobi(H)$ is graded by the internal degree:
$$
\jacobi(H) = \bigoplus_{i=0}^\infty \jacobi_i(H),
$$
the degree $0$ part being spanned by the empty diagram $\varnothing$.
We denote by $\jacobi^c(H)$ the subgroup of $\jacobi(H)$ spanned by connected Jacobi diagrams.
Its internal degree $i$ part $\jacobi^c_{i}(H)$ is in turn graded by the loop degree: 
$$ 
\jacobi^c_{i}(H)=\bigoplus_{l=0}^{d_i} \jacobi^c_{i,l}(H), 
$$ 
where $d_i=(i+2)/2$ if $i$ is even, and $d_i=(i-1)/2$ if $i$ is odd.
The projection onto the loop-degree $l$ part of $\jacobi^c_{i}(H)$ is denoted by
$$ 
p_{i,l}:\jacobi^c_{i}(H)\longrightarrow \jacobi^c_{i,l}(H).
$$ 
There is also a version of  $\jacobi(H)$ with rational coefficients: 
this $\Q$-vector space is denoted by $\jacobi(H_\Q)$ 
and is canonically isomorphic to $\jacobi(H)\otimes \Q$.

The space $\jacobi(H)$ is well suited for computations which, for instance,
involve the representation theory of the symplectic group $\Sp(H)$.
However, from a topological point of view, 
the following variant of $\jacobi(H)$, which has been introduced by Habiro in \cite{Habiro}, is more convenient to use with clasper calculus:  
see \S \ref{subsec:surgery_map} in this connection.
$$
\jacobi^{<}(H) := 
\frac{\Z\cdot \left\{ \begin{array}{c} \hbox{Jacobi diagrams without strut component and}\\
\hbox{with external vertices colored by } H \hbox{ and totally ordered}  \end{array} \right\}}
{\hbox{AS, IHX, loop, multilinearity, STU-like}}.
$$
The AS, IHX, loop and multilinearity relations are  as before,
while the ``STU-like'' relation is defined as follows:
$$
\begin{array}{c}
\labellist \small \hair 2pt
\pinlabel {$=\quad \quad \omega(x,y)\cdot$} at 395 50
\pinlabel {$x$} [t] at 2 0
\pinlabel {$y$} [t] at 58 0
\pinlabel {$y$} [t] at 203 0
\pinlabel {$x$} [t] at 256 0
\pinlabel {$-$} at 130 45
\pinlabel {$<$} at 30 4
\pinlabel {$<$} at 229 4
\pinlabel {$\cdots<$} [r] at 0 4
\pinlabel {$\cdots<$} [r] at 198 4
\pinlabel {$< \cdots$} [l] at 61 4
\pinlabel {$< \cdots$} [l] at 261 4
\endlabellist
\centering
\includegraphics[scale=0.6]{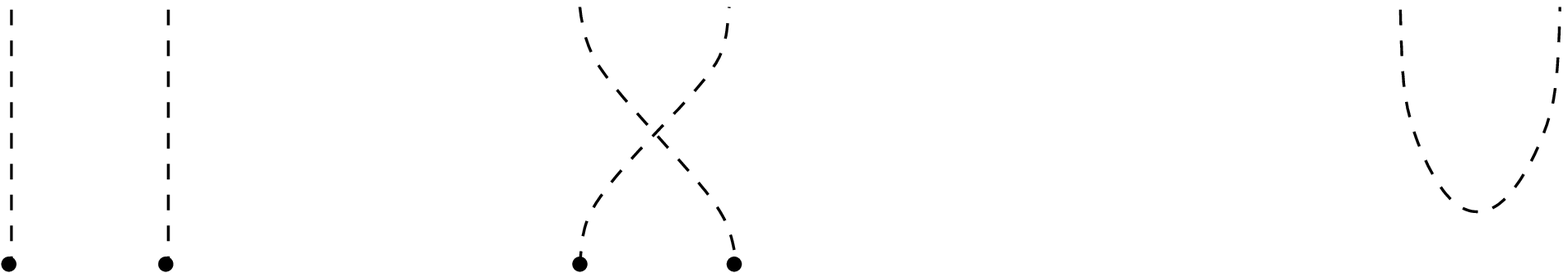}
\end{array}
$$
(Recall that $\omega:H \times H \to \Z$ denotes the intersection pairing.) 
Again, there is a rational version of $\jacobi^{<}(H)$, 
which is denoted by $\jacobi^{<}(H_\Q)$ and which is canonically isomorphic to $\jacobi(H)\otimes \Q$.
With rational coefficients, there is an isomorphism
$$ 
\chi: \jacobi(H_\Q) \stackrel{\simeq}{\longrightarrow} \jacobi^{<}(H_\Q)
$$
defined, for all Jacobi diagram $D \in \jacobi(H_\Q)$ with $e$ external vertices, by
$$
\chi(D) := \frac{1}{e!} \cdot \left(\hbox{sum of all ways of ordering the $e$ external vertices of $D$} \right)
$$
(See  \cite[Proposition 3.1]{HM_SJD}.) 
Besides there is another version of the abelian group $\jacobi^{<}(H)$,
which is denoted by  $\jacobi^{<}(-H)$ and is defined as $\jacobi^{<}(H)$, 
except that one uses the symplectic form $-\omega$ in the STU-like relation instead of $\omega$.
These two spaces are canonically isomorphic, via the map  
 $$ 
 s: \jacobi^{<}(-H) \stackrel{\simeq}{\longrightarrow} \jacobi^{<}(H) 
$$
defined by $s(D) := (-1)^{w(D)}D$ for any Jacobi diagram $D$, 
where $w(D)$ denotes the Euler characteristic of $D$ modulo $2$.

Finally, we shall need a third space of Jacobi diagrams.
We denote by $L^\pm$ the abelian group freely generated by the set
$$
\{1^+,\dots,g^+\} \cup \{1^-,\dots,g^-\},
$$
which consists of two copies of the finite set  $\{1,\dots,g\}$,  labeled by ``$+$'' and ``$-$'' respectively.
Then, we consider the abelian group
$$
\jacobi(L^\pm) := 
\frac{\Z\cdot \left\{ \begin{array}{c} \hbox{Jacobi diagrams without strut component}\\
\hbox{and with external vertices colored by } L^\pm \end{array} \right\}}
{\hbox{AS, IHX, loop, multilinearity}}
$$
and its version $\jacobi(L^\pm_\Q)$ with rational coefficients.
We also denote by $\jacobi^c(L^\pm)$ the subgroup of $\jacobi(L^\pm)$ spanned by connected Jacobi diagrams.
There is a projection of $\jacobi^c_i(L^\pm)$ 
onto its loop-degree $l$ part which we denote by
$$ 
p_{i,l}:\jacobi^c_{i}(L^\pm)\longrightarrow \jacobi^c_{i,l}(L^\pm).
$$
Of course, if we choose a system of meridians and parallels $(\alpha_i,\beta_i)_{i=1}^g$ on the surface $\Sigma$
as depicted in Figure \ref{fig:basis}, then we have an ``obvious'' isomorphism between $\jacobi(L^\pm)$ and $\jacobi(H)$
which is induced by the group isomorphism 
\begin{equation}\label{eq:L+-_H}
L^\pm \stackrel{\simeq}{\longrightarrow} H,\quad \ i^- \longmapsto [\alpha_i], \ j^+ \longmapsto [\beta_j].
\end{equation}
Observe that the subgroup generated by $\{1^-,\dots,g^-\}$ (respectively $\{1^+,\dots,g^+\}$) then corresponds
to the Lagrangian subgroup $\langle \alpha_1,\dots, \alpha_g \rangle$ of $H$ (respectively $\langle \beta_1,\dots, \beta_g \rangle$).
But there is also a ``non-obvious'' isomorphism between $\jacobi(L^\pm_\Q)$ and $\jacobi(H_\Q)$, 
namely the map $\kappa$ defined by the following composition:
\begin{equation}
\label{eq:kappa}
\xymatrix{
\jacobi(L^\pm_\Q) \ar[r]^-\varphi_-\simeq \ar@{-->}@/_1.8pc/[rrr]_-\kappa
&  \jacobi^{<}(-H_\Q) \ar[r]^-s_-\simeq & \jacobi^<(H_\Q) \ar[r]^-{\chi^{-1}}_-\simeq & \jacobi(H_\Q)
}
\end{equation}
Here the isomorphism $\varphi$ is  defined by declaring 
that ``each $i^-$-colored vertex should be lower than any $i^+$-colored vertex''
and by changing the colors of external vertices according to (\ref{eq:L+-_H}).
(See \cite[Lemma 8.4]{CHM}.)  Note that $\kappa$ is explicitly given by the formula 
\begin{equation}
\label{eq:kappa_formula}
\kappa(D) = (-1)^{w(D)}\cdot 
\left( \! \! \!  \left. \begin{array}{c} \hbox{sum of all ways of $(\times 1/2)$-gluing \emph{some} $i^-$-colored}\\
\hbox{vertices of $D$ with \emph{some} of its $i^+$-colored vertices} \end{array}
\! \right| \! \begin{array}{l} j^+ \mapsto [\beta_j] \\ j^- \mapsto [\alpha_j] \end{array}  \! \! \! 
\right)
\end{equation}
for any Jacobi diagram $D$, 
where $w(D)$ denotes the Euler characteristic of $D$ modulo $2$, 
and where a ``$(\times 1/2)$-gluing'' means a gluing together with a multiplication by $1/2$.
The reasons to be interested in this more sophisticated isomorphism $\kappa$ will be apparent in the next subsection.

\begin{figure}
\begin{center}
{\labellist \small \hair 0pt 
\pinlabel {$\alpha_1$} [l] at 310 105
\pinlabel {$\alpha_g$} [l] at 568 65
\pinlabel {$\beta_1$} [b] at 216 80
\pinlabel {$\beta_g$} [b] at 472 69
\endlabellist}
\includegraphics[scale=0.52]{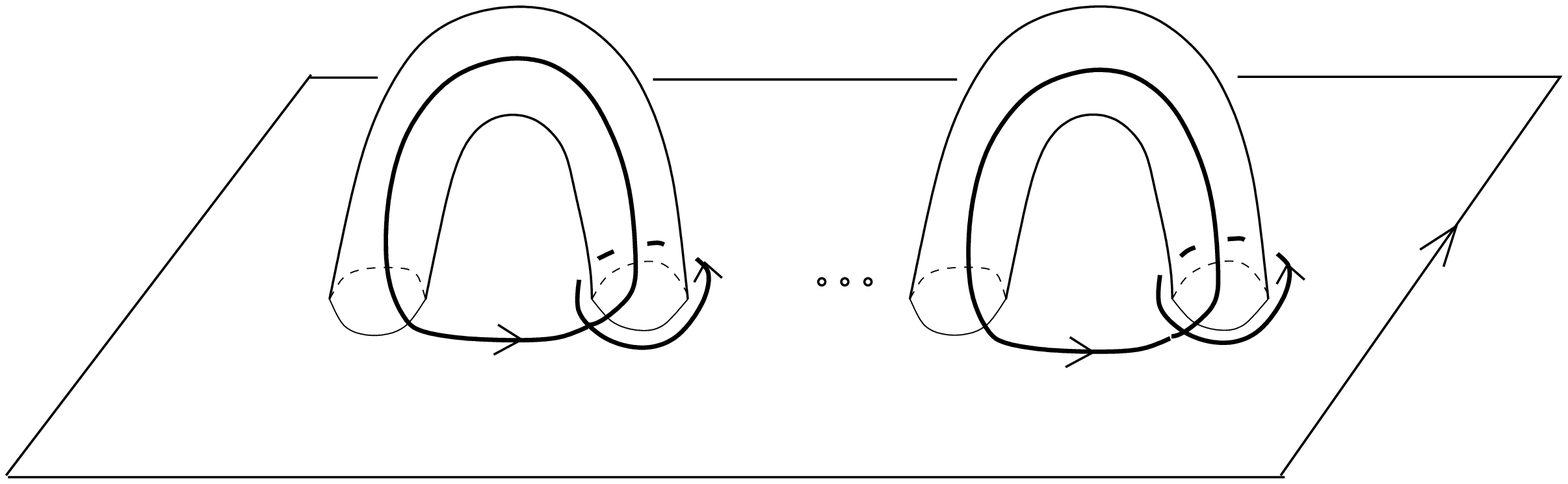}
\end{center}
\caption{The surface $\Sigma$ and a system of meridians and parallels $(\alpha,\beta)$.}
\label{fig:basis}
\end{figure}

We conclude this review of diagrammatic spaces by the following technical fact.

\begin{lemma}\label{lem:no_torsion}
The abelian groups $\jacobi_2(H)$ and  $\jacobi^{<}_2(H)$ are torsion-free. 
\end{lemma}

\begin{proof}
The isomorphism $\varphi$, whose definition is recalled in the previous paragraph, exists with integral coefficients.
(Indeed, the proof given in  \cite[Lemma 8.4]{CHM} works with coefficients in $\Z$ as well.)
Therefore, the abelian group $\jacobi^{<}(H)$ is isomorphic to $\jacobi(L^\pm) \simeq \jacobi(H)$,
so that it is enough to prove the lemma for $\jacobi_2(H)$.

The abelian group $\jacobi_2(H)$ is the direct sum of $\jacobi_2^c(H)$ and $S^2\jacobi_1^c(H) \simeq S^2 \Lambda^3H$.
To see that $\jacobi_2^c(H)$ has no torsion, we use the loop degree:
$$  
\jacobi^c_2(H) =  \jacobi^c_{2,0}(H)\oplus \jacobi^c_{2,1}(H)\oplus \jacobi^c_{2,2}(H).
$$
By the multilinearity relation, we have 
$$ 
\jacobi^c_{2,1}(H)=\left\langle \left. \phin{h}{h'} \right\vert h,h'\in H \right\rangle\simeq S^2H 
\quad \textrm{and}\quad \jacobi^c_{2,2}(H)= \left\langle \thetagraph \right\rangle\simeq \Z. 
$$
Thus it remains to prove that the group $\jacobi^c_{2,0}(H)$ is torsion-free.  
Note that we have an isomorphism $\jacobi^c_{2,0}(H) \simeq \frac{S^2 \Lambda^2 H}{\Lambda^4 H}$ 
defined by $\hn{a}{b}{c}{d} \longmapsto \left((a\wedge b) \leftrightarrow (c\wedge d)\right)$.
Now, recall from the proof of Proposition \ref{prop:Morita-Levine} the homomorphism
$$
\jacobi^c_{2,0}(H) \simeq \frac{S^2 \Lambda^2 H}{\Lambda^4 H} \stackrel{\eta'}{\longrightarrow} \DD'_2(H)
$$
defined by
$$
\hn{a}{b}{c}{d} \longmapsto a \otimes [b,[c,d]] + b \otimes [[c,d],a]+ c \otimes [d,[a,b]] + d \otimes [[a,b],c].
$$
As shown by Levine \cite{Levine_addendum}, this map is an isomorphism.
Since $\DD'_2(H)$ is a subgroup of $\DD_2(H)$ which has no torsion, 
we deduce that $\jacobi^c_{2,0}(H)$ is torsion-free.
\end{proof}

\subsection{The surgery map}\label{subsec:surgery_map}

For each integer $k\geq 2$, Habiro has defined in \cite{Habiro} a surjective homomorphism
$$ 
\psi_k: \jacobi^{<,c}_{k}(H)\longrightarrow \frac{Y_k\cyl}{Y_{k+1}},
$$
which sends each connected Jacobi diagram $D\in \jacobi^{<,c}_k(H)$ 
to the homology cylinder obtained from $(\Sigma \times I)$ by surgery 
along a graph clasper $C(D)$ of degree $k$ with the same shape as $D$.
This ``topological realization'' $C(D)$ of the diagram $D$ is defined in the following way:

\begin{enumerate}
\item Thicken $D$ to a compact oriented surface using the vertex-orientation of $D$, 
so that vertices are thickened to disks, and edges to bands.  
For each disk coming in this way from an external vertex, cut a smaller disk in the interior, 
so as to produce an oriented compact surface $S(D)$, decomposed into disks, bands and annuli.
The orientation of $S(D)$ induces an orientation on the core of each annulus 
as shown on Figure \ref{fig:leaf_orientation}:
\begin{figure}[h!]
\labellist
\small\hair 2pt
\pinlabel {{\tiny $+$}} [ c] at 395 113
\endlabellist
\includegraphics[scale=0.3]{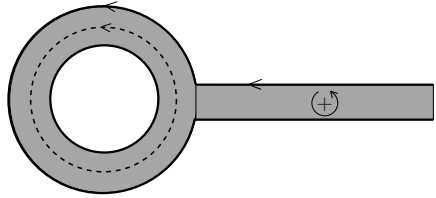}
\caption{How to orient a leaf.}
\label{fig:leaf_orientation}
\end{figure}
\item Embed $S(D)$ into the interior of $(\Sigma \times I)$ in such a way that each annulus of $S(D)$ 
represents in $H\simeq H_1(\Sigma \times I)$ the color  of the corresponding external vertex of $D$. 
The embedded annuli should be in disjoint ``horizontal layers'' of $(\Sigma \times I)$, 
and their ``vertical height'' along $I$ should respect the total ordering of the external vertices of $D$.  
The result is a graph clasper $C(D)$ in $(\Sigma \times I)$.
\end{enumerate}

\noindent
Actually, we will only need the surgery map $\psi_k$ in degree $k=2$.
The degree $1$ case is special.  
We have $\jacobi^{<,c}_1(H)=\jacobi^{<}_1(H)\simeq \jacobi_1(H)$,
and the surgery map $\psi_1$ is still well-defined on that group (and it is surjective) 
\emph{provided} the torsion of the abelian group $\cyl/Y_2$ is ignored:
$$
\psi_1: \jacobi_{1}(H)\longrightarrow \frac{\cyl/Y_2}{\Tors(\cyl/Y_2)}.
$$
(The group $\cyl/Y_2$ contains $2$-torsion \cite{Habiro}:
its explicit surgery description involves spin structures of $\Sigma$ \cite{MM}.) 

The following lemma is needed in the proof of Lemma \ref{lem:d^1beta}.
We shall prove it by means of the surgery map $\psi_2$.

\begin{lemma}\label{lem:tau2_loop}
Let $N \in \cyl$ be such that $N \stackrel{Y_2}{\sim} \Sigma \times I$ and $\tau_2(N)=0$.
Then there exists a disjoint union $L$ of looped graph claspers of degree $2$
and graph claspers of degree $3$
in $(\Sigma \times I)$ such that surgery along $L$ yields $N$.
\end{lemma}

\begin{proof}
By surjectivity of $\psi_2$,
we can find $y \in \jacobi^{<,c}_2(H)$ such that $\psi_2(y) = \{N\} \in Y_2\cyl/Y_3$.
We set $x:=\varphi^{-1} s^{-1}(y) \in \jacobi^{c}_2(L^\pm)$,
where the isomorphisms $s$ and $\varphi$ are recalled in \S \ref{subsec:Jacobi}.
Lemma \ref{lem:tau2} implies that the following diagram is commutative:
$$
\xymatrix{
\jacobi^{<,c}_2(H) \ar[rr]^-{\psi_2} & &  \frac{Y_2 \cyl}{Y_3} \ar@{->>}[d]^-{\tau_2}\\
\jacobi_2^c(L^\pm) \ar[u]^-{s \varphi}_-\simeq \ar@{->>}[r]_-{-p_{2,0}} & 
\jacobi_{2,0}^c(L^\pm) \ar[r]_-\simeq & \frac{S^2 \Lambda^2 H}{\Lambda^4H}
}
$$
(Here, the bottom isomorphism is defined by $\hn{a}{b}{c}{d} \longmapsto \left((a\wedge b) \leftrightarrow (c\wedge d)\right)$
where $a,b,c,d \in L^\pm$ are considered as elements of $H$ by (\ref{eq:L+-_H}).)
By assumption, we have $\tau_2 \psi_2(y)=0$  and we deduce that $p_{2,0}(x)=0$,
\ie $x$ only consists of looped Jacobi diagrams. 
The same  deduction applies to $y=s \varphi(x)$ and the conclusion follows.
\end{proof}

\subsection{The LMO homomorphism}\label{subsec:LMO}

We briefly review the LMO homomorphism, its construction and main properties.
For this purpose, we fix a system of meridians and parallels  $(\alpha_i,\beta_i)_{i=1}^g$  
on the surface $\Sigma$   as shown in Figure \ref{fig:basis}.  
Then one can turn any homology cylinder $M$ over $\Sigma$ into a homology $3$-ball $B$
by gluing, for each $i\in \{1,\dots,g\}$, a $2$-handle to the surface $\partial_-M$ along the curve $m_-(\alpha_i)$, 
and a $2$-handle to the surface $\partial_+M$ along the curve $m_+(\beta_i)$.  
The cores of these $2$-handles define a $(2g)$-component framed tangle $\gamma$ in the homology $3$-ball $B$.
By taking the Kontsevich--LMO invariant of the pair $(B,\gamma)$  and after an appropriate normalization, 
one can thus associate to $M\in \cyl$ an element $\widetilde{Z}^Y(M)$ of (the degree completion of) $\jacobi(L^\pm_\Q)$.
See \cite{CHM} where the target is denoted by $\jacobi^Y\left(\set{g}^+ \cup \set{g}^-\right)$.
The important point is that the colors $1^-,\dots,g^-$ refer to the curves $\alpha_1,\dots,\alpha_g$
while $1^+,\dots,g^+$ refer to  $\beta_1,\dots,\beta_g$, 
so that the definition of $\widetilde{Z}^Y$ depends on the choice of $(\alpha,\beta)$.
(The definition is also dependent on the choice of an associator for the Kontsevich integral.)
The space  $\jacobi(L^\pm_\Q)$, equipped with the multiplication
\begin{equation}\label{eq:star_product}
D \star E :=
\left(\! \!  \begin{array}{c}
\hbox{sum of all ways of gluing \emph{some} of the $i^+$-colored vertices of $D$}\\
\hbox{to \emph{some} of the $i^-$-colored vertices of $E$, for all $i=1,\dots,g$}
\end{array}\! \!  \right) 
\end{equation}
of $L^\pm_\Q$-colored Jacobi diagrams $D$ and $E$, is an associative $\Q$-algebra.
The aforementioned normalization is done in such a way that
$$
\widetilde{Z}^Y: \cyl \longrightarrow \jacobi(L^\pm_\Q)
$$
is a monoid homomorphism. 

The push-out of the multiplication $\star$ by the isomorphism $\kappa$ defined at (\ref{eq:kappa}) is still denoted by $\star$.
For any $H_\Q$-colored Jacobi diagrams $D$ and $E$, the multiplication $\star$ on $\jacobi(H_\Q)$
is  explicitly given by the formula
$$
D \star E := \sum_{\substack{V' \subset V,\ W' \subset W\\ 
\beta\ :\ V' \stackrel{\simeq}{\longrightarrow} W'}}\
\frac{1}{2^{|V'|}} \cdot \prod_{v\in V'} \omega\big(\hbox{color}(v),\hbox{color}(\beta(v))\big)\ \cdot (D \cup_\beta E)
$$
where $V$ and $W$ denote the sets of external vertices of $D$ and $E$ respectively,
and where the sum is taken over all ways $\beta$ of identifying a part $V'$ of $V$ with a part $W'$ of $W$.
Then  the \emph{LMO homomorphism} is defined in  \cite{HM_SJD} as the composition
$$
\xymatrix{
 \cyl\ar@{-->}@/_1.2pc/[rr]_-Z \ar[r]^-{\widetilde{Z}^Y} &\jacobi(L^\pm_\Q) \ar[r]^-\kappa_-\simeq &\jacobi(H_\Q). 
}
$$

The LMO homomorphism is universal  
among rational-valued finite-type invariants of homology cylinders \cite{CHM,HM_SJD}.
In terms of the surgery map introduced in \S \ref{subsec:surgery_map},
this universal property of $Z$ amounts to the following commutative diagram: 
\begin{equation}
\label{eq:univLMO}
\xymatrix{
\jacobi^{<,c}(H_\Q) \ar@{->>}[rr]^-{\psi\otimes \Q}\ar[rrd]_-{\chi^{-1}}^-\simeq
&& \left(\Gr^Y \cyl\right)\otimes \Q \ar[d]^-{\Gr Z} \\
&& \jacobi^{c}(H_\Q).
}
\end{equation}
It follows that the maps $\psi \otimes \Q$ and $\Gr Z$ are isomorphisms.

The LMO homomorphism is compatible with stabilizations of the surface $\Sigma$.
More precisely, assume that the surface $\Sigma$ has been stabilized to a surface $\Sigma^s$ of genus $g^s$
(as shown on Figure \ref{fig:stabilization})
and that the system of meridians and parallels  $(\alpha_i,\beta_i)_{i=1}^{g^s}$ chosen on $\Sigma^s$
extends that of $\Sigma$. We denote $H^s:=H_1(\Sigma^s)$ which contains a copy of $H$.
Then, the following diagram is commutative:
$$
\xymatrix{
\cyl(\Sigma) \ar[r]  \ar[d]_-Z & \cyl(\Sigma^s) \ar[d]^-Z \\
\jacobi(H_\Q) \ar[r] & \jacobi(H_\Q^s)
}
$$

\subsection{Johnson, Alexander and Casson from LMO}\label{subsec:invariantsLMO}

The degree $1$ part of the LMO homomorphism is equivalent to the first Johnson homomorphism \cite{CHM}. 
More precisely, we have the commutative diagram\\[-0.5cm]
\begin{equation}
\label{eq:tau1_LMO}  
\xymatrix{
\cyl/Y_2 \ar[d]_{\tau_1} \ar[r]^{Z_1} & \jacobi^{c}_1(H_{\Q}) \ar[d]^-\simeq \\
\Lambda^3 H  \ar@{>->}[r] &  \Lambda^3 H_\Q
}
\end{equation}
where the isomorphism on the right side is given by $\yn{a}{c}{b} \longmapsto a \wedge b \wedge c$.

Similarly, the second Johnson homomorphism corresponds to the ``tree-reduction'' 
of the degree $2$ part of the LMO homomorphism.

\begin{lemma}[See \cite{CHM}] \label{lem:tau2_LMO}
There is a commutative diagram
\[  
\xymatrix{
\Jcob/Y_{3} \ar[d]_{\tau_2} \ar[r]^{Z_2} & \jacobi^{c}_2(H_{\Q}) \ar@{->>}[d]^{p_{2,0}}\\
\frac{\left(\Lambda^2 H \otimes \Lambda^2H\right)^{\mathfrak{S}_2}}{\Lambda^4H}  \ar@{>->}[r] & \jacobi^c_{2,0}(H_\Q)
} 
\]
where the bottom monomorphism is given by 
$\left((a\wedge b) \leftrightarrow (c\wedge d)\right) \mapsto \hn{a}{b}{c}{d}$.
\end{lemma}

Next, the quadratic part of the relative RT torsion can be extracted 
from the degree $2$ part of the LMO homomorphism in the following manner.

\begin{lemma}\label{lem:alpha_LMO}
There is a commutative diagram 
\[  
\xymatrix{
\Jcob/Y_{3} \ar[d]_{-\frac{\alpha}{2}} \ar[r]^{Z_2} & \jacobi^{c}_2(H_{\Q}) \ar@{->>}[d]^{p_{2,1}}\\
\frac{1}{2}S^2H \ar@{>->}[r] & \jacobi^c_{2,1}(H_ \Q),
}
\]
where the monomorphism on the bottom is defined by $a\cdot b  \mapsto \phin{a}{b}$. 
\end{lemma}

\begin{proof}
The quotient group  $\frac{\Jcob/Y_{3}}{Y_2\cyl/Y_{3}} \simeq \Jcob/Y_{2}$ 
consists only of order $2$ elements, as follows from \cite{MM}.  
So, for any element $\{M\}$ of $\Jcob/Y_{3}$, we have that $\{M\circ M\}=\{M\}^2$ lies in $Y_2\cyl/Y_{3}$.
We have  that $\alpha(M^2)=2\alpha(M)$ (by Proposition \ref{prop:alpha_properties})
and $Z_2(M^2)=2Z_2(M)$  (since we have $Z(M^2)= Z(M) \star Z(M)$ and $Z_1(M)=0$).
Therefore it suffices to establish the commutativity of the diagram on the subgroup $Y_2\cyl/Y_{3}$ of $\Jcob/Y_{3}$.  

Since the map $\psi_2:\jacobi^{<,c}_2(H) \to Y_2\cyl/Y_{3}$ is surjective, it is enough to check that
$\alpha \psi_2(D)= -2 p_{2,1} Z_2 \psi_2(D)$ for any generator $D$ of $\jacobi^{<,c}_{2}(H)$.  
The generators can be of three types, namely H graphs, $\Phi$ graphs or $\Theta$ graphs.
Let us first compute the composition $\alpha \circ \psi_2$ on these three types of generators.  
By Proposition \ref{prop:alpha_properties}, we have 
\begin{equation} \label{eq:phi}
\alpha  \psi_2\left(\phio{h}{h'}\right)= -2h  h'\in S^2H. 
\end{equation}
The same proposition also implies that 
\begin{equation} \label{eq:theta}
\alpha \psi_2\left(\thetagraph\right)=0\in S^2H. 
\end{equation}
\begin{claim} For any $h,i,j,k \in H$, we have
\begin{equation}\label{eq:H}
\alpha  \psi_2\big(\ho{h}{i}{j}{k}\big)= 
\omega(h,j) \cdot  ik- \omega(h,k) \cdot  ij- \omega(i,j) \cdot  hk+ \omega(i,k) \cdot  hj\in S^2H. 
\end{equation}
\end{claim}

\noindent
In order to prove this claim, we use the decomposition of $\jacobi^{c}_{2}(H_\Q)$ 
into irreducible $\Sp(H_\Q)$-modules. 
This can be found in \cite[\S 5]{HM_SJD}, whose notation we follow:
\begin{eqnarray*}
\jacobi^{c}_{2}(H_\Q) & = & \jacobi^c_{2,0}(H)\oplus \jacobi^c_{2,1}(H)\oplus \jacobi^c_{2,2}(H)\\
 & = & \left(\Gamma_0\oplus \Gamma_{\omega_2}\oplus 
 \Gamma_{2\omega_2}\right)\oplus \left(\Gamma_{2\omega_1}\right)\oplus \left( \Gamma_0\right).
\end{eqnarray*}
(Note that this formula was obtained in \cite{HM_SJD} for $g\ge 3$ only.
But, since both invariants $\alpha$ and $Z$ are well-behaved with respect to stabilization,
we can assume without loss of generality that the surface $\Sigma$ has arbitrary high genus $g$.) 
Since the composition 
\[ 
\xymatrix{
\jacobi^{c}_{2}(H_\Q)\ar[r]_{\simeq}^{\chi} & \jacobi^{<,c}_{2}(H_\Q)\ar[r]_{\simeq}^{\psi_2} 
& \frac{Y_2\cyl}{Y_{3}}\otimes \Q \ar@{->}[r]^-\alpha & S^2 H_\Q\simeq \Gamma_{2\omega_1} 
}
\]
is $\Sp(H_\Q)$-equivariant, we deduce that it vanishes on the subspace  $\jacobi^c_{2,0}(H) \oplus \jacobi^c_{2,2}(H)$.
Thus, using the explicit formula for $\chi^{-1}$  given in \cite[Proposition 3.1]{HM_SJD},
we obtain that
\begin{eqnarray*}
&&\alpha  \psi_2\big(\ho{h}{i}{j}{k}\big)\\
& = & \alpha \psi_2 \chi\left( \chi^{-1} \big(\ho{h}{i}{j}{k}\big)\right) \\
& = & \alpha  \psi_2 \chi\left( p_{2,1} \chi^{-1} \big(\ho{h}{i}{j}{k}\big)\right) \\
 & = & \frac{1}{2} \alpha  \psi_2  \chi
 \left( -\omega(h,j) \phin{i}{k}+\omega(h,k) \phin{i}{j}+\omega(i,j) \phin{h}{k}-\omega(i,k) \phin{h}{j} \right).
\end{eqnarray*}
Besides we deduce from (\ref{eq:phi}) that
$$ 
\alpha \psi_2 \chi \left(\phin{h}{h'}\right)
= \alpha  \psi_2\left(\frac{1}{2}\phio{h}{h'}+\frac{1}{2}\phio{h'}{h}\right)= -2 h h'. 
$$
and the claim follows.

Now, in order to complete the proof of the lemma, we use (\ref{eq:univLMO})
to deduce that the composition $p_{2,1} Z_2 \psi_2: \jacobi^{<,c}_2(H) \to \jacobi_{2,1}^c(H_\Q)$ coincides with $p_{2,1} \chi^{-1}$.
Therefore, the formula for $\chi^{-1}$  in \cite[Proposition 3.1]{HM_SJD} gives
$$
p_{2,1} Z_2  \psi_2\left( \thetagraph \right) = 0, \quad
p_{2,1} Z_2  \psi_2\big( \phio{h}{h'} \big) = \phin{h}{h'}\vspace{-0.2cm}
$$
and  
$$
p_{2,1} Z_2  \psi_2\big(\ho{h}{i}{j}{k}\big) =
\frac{1}{2}\big(  -\omega(h,j) \phin{i}{k}+\omega(h,k) \phin{i}{j}+\omega(i,j) \phin{h}{k}-\omega(i,k) \phin{h}{j}\big).
$$
We conclude thanks to (\ref{eq:phi}),  (\ref{eq:theta}) and (\ref{eq:H}) .
\end{proof}

The relationship between the LMO homomorphism $Z$ and the Casson invariant $\lambda_j$ is more subtle.
First of all, the Heegaard embedding $j:\Sigma \to S^3$ necessary to the definition of $\lambda_j$ (see \S \ref{subsec:Casson})
is chosen compatibly with the system of meridians and parallels $(\alpha,\beta)$ 
on which the construction of $Z$ depends (see \S \ref{subsec:LMO}).
More precisely, we assume that the curves $j(\alpha_i)\subset F_g$ bound disks
in the lower handlebody of the Heegaard splitting of $S^3$, 
while the curves $j(\beta_i) \subset F_g$ bound disks in the upper handlebody.
The following is observed in \cite{CHM} from the relation 
between the Casson invariant of homology $3$-spheres and the LMO invariant \cite{LMO}:
\begin{equation}\label{eq:cassonLMO}
\xymatrix{
\Jcob/Y_{3} \ar[d]_{\lambda_j} \ar[r]^{\widetilde{Z}^Y_2} & \jacobi^{c}_2(L^\pm_\Q) \ar@{->>}[d]^{p_{2,2}}\\
\Z \ar[r]_(.35){\cdot\frac{1}{2}\thetagraph} &\jacobi^{c}_{2,2}(H_\Q)  
} 
\end{equation}
Since $\widetilde{Z}^Y_2$ is obtained from $Z_2$  by post-composing 
with the map $\kappa^{-1}=\varphi^{-1} s \chi$, 
we see that the $\Theta$ part of $\widetilde{Z}^Y_2$ 
is equal to minus the $\Theta$ part of $Z$ plus something derived from its H part and its $\Phi$ part.
In particular, we deduce the following from Lemma \ref{lem:tau2_LMO} and Lemma \ref{lem:alpha_LMO}.

\begin{lemma}\label{lem:Casson_LMO}
Let $M\in \Jcob$ be such that $\tau_2(M)=0$ and $\alpha(M)=0$.
Then, we have
$$
p_{2,2} Z_2(M) = -\frac{\lambda_j(M)}{2} \cdot \thetagraph.
$$
\end{lemma}

\noindent
The  connection between $Z_2$ and $\lambda_j$ is further investigated in \S \ref{sec:core_Casson}, 
where we extend the \emph{core} of the Casson invariant to homology cylinders.

\section{Characterization of the $Y_3$-equivalence for homology cylinders}\label{sec:Y3}

In this section we prove the characterization of the $Y_3$-equivalence relation on $\cyl$
(Theorem~A), we give a diagrammatic description of the group $\cyl/Y_3$ 
and we deduce certain properties of this group. 

\subsection{Proof of Theorem~A}

We now show that, given two homology cylinders $M$ and $M'$ over $\Sigma$, the following assertions are equivalent:
\begin{enumerate}
\item[(a)] $M$ and $M'$ are $Y_3$-equivalent;
\item[(b)] $M$ and $M'$ are not distinguished by any Goussarov--Habiro finite-type invariants of degree at most $2$;
\item[(c)] $M$ and $M'$ share the same invariants $\rho_3, \alpha$ and $\lambda$;
\item[(d)] The LMO homomorphism $Z$ agrees on $M$ and $M'$ up to degree $2$.
\end{enumerate} 
The implication (a)$\Rightarrow$(b) follows from Lemma \ref{lem:Y_FTI}.
The implication (b)$\Rightarrow$(d) is guaranteed by the fact that
the  degree $i$ part of the LMO homomorphism is a finite-type invariant of degree $i$ \cite{CHM}.
Thus it remains to prove that  (a), (c) and (d) are equivalent.

The implication (a)$\Rightarrow$(c) says
that $\rho_3$, $\lambda$ and $\alpha$ are invariant under $Y_3$-equivalence.
We have seen in  \S \ref{subsec:Alexander} and \S \ref{subsec:Casson}
that $\alpha$ and $\lambda$ are finite-type invariants of degree $2$:
we deduce from Lemma \ref{lem:Y_FTI} that $\alpha$ and $\lambda$ are $Y_3$-invariant.
The map $\rho_3: \cyl \to \Aut(\pi/\Gamma_4 \pi)$ is $J_3$-invariant (by Lemma \ref{lem:rho_k} below) so that it is also $Y_3$-invariant.

We now prove (c)$\Rightarrow$(d).
Two homology cylinders $M$ and $M'$ satisfying (c) 
cannot be distinguished by the first Johnson homomorphism, nor by the Birman--Craggs homomorphism according to  Lemma \ref{lem:beta}.
We deduce from the characterization of the $Y_2$-equivalence given in \cite{MM} that $M\stackrel{Y_2}{\sim}  M'$.   
The monoid $\cyl/Y_3$ being a group \cite{Goussarov,Habiro}, 
one can find a  $D\in Y_2\cyl$ such that  
\begin{equation}
\label{eq:D}
M\stackrel{Y_3}{\sim} D\circ M'.
\end{equation}
Since $Z$ is a monoid homomorphism,
we deduce that $Z_{\leq 2} (M)$ is the degree $\leq 2$ truncation of $Z_{\leq 2}(D) \star  Z_{\leq 2}(M')$, 
so it suffices to show that $Z_{\leq 2}(D)= \varnothing$ (the empty diagram).
But this is equivalent to $Z_2(D)=0$ given that $D \stackrel{Y_2}{\sim} \Sigma \times I$.
The decomposition (\ref{eq:D}) and the assumption (c) have three consequences:
\begin{enumerate}
\item $\rho_3(D)=1$,  which implies that $\tau_2(D)=0$,
\item $\alpha(D)=0$,
\item and $\lambda(D)=0$ by formula (\ref{eq:sum_formula}).
\end{enumerate}
Thanks to Lemma \ref{lem:tau2_LMO}, Lemma \ref{lem:alpha_LMO} and Lemma \ref{lem:Casson_LMO}, 
we deduce that $Z_2(D)=0$  as desired.

We now prove (d)$\Rightarrow$(a). Let $M,M' \in \cyl$ be such that $Z_{\leq 2}(M)=Z_{\leq 2}(M')$.
Again, since $\cyl/Y_3$ is a group, one can find a $D\in \cyl$ such that we have the decomposition (\ref{eq:D}). 
We deduce from (\ref{eq:D}) 
and the assumption (d) that $Z_1(D)=0$ 
which, according to  (\ref{eq:tau1_LMO}), is equivalent to $\tau_1(D)=0$.
Next, we deduce from (\ref{eq:D}) 
and the assumption (d)
that $Z_2(D)=0$ which, by Lemma \ref{lem:tau2_LMO}, Lemma \ref{lem:alpha_LMO} and Lemma \ref{lem:Casson_LMO},
imply that $\tau_2(D)=0$, $\alpha(D)=0$ and $\lambda(D)=0$. 
So we have $\beta(D)=0$ by Lemma \ref{lem:beta} and, since the $Y_2$-equivalence is classified by $(\tau_1,\beta)$,
we obtain that $D \stackrel{Y_2}{\sim} \Sigma \times I$.
The universal property of the LMO homomorphism (\ref{eq:univLMO}) gives, in degree $2$,
the following commutative diagram\\[-0.5cm]
$$
\xymatrix{
\jacobi^{<,c}_2(H) \ar@{->>}[r]^{\psi_2} \ar[d]_{- \otimes \Q} &  Y_2\cyl/Y_{3} \ar[d]^{Z_{2}}\\
\jacobi^{<,c}_{2}(H_\Q) &  \ar[l]^{\chi}_-\simeq \jacobi^{c}_{2}(H_\Q).
}
$$
The left vertical arrow is injective by Lemma \ref{lem:no_torsion}, which proves that $Z_{2}$ is injective.  
Hence $Z_{2}(D)=0$ implies that $D\stackrel{Y_3}{\sim} (\Sigma\times I)$ 
which, combined with the decomposition (\ref{eq:D}), 
show that $M$ and $M'$ satisfy (a).  

This concludes the proof of Theorem A and shows the following. 

\begin{corollary}\label{cor:varphi2}
The abelian group $Y_2\cyl/Y_{3}$ is torsion-free, and the surgery map 
$$
\psi_2: \jacobi^{<,c}_2(H) \rightarrow  Y_2\cyl/Y_{3} 
$$ 
is an isomorphism.   
\end{corollary}

\subsection{Diagrammatic description of $\Jcob/Y_{3}$} \label{subsec:description1} 

We aim at giving in this section a diagrammatic description of the group $\cyl/Y_{3}$ using the previous results.
One way to do that would be to start from the central extension
\begin{equation}
\label{eq:not_the_best}  
\xymatrix{0 \ar[r] & Y_2\cyl/Y_3 \ar[r] & \cyl/Y_3 \ar[r] & \cyl/Y_2 \ar[r] & 1} 
\end{equation}
and to use the diagrammatic descriptions of $\cyl/Y_2$ and $Y_2\cyl/Y_3$ 
which are given by \cite{MM} and Corollary \ref{cor:varphi2} respectively.
However this method is not the most convenient one since $\cyl/Y_2$ is \emph{not} torsion-free.
Instead we shall proceed in the following way.

\begin{lemma}\label{lem:Habiro_sec}
We have a central extension of groups
\begin{equation}
\label{eq:Habiro_sec}
 \xymatrix{0 \ar[r] & \Jcob/Y_3 \ar[r] & \cyl/Y_3 \ar[r]^{\tau_1} & \Lambda^3 H \ar[r] & 1} 
\end{equation}
where $\Jcob = \cob[2]$ denotes the second term of the Johnson filtration of $\cob$.
\end{lemma}

\begin{proof}
As observed in the proof of Lemma \ref{lem:d^2beta},
it follows from \cite{MM} that $\Jcob/Y_3$ is generated by elements of the form $(\Sigma\times I)_G$, 
where $G$ is either a degree $2$ graph clasper or a $Y$-graph with (at least) one special leaf.  
In the first case, $(\Sigma\times I)_G$ is a central element in $\cyl/Y_3$ by Lemma \ref{lem:crossingchange}. 
The same holds in the second case by Lemma \ref{lem:crossingchange} and Lemma \ref{lem:special}. 
\end{proof}

The central extension (\ref{eq:Habiro_sec}) is, up to equivalence,
uniquely determined by its characteristic class in $H^2(\Lambda^3H;\Jcob/Y_3)$.  
Since $\Lambda^3H$ \emph{is} torsion-free, we have by the universal coefficient theorem that 
\begin{equation}\label{eq:iso2cocycle}
H^2(\Lambda^3H;\Jcob/Y_3)\simeq \textrm{Hom}(H_2(\Lambda^3H),\Jcob/Y_3)\simeq  \textrm{Hom}(\Lambda^2\Lambda^3H,\Jcob/Y_3). 
\end{equation}
Thus, in order to describe the (isomorphism type of) the group $\cyl/Y_3$, we shall proceed in two steps:
first, we shall give in this subsection a diagrammatic description of the  group $\Jcob/Y_3$
and, second, we shall give in the next subsection a diagrammatic description 
of the characteristic class of (\ref{eq:Habiro_sec}) in  $\textrm{Hom}(\Lambda^2\Lambda^3H,\Jcob/Y_3)$.  

Our diagrammatic description of $\Jcob/Y_{3}$ is derived from a group homomorphism 
$$ 
\psi_{[2]}: \jacobi^{<,c}_2(H)\oplus \Z\!\cdot\!(H\times H) \oplus \Z\!\cdot\! H\oplus \Z\longrightarrow \Jcob/Y_{3},
$$
where $\Z\!\cdot\!(H\times H)$ and $\Z\! \cdot\! H$ denote the free abelian groups generated 
by the sets $H\times H$ and $H$ respectively. We define $\psi_{[2]}$ in the following way:
\begin{itemize}
 \item For all $D\in \jacobi^{<,c}_2(H)$,  we set $\psi_{[2]}(D):=\psi_2(D)$ where $\psi_2$ is 
       Habiro's map as defined in \S \ref{subsec:surgery_map}.
 \item For all $(h,h')\in H\times H$, we set $\psi_{[2]}(h,h')$ to be the $Y_3$-equivalence class of 
       $(\Sigma \times I)_{Y_{h,h'}}$
       where $Y_{h,h'}$ is a $Y$-graph  obtained as follows. 
       Consider the oriented surface $S$ consisting of a disk, connected by three bands to three annuli,
       whose cores are oriented as in Figure \ref{fig:leaf_orientation}.
       Embed $S$ into the interior of $(\Sigma \times I)$ so as to obtain a $Y$-graph with one special leaf 
       and two  other leaves satisfying the following: they should  have framing number zero, 
       one should represent $h\in H$ and lie in $\Sigma \times [-1,0]$
       while the other one should represent $h'\in H$ and lie in $\Sigma \times [0,1]$.  
 \item For all $h\in H$, we set $\psi_{[2]}(h)$ to be the $Y_3$-equivalence class of $(\Sigma \times I)_{Y_h}$
       where $Y_h$ is a $Y$-graph obtained as follows. We take the same surface $S$ as before
       but we embed it into the interior of $(\Sigma \times I)$ so as to obtain a $Y$-graph with two special leaves, 	
	   and one leaf which is required to represent $h\in H$ and to have framing number zero. 
 \item We set $\psi_{[2]}(1)$ to be the $Y_3$-equivalence class of $(\Sigma \times I)_{Y_s}$ 
       where $Y_{s}$ is a $Y$-graph with three special leaves.   
\end{itemize}
\noindent 
(We refer the reader to Appendix \ref{subsec:framing_numbers} for 
the definition of framing numbers in $\Sigma\times I$.)

\begin{lemma}
 The map $\psi_{[2]}$ is a well-defined homomorphism, and it is surjective.  
\end{lemma}

\begin{proof}
By \S \ref{subsec:surgery_map},
the fact that $\psi_{[2]}$ is well-defined needs only to be checked on the summand
$ \Z\!\cdot\!(H\times H)\oplus \Z\!\cdot\! H\oplus \Z$.
The independence on the choice of the disk and bands follows from 
Lemma \ref{lem:as_special} and Lemma \ref{lem:slide_special}, respectively.
We now check the independence on the choice of the leaves. 
Let $K$ and $K'$ be two possible choices for an oriented 
leaf of a $Y$-graph, representing the same element in $H$.  
Then $K$ and $K'$ cobound an embedded oriented 
surface $F$ of genus $g(F)$ in $(\Sigma\times I)$. 
Furthermore, since the framing numbers of $K$ and $K'$ in $(\Sigma \times I)$ are equal, 
we can find such a surface $F$ such that $K$ and $K'$ are $0$-framed with respect to $F$.
(This follows, for instance, from Lemma \ref{lem:framed_connect}.)
By Lemma \ref{lem:slide_special}, we can assume that $F$ does not intersect any edge of the $Y$-graph. 
We can also freely assume that $F$ does not intersect any other leaf of the $Y$-graph, 
since they are either special leaves, 
or lying in a different ``horizontal layer'' of $(\Sigma\times I)$.  
So, if we have $g(F)=0$, then $K$ and $K'$ are isotopic as  framed knots and we are done. 
Otherwise, we can decompose $K$ as a framed connected sum of $K'$ and $g(F)$ framed knots,
each bounding a genus $1$ surface disjoint from the $Y$-graph and 
being $0$-framed with respect to it.  
The result then follows from Corollary \ref{cor:surf}.  

The surjectivity of $\psi_{[2]}$ is proved by refining 
the argument used at the beginning of the proof of Lemma \ref{lem:d^2beta}.
Let $M\in \Jcob$. Using the notation (\ref{eq:affine_functions}), the quadratic function $\beta(M)\in B_{\leq 2}$ can be decomposed as
$$
\beta(M) = \varepsilon \cdot \overline{1} + \sum_{i=1}^m \overline{g_i} + \sum_{j=1}^n \overline{h_j} \cdot  \overline{h'_j}
$$
where $\varepsilon \in \{0,1\}$, $g_1,\dots, g_m \in H$ and $h_1,h'_1,\dots, h_n,h'_n \in H$ 
for some positive integers $m,n$.
Let $Y_s$, $Y_g$ (for $g\in H$) and $Y_{h,h'}$ (for $h,h'\in H$) be the $Y$-graphs 
with special leaves described in the definition of $\psi_{[2]}$. Then, using Lemma \ref{lem:Y_beta}, we see that
$$
\beta(M) = \varepsilon\cdot \beta\left((\Sigma \times I)_{Y_s}\right) +  
\sum_{i=1}^m \beta\left((\Sigma \times I)_{Y_{g_i}} \right)  + \sum_{j=1}^n   \beta\left((\Sigma \times I)_{Y_{h_j,h'_j}}  \right).
$$
Since the $Y_2$-equivalence is classified by the couple $(\tau_1,\beta)$, we deduce that 
$$
M \overset{Y_2}{\sim} {(\Sigma \times I)_{Y_s}}^\varepsilon  \circ \prod_{i=1}^m (\Sigma \times I)_{Y_{g_i}} 
\circ  \prod_{j=1}^n (\Sigma \times I)_{Y_{h_j,h'_j}}. 
$$
Therefore, by clasper calculus, there exists a $D \in \cyl$ such that $D\overset{Y_2}{\sim} (\Sigma \times I)$ and 
$$
M \overset{Y_3}{\sim} D \circ  {(\Sigma \times I)_{Y_s}}^\varepsilon  \circ \prod_{i=1}^m (\Sigma \times I)_{Y_{g_i}} 
\circ  \prod_{j=1}^n (\Sigma \times I)_{Y_{h_j,h'_j}}.
$$
We conclude using the fact that the restriction of $\psi_{[2]}$ to the summand $\jacobi^{<,c}_2(H)$,
namely $\psi_2$, is surjective onto the group $Y_2\cyl/Y_3$.
\end{proof}

Next we set
$$
\jacobi_{[2]}^{<,c}(H) := 
\frac{\jacobi^{<,c}_2(H)\oplus  \Z\!\cdot\!(H\times H)\oplus \Z\!\cdot\!H \oplus \Z}{(G_0,G_1,G_2,G_3, D_1,D_2,D_3)}, 
$$
where the relations $(G_0,G_1,G_2,G_3)$ and  $(D_1,D_2,D_3)$ are defined as follows:
\begin{enumerate}
 \item[($G_0$)]\quad $(h,h) - (h)$ \ for all $h\in H$,
 \item[($G_1$)]\quad $2\cdot (h,k) +\hob{h}{h}{k}{k} + \phio{h}{k}$ \ for all $h,k\in H$,
 \item[($G_2$)]\quad $2\cdot (h) + \phio{h}{h}$ \ for all $h\in H$,
 \item[($G_3$)]\quad $2\cdot 1+\thetagraph$,
 \item[($D_1$)]\quad $(h+h',k) - (h,k) - (h',k)+ \omega(h,h')\cdot (k) + \hob{h}{h'}{k}{k}$ \ for all $h,h',k\in H$, 
 \item[($D_2$)]\quad $(h,k+k') - (h,k) - (h,k') + \omega(k,k')\cdot(h) + \hob{h}{h}{k}{k'}$ \ for all $h,k,k'\in H$,
 \item[($D_3$)]\quad $(h+h')  -  (h) - (h') + \omega(h,h') \cdot $1$ + \phio{h}{h'}$ \ for all $h,h' \in H$.
\end{enumerate}
Here  the generator of the summand $\Z$ is denoted by $1$,
the generators of the summand $ \Z\!\cdot\! H$ are denoted by $(h)$ with $h\in H$, 
and the generators of the summand $ \Z\! \cdot\!(H\times H)$ are denoted by $(h,k)$ with $h,k\in H$.
Observe that, thanks to the relation ($G_0$), we could get rid of the summand $ \Z\!\cdot\! H$.
Besides, ($G_2$) is a consequence of ($G_0$) and $(G_1)$. 
Here is yet another relation in $\jacobi_{[2]}^{<,c}(H)$:

\begin{lemma}
For all $h,h'\in H$, we have
\begin{equation}\label{eq:symmetry_defect}
(h,h') - (h',h) - \omega(h,h') \cdot \phio{h'}{h} - \frac{\omega(h,h')(\omega(h,h')-1)}{2}\cdot \thetagraph = 0 \in \jacobi_{[2]}^{<,c}(H).
\end{equation}
\end{lemma}

\begin{proof}
We set $k := h +h'$. Using $(G_0)$, $(D_1)$ and $(D_2)$, we get
\begin{eqnarray*}
(k) \quad = \quad (k,k) &= &  (h,k) + (h',k) - \omega(h,h')\cdot (k) - \hob{h}{h'}{k}{k}\\
&=&  \left((h,h') + (h,h) - \omega(h',h)\cdot(h) -  \hob{h}{h}{h'}{h}\right) \\
&& + \left(   (h',h') + (h',h) - \omega(h',h)\cdot(h') - \hob{h'}{h'}{h'}{h} \right)\\
&& - \omega(h,h')\cdot (k) - \hob{h}{h'}{k}{k}.
\end{eqnarray*}
It follows from the IHX and STU-like relations that, for all $a,b,c\in H$,  
$$
\hob{a}{b}{b}{c}  = 0 \ \in \jacobi_{[2]}^{<,c}(H).
$$
We deduce that
\begin{eqnarray*}
&&(1+\omega(h,h'))\cdot (k) - (1+\omega(h,h')) \cdot (h)  - (1+\omega(h,h')) \cdot  (h') \\
 &=& (h,h') + (h',h)   -  \hob{h}{h}{h'}{h}  - \hob{h}{h'}{k}{k}\\
 &=& (h,h') + (h',h)   +   \hob{h}{h}{h'}{h'} \\
 &=& (h',h)-(h,h') - \phio{h}{h'},
\end{eqnarray*}
where the last equality follows from relation $(G_1)$.  
Using now $(D_3)$, we get
$$
 (h',h) - (h,h') - \phio{h}{h'} = -(1+\omega(h,h')) \omega(h,h') \cdot 1 - (1+\omega(h,h') ) \cdot \phio{h}{h'},
$$
and, using $(G_3)$, we obtain
$$
(h',h) - (h,h') = 
\frac{(1+\omega(h,h')) \omega(h,h')}{2} \cdot \thetagraph - \omega(h,h')\cdot \phio{h}{h'}.                
$$
The STU-like relation allows us to conclude. 
\end{proof}

\begin{theorem}\label{thm:iso_KC/Y3}
The map $\psi_{[2]}$ factorizes to an isomorphism 
$$ 
\psi_{[2]}: \jacobi_{[2]}^{<,c}(H) \overset{\simeq}{\longrightarrow} \Jcob/Y_{3} 
$$ 
and the group $\jacobi_{[2]}^{<,c}(H)$ is a free abelian group with the same rank as $\jacobi^{<,c}_2(H)$.
\end{theorem}

Before proving this theorem, we shall draw two of its consequences.
First, after one has chosen a basis of the free abelian group $H$, 
one can derive from this diagrammatic description a presentation of the abelian group $\Jcob/Y_{3}$. 
Second, Theorem \ref{thm:iso_KC/Y3} implies the following.

\begin{corollary}
We have the following commutative diagram of abelian groups, whose rows are short exact sequences:
$$
\xymatrix{
0 \ar[r] & Y_2 \cyl/Y_3 \ar[r] & \Jcob/Y_3 \ar[r]^\beta & B_{\leq 2} \ar[r] & 0\\
0 \ar[r] & \jacobi^{<,c}_2(H) \ar[r] \ar[u]^-{\psi_2}_-\simeq & \jacobi_{[2]}^{<,c}(H) \ar[u]^-{\psi_{[2]}}_-\simeq  \ar[r]^b
& B_{\leq 2} \ar@{=}[u] \ar[r] & 0.
}
$$
Here $B_{\leq 2}$ is the space of polynomial functions $\Spin(\Sigma)\to \Z_2$ of degree $\leq 2$
and, using the notation (\ref{eq:affine_functions}), we define the homomorphism $b$ as follows:
$b$ is trivial on $\jacobi^{<,c}_2(H)$, $b$ sends $1\in \Z$ to the constant function $\overline{1}$,
$(h)\in \Z\!\cdot\! H$ to the affine function $\overline{h}$,
and  $(h,k)\in  \Z\!\cdot\!(H\times H)$ to the quadratic function $ \overline{h} \cdot \overline{k}$.
\end{corollary}

\begin{proof}
The fact that $b$ is well-defined is easily checked using the following formula:
$$
\forall h,k \in H, \quad
\overline{h+k} = \overline{h} + \overline{k} + \omega(h,k) \cdot \overline{1} \ \in B_{\leq 1}.
$$
The commutativity of  the diagram follows from the definition of $\psi_{[2]}$  and from Lemma \ref{lem:Y_beta}.
The top sequence is exact according to \cite{MM}.
Since the vertical maps are isomorphisms (by Corollary \ref{cor:varphi2} and Theorem  \ref{thm:iso_KC/Y3}),
the bottom sequence is exact too.
\end{proof}

\begin{proof}[Proof of Theorem \ref{thm:iso_KC/Y3}]
We start by proving that $\psi_{[2]}$ vanishes on $(G_0),(G_1), (G_3)$, 
and we recall that $(G_2)$ is a consequence of $(G_0)$ and $(G_1)$. First of all, let us prove that
\begin{equation}
\label{eq:G0}
\forall h\in H, \
\psi_{[2]}(h) = \psi_{[2]}(h,h).
\end{equation}
Let $Y_h$ be a $Y$-graph as described in the definition of $\psi_{[2]}$:
its ``non-special'' leaf is denoted by $L$.
By Lemma \ref{lem:slide}, the $Y$-graph $Y_h$ is equivalent to a $Y$-graph $Y_h'$ with one special leaf
and two other leaves given by $L$ and its parallel $L^\parallel$. Now, using the ``framing number zero'' assumption on $L$,
we can (up to $Y_3$-equivalence) put  $L$ and $L^\parallel$ in two  disjoint ``horizontal layers''.
More precisely, since we have $\Lk_+(L,L^\parallel) = \Fr(L)=0$,
one can find a framed oriented knot $K$ in a neighborhood of the top surface $\Sigma\times \{+1\}$ and a compact oriented
surface $S$ disjoint from $L$ such that $\partial S =  L^\parallel \sqcup (-K)$, and both knots are $0$-framed with respect to $S$. 
By Lemma \ref{lem:slide_special}, 
we can assume that the edges of $Y_h'$ do not intersect $S$.
Next, using Corollary \ref{cor:surf}, we can find a $Y$-graph $Y_h''$ with one special leaf, one leaf given by $L$ and another leaf given by $K$
such that $(\Sigma \times I)_{Y_{h}''} \overset{Y_3}{\sim} (\Sigma \times I)_{Y_h'}$.
Since we have $\Fr(K)=\Fr(L^\parallel)=\Fr(L)=0$, 
the graph $Y_h''$ can play the role of $Y_{h,h}$  in the definition of $\psi_{[2]}$. We conclude that
$$
\psi_{[2]}(h)= \left\{ (\Sigma \times I)_{Y_{h}} \right\}  = \left\{ (\Sigma \times I)_{Y_{h}''}\right\} = \psi_{[2]}(h,h).
$$
Next, the  relation
\begin{equation}\label{eq:G3}
 2\psi_{[2]}(1)=-\psi_{[2]}(\thetagraph)
\end{equation}
is an immediate consequence of Corollary \ref{cor:2spe}(2) and Lemma \ref{lem:twist_as}.
To check the  relation 
\begin{equation}\label{eq:G1}
\forall h,k\in H, \
2\psi_{[2]}(h,k)= - \psi_{[2]}\big(\hob{h}{h}{k}{k}\big) - \psi_{[2]}\big(\phio{h}{k}\big),
\end{equation}
we consider a $Y$-graph $Y_{h,k}$ as described in the definition of $\psi_{[2]}$.
Then  Lemma \ref{lem:doubling} tells that, up to $Y_3$-equivalence, 
we can replace two copies of $Y_{h,k}$ 
by a $\Phi$-graph and an H-graph, which has two pairs of parallel leaves. 
Then, using the ``framing number zero'' assumption on the non-special leaves of $Y_{h,k}$
and an argument similar to the previous lines,
the two leaves of each pair can be put  in two  disjoint ``horizontal layers''. Relation (\ref{eq:G1}) follows.

We now show that $\psi_{[2]}$ vanishes on $(D_3)$. 
More precisely, we show that Corollary \ref{cor:split_2spe} implies that, for all $h,h'\in H$, 
\begin{equation}\label{eq:D3}
 \psi_{[2]}(h+h') = \psi_{[2]}(h) + \psi_{[2]}(h') 
  - \omega(h,h')\cdot \psi_{[2]}(1) - \psi_{[2]}(\phio{h}{h'}).
\end{equation} 
Let $K_h$ (respectively $K_{h'}$) be an oriented framed knot in $\Sigma\times [-1,0]$ 
(respectively in $\Sigma\times [0,1]$) with framing number zero and 
representing $h\in H$ (respectively $h'\in H$). 
Let $K_h\sharp K_{h'}$ denote a framed connected sum of $K_h$ and $K_{h'}$.  
By Corollary \ref{cor:split_2spe} and Lemma \ref{lem:slide}, we have 
$$
(\Sigma\times I)_{Y} 
   \stackrel{Y_3}{\sim} \psi_{[2]}(h) + \psi_{[2]}(h')- \psi_{[2]}(\phio{h}{h'}), 
$$
where $Y$ denotes a $Y$-graph with two special leaves, the 
third leaf being a copy of $K_h\sharp K_{h'}$.  
On the other hand, by Lemma \ref{lem:linking_numbers} and Lemma \ref{lem:framed_connect}, 
the framing number of $K_h\sharp K_{h'}$ is equal to $-\omega(h,h')$.  
Hence, we can reduce the framing number of $K_h\sharp K_{h'}$ to zero by 
adding $|\omega(h,h')|$ isolated $(+1)$-twists or $(-1)$-twists, 
depending on whether $\omega(h,h')$ is positive or negative respectively.  
Suppose that $\omega(h,h')\leq 0$. 
Then by Corollary \ref{cor:split_2spe} and Lemma \ref{lem:special}, we have
$$ 
(\Sigma\times I)_{Y} + (-\omega(h,h'))\cdot (\Sigma\times I)_{Y_{s}} 
\stackrel{Y_3}{\sim} \psi_{[2]}(h+h') 
$$
where $Y_{s}$ is a $Y$-graph with three special leaves, and equation (\ref{eq:D3}) follows.
The case $\omega(h,h')\geq 0$ is shown similarly.

By the exact same arguments, one can use Lemma \ref{lem:split_special} to check that 
\begin{equation}\label{eq:D1}
 \psi_{[2]}(h+h',k) = \psi_{[2]}(h,k) + \psi_{[2]}(h',k) 
  - \omega(h,h')\cdot \psi_{[2]}(k) - \psi_{[2]}(\hob{h}{h'}{k}{k})
\end{equation}
for all $h,h',k \in H$.
Here again, we need the ``framing number zero'' assumption for the ``non-special'' leaves
in the definition of $\psi_{[2]}$.
Similarly, we have
\begin{equation}\label{eq:D2}
\psi_{[2]}(h,k+k') = \psi_{[2]}(h,k) + \psi_{[2]}(h,k') 
 - \omega(k,k')\cdot \psi_{[2]}(h) - \psi_{[2]}(\hob{h}{h}{k}{k'}).  
\end{equation}
for all $h,k,k'\in H$.
We thus have that $\psi_{[2]}$ vanishes on $(D_1)$ and $(D_2)$. 

So far, we have shown that $\psi_{[2]}$ factorizes to a surjective map $ \psi_{[2]}: \jacobi_{[2]}^{<,c}(H) \to \Jcob/Y_{3}$.
To prove that it is actually an isomorphism, we consider the subgroup
$``\frac{1}{2}" \jacobi^{<,c}_2(H)$
of $\jacobi^{<,c}_2(H)\otimes \Q$ generated by $\jacobi^{<,c}_2(H)$ and the following elements:
$$
\frac{1}{2}\thetagraph
\quad \hbox{and} \quad
\frac{1}{2}\big(\hob{h}{h}{k}{k} + \phio{h}{k}\big)   \ \hbox{(for all $h,k\in H$)}.
$$
Thus we have the inclusions
$$
\jacobi^{<,c}_2(H) \subset ``\frac{1}{2}" \jacobi^{<,c}_2(H) 
\subset \frac{1}{2} \jacobi^{<,c}_2(H)  \subset \jacobi^{<,c}_2(H)\otimes \Q.
$$
We also consider the homomorphism of abelian groups
$$
\digamma: \jacobi_{[2]}^{<,c}(H) \longrightarrow ``\frac{1}{2}" \jacobi^{<,c}_2(H) , 
$$ 
which is  the identity on $\jacobi^{<,c}_2(H)$ and is defined as follows
on  $ \Z\!\cdot\!(H\times H)\oplus \Z\!\cdot\! H \oplus \Z$:
$$
\digamma(h,k) := -\frac{1}{2}\big(\hob{h}{h}{k}{k} + \phio{h}{k}\big), \
\digamma(h) := - \frac{1}{2} \phio{h}{h}, \
\digamma(1) := -\frac{1}{2}\thetagraph.
$$ 
A straightforward computation, based on the multilinearity and STU-like relations in $\jacobi^{<,c}_2(H)\otimes \Q$,
shows that $\digamma$ is well-defined. We shall prove the following.

\begin{claim}\label{claim:iso}
The map $\digamma$ is an isomorphism and $``\frac{1}{2}" \jacobi^{<,c}_2(H)$ is a lattice of  $\jacobi^{<,c}_2(H)\otimes \Q$.
\end{claim}

\noindent
This will conclude the proof of the theorem since, 
by the universal property of the LMO homomorphism (\ref{eq:univLMO}), 
we have the following commutative diagram:
$$
\xymatrix{
 \jacobi_{[2]}^{<,c}(H) \ar@{->>}[d]_(0.5){\psi_{[2]}} \ar[r]^(0.5){\digamma} & ``\frac{1}{2}" \jacobi^{<,c}_2(H)  \ar@{^{(}->}[d]\\
 \Jcob/Y_{3}  \ar[r]_(0.45){\chi\circ Z_2} & \jacobi^{<,c}_{2}(H)\otimes \Q.  
}
$$

We now prove Claim \ref{claim:iso}. For that purpose, we shall work with the abelian group $\jacobi(L^\pm)$ 
defined in \S \ref{subsec:Jacobi}, which is isomorphic to $\jacobi^<(H)$ via the composition
$$
\jacobi(L^\pm) \mathop{\longrightarrow}_\simeq^\varphi \jacobi^<(-H) \mathop{\longrightarrow}_\simeq^s \jacobi^<(H).
$$
In particular, recall that $L^\pm$ denotes the abelian group freely generated 
by the set $\{1^+,\dots,g^+\}\cup\{1^-,\dots,g^-\}$.  
There is a natural order $\preceq$ on this set, 
which declares that $i^-\preceq j^-\preceq i^+\preceq j^+$ if $i\leq j$.
The loop degree gives the following decomposition:
$$ 
\jacobi^{c}_2(L^\pm)=  
\underbrace{\jacobi^{c}_{2,0}(L^\pm)}_{A:=} \oplus 
\underbrace{\jacobi^{c}_{2,1}(L^\pm)}_{B:=} \oplus \underbrace{\jacobi^{c}_{2,2}(L^\pm)}_{C:=}.
$$
The abelian group $B$ being freely generated by the elements
$$
\phin{x}{y}  \ \hbox{(for all $x,y\in \{1^+,\dots,g^+\}\cup\{1^-,\dots,g^-\}$ such that $x\preceq y$)},
$$
we also have the decomposition $\jacobi^{c}_2(L^\pm)= A \oplus B' \oplus C$
where $B'$ is the subgroup of $\jacobi^{c}_{2}(L^\pm)$ generated by the elements
$$
b(x,y) := -\phin{x}{y}  + \hn{x}{y}{y}{x}
\ \hbox{(for all $x,y$ such that $x\preceq y$)}.
$$
We then consider the subgroup 
$$
``\frac{1}{2}" \jacobi^{c}_2(L^\pm) :=
A \oplus \frac{1}{2} B' \oplus \frac{1}{2}C
$$
of $(A\oplus B' \oplus C)\otimes \Q= \jacobi^{c}_2(L^\pm)\otimes \Q$.
This is a lattice of $\jacobi^{c}_2(L^\pm)\otimes \Q$ satisfying
$$
\jacobi^{c}_2(L^\pm) \subset ``\frac{1}{2}" \jacobi^{c}_2(L^\pm) 
\subset \frac{1}{2} \jacobi^{c}_2(L^\pm)  \subset \jacobi^{c}_2(L^\pm)\otimes \Q.
$$
We also consider the group homomorphism 
$$
\Upsilon:``\frac{1}{2}" \jacobi^{c}_2(L^\pm) \longrightarrow  \jacobi_{[2]}^{<,c}(H) 
$$
that coincides with $s\circ \varphi$ on $A$ 
and is defined on the basis of $\frac{1}{2} B' \oplus \frac{1}{2}C$ as follows:
$$
 \Upsilon\left(\frac{1}{2}b(x,y)\right) := \{(x,y)\}
\ \hbox{(for all $x,y$ with $x\preceq y$)}
\quad \hbox{and} \quad
\Upsilon\left(\frac{1}{2} \thetagraph \right) := \{1\}.
$$
Here an element $x$ of $\{1^+,\dots,g^+\}\cup \{1^-,\dots,g^-\}$ is regarded as an element of $H$ by (\ref{eq:L+-_H}).
By construction of $\digamma$ and $\Upsilon$ we have the following commutative diagram:
$$
\xymatrix{
& ``\frac{1}{2}" \jacobi^{c}_2(L^\pm) \ar[ld]_-\Upsilon \ar@{^{(}->}[r]  \ar[d]^-{s\circ \varphi}&
 \jacobi^{c}_2(L^\pm)\otimes \Q  \ar[d]^-{s\circ \varphi}_-\simeq \\
\jacobi_{[2]}^{<,c}(H) \ar@{->>}[r]_-\digamma & 
``\frac{1}{2}" \jacobi^{<,c}_2(H) \ar@{^{(}->}[r] &\jacobi^{<,c}_2(H)\otimes \Q 
}
$$
Therefore Claim \ref{claim:iso} will follow from the surjectivity of $\Upsilon$.

In order to prove that $\Upsilon$ is surjective, 
observe that we have $\Upsilon(x)=\{s\circ\varphi(x)\}$ 
for all $x\in \jacobi_2^c(L^\pm)\subset ``\frac{1}{2}"\jacobi_2^c(L^\pm)$,
so that any element of $\jacobi_{[2]}^{<,c}(H)$ coming from the summand $\jacobi_2^{<,c}(H)$ belongs to the image of $\Upsilon$.
It is also clear that $\{1\}$ is in the image of $\Upsilon$.
Thus, we just have to check that $(h)$ belongs to $\Img(\Upsilon)$ for all $h\in H$,
and that $(h,k)$ belongs to $\Img(\Upsilon)$  for all $h,k\in H$. Relation $(D_3)$ implies that
$$
\forall h_1,h_2\in H, \ (h_1+h_2) \equiv (h_1) +(h_2)  \mod \Img(\Upsilon)
$$
so that $(h)$ can be decomposed as a  sum of some  $(x)$'s
(with $x\in \{1^+,\dots,g^+,1^-,\dots,g^-\}$). 
Since we have $(x)=\Upsilon\left(\frac{1}{2}b(x,x)\right)$, this shows that $(h)\in \Img(\Upsilon)$.
Next, relations $(D_1)$ and $(D_2)$ imply that
$$
\forall h_1,h_2\in H, \ (h_1+h_2,k) \equiv  (h_1,k) + (h_2,k) \mod \Img(\Upsilon),
$$
$$
\forall k_1,k_2\in H, \ (h,k_1+k_2) \equiv  (h,k_1) + (h,k_2) \mod \Img(\Upsilon),
$$
so that $(h,k)$ writes as a sum of some  $(x,y)$'s (where $x,y\in \{1^+,\dots,g^+,1^-,\dots,g^-\}$). 
Since we have $(x,y)=\Upsilon\left(\frac{1}{2}b(x,y)\right)$ if $x\preceq y$ and 
since (\ref{eq:symmetry_defect}) implies that $(x,y) \equiv (y,x) \mod \Img(\Upsilon)$ in general,
all this shows that $(h,k)\in \Img(\Upsilon)$. 
\end{proof}

\subsection{Diagrammatic description of $\cyl/Y_{3}$}\label{subsec:description2} 

We can now give the diagrammatic description of the isomorphism type of the group $\cyl/Y_3$.  

\begin{theorem}\label{thm:2cocycle}
The characteristic class of the central extension 
\begin{equation}\label{eq:Habiro_sec_bis} 
\xymatrix{
0 \ar[r] & \Jcob/Y_3 \ar[r] & \cyl/Y_3 \ar[r]^{\tau_1} & \Lambda^3 H \ar[r] & 1} 
\end{equation}
seen as an element of 
$$
H^2\left(\Lambda^3H; \Jcob/Y_3\right) \simeq \Hom\left(\Lambda^2 \Lambda^3H,\Jcob/Y_3\right)
\overset{\psi_{[2]}}{\simeq} \Hom\left(\Lambda^2 \Lambda^3H,\jacobi_{[2]}^{<,c}(H)\right)
$$
is  the antisymmetric bilinear map
$[\cdot,\cdot]: \Lambda^3 H\times \Lambda^3 H \to \jacobi_{[2]}^{<,c}(H)$ defined by 
$$
[a\wedge b\wedge c,d\wedge e\wedge f] := 
\left(\begin{array}{c}
\quad \omega(c,d)\ho{a}{b}{e}{f} - \omega(b,d)\ho{a}{c}{e}{f} + \omega(a,d)\ho{b}{c}{e}{f}  \\
+ \omega(c,e)\hob{d}{a}{b}{f} - \omega(b,e)\hob{d}{a}{c}{f} + \omega(a,e)\hob{d}{b}{c}{f}  \\
+ \omega(c,f)\ho{d}{e}{a}{b} - \omega(b,f)\ho{d}{e}{a}{c} + \omega(a,f)\ho{d}{e}{b}{c}
\end{array}\right).
$$
\end{theorem}

\noindent
Theorem \ref{thm:2cocycle} is a refinement of a result of Morita on the Torelli group. 
Indeed he described in \cite[Theorem 3.1]{Morita_Casson2} the characteristic class of the  central extension
 $$  
\xymatrix{0 \ar[r] & \tau_2\left(\Johnson \right) \ar[r]^{\tau_2^{-1}} 
& \Torelli/\Torelli[3] \ar[r]^{\tau_1} & \Lambda^3 H \ar[r] & 1}. 
$$
According to Lemma \ref{lem:tau2}, the map 
$\tau_2 \circ \psi_{[2]}:\jacobi_{[2]}^{<,c}(H) \to \frac{\left(\Lambda^2 H \otimes \Lambda^2H\right)^{\mathfrak{S}_2}}{\Lambda^4H}$
sends $\ho{a}{b}{c}{d}$ to $(a\wedge b)  \leftrightarrow (c\wedge d)$. 
Thus, taking into account the fact that Morita's $\tau_2$ differs from ours by a minus sign, 
one sees that Theorem \ref{thm:2cocycle}  generalizes Morita's description.

\begin{proof}[Proof of Theorem \ref{thm:2cocycle}]
Let us denote by $e$ the characteristic class in 
$$
H^2(\Lambda^3H;\Jcob/Y_3)\simeq  \Hom(\Lambda^2\Lambda^3H,\Jcob/Y_3)
$$ 
of the central extension (\ref{eq:Habiro_sec_bis}), and let $s$ be a setwise section of $\tau_1$.  
Then the cohomology class $e$ is represented by the $2$-cocycle $c$ (in the bar complex) defined by 
$$ 
\forall x,y\in \Lambda^3 H, \
c(x\vert y):=s(x)s(y)s(xy)^{-1} \ \in \Jcob/Y_3 \subset \cyl/Y_3,
$$ 
so that $e\in \Hom(\Lambda^2\Lambda^3H,\Jcob/Y_3)$ is given by 
\begin{eqnarray*}
\forall x,y\in \Lambda^3 H, \ e(x\wedge y) = c\big((x\vert y)-(y\vert x) \big) &=&
c(x\vert y)c(y\vert  x)^{-1}\\
& = & s(x)s(y)s(xy)^{-1}s(yx)s(x)^{-1}s(y)^{-1}\\ 
&= & [ s(x),s(y) ].
\end{eqnarray*} 
This shows that $e$ is determined by the Lie bracket of the Lie ring of homology cylinders $\Gr^Y\cyl$,
in the sense that the following diagram is commutative:
$$  
\xymatrix{
\cyl/Y_2\times \cyl/Y_2 \ar[r]^(0.6){[\cdot,\cdot]} \ar@{->>}[d]_-{\tau_1\times \tau_1}
& Y_2\cyl/Y_{3}\ar@{^{(}->}[d] \\
\Lambda^3 H\times \Lambda^3 H \ar[r]_{e} &  \Jcob/Y_3. 
 }
$$
But $\tau_1$ induces an isomorphism from $\frac{\cyl/Y_2}{\Tors(\cyl/Y_2)}$ to $\Lambda^3H\simeq \jacobi^{<,c}_1(H)$ \cite{MM},
with inverse given by the surgery map $\psi_1$ of \S \ref{subsec:LMO}. 
Moreover, the Lie bracket of $\Gr^Y\cyl$ factorizes to $\frac{\cyl/Y_2}{\Tors(\cyl/Y_2)}$ 
since $Y_2\cyl/Y_{3}$ is torsion-free by Corollary \ref{cor:varphi2}. Thus, we obtain
$$
\xymatrix{
\frac{\cyl/Y_2}{\Tors(\cyl/Y_2)}\times \frac{\cyl/Y_2}{\Tors(\cyl/Y_2)} \ar[r]^-{[\cdot,\cdot]}  & Y_2\cyl/Y_{3} \ar@{^{(}->}[d]  \\
\Lambda^3 H\times \Lambda^3 H \ar[r]_-{e}\ar[u]_{\simeq}^{\psi_1\times \psi_1} & \Jcob/Y_3. 
}
$$
Since the surgery map $\psi$ preserves the Lie brackets, 
we obtain that, for all $a\wedge b \wedge c\in \Lambda^3H$ and  $d\wedge e \wedge f\in \Lambda^3H$,
\begin{eqnarray*}
\psi_{[2]}^{-1}\circ e\big((a \wedge b \wedge c)\wedge (d\wedge e \wedge f)\big)
&=& \psi_2^{-1}\left(\big[\psi_1(a \wedge b \wedge c),\psi_1(d\wedge e \wedge f)\big]\right)\\
 &=&   \left[\yo{a}{b}{c},\yo{d}{e}{f}\right]\\
& = &\yo{a}{b}{ \ \ c<} \yo{d}{e}{f} - \yo{d}{e}{ \ \ f <}\yo{a}{b}{c}.
\end{eqnarray*}
The desired formula follows from the STU-like relation.
\end{proof}

We conclude this section with a few consequences of the previous results on the structure of the group $\cyl/Y_3$.
First of all, a presentation of $\cyl/Y_3$ could be obtained from 
a presentation of $\Jcob/Y_3$ (which is discussed after Theorem \ref{thm:iso_KC/Y3})
using the short exact sequence (\ref{eq:Habiro_sec_bis}). 
Besides, we can deduce the following assertions.

\begin{corollary}\label{cor:prop_modY3}
The group $\cyl/Y_{3}$ has the following properties:
\begin{itemize}
\item[(i)] It is torsion-free;
\item[(ii)] Its center is $\Jcob/Y_3$;
\item[(iii)] Its commutator subgroup is strictly contained in $Y_2\cyl/Y_3$,
and it is the image of $\mcyl(\Gamma_2 \Torelli)$ 
by the canonical projection $\cyl \to \cyl/Y_3$.
\end{itemize}
\end{corollary}

\begin{proof}
According to Theorem \ref{thm:iso_KC/Y3}, the group $\Jcob/Y_{3}$ is free abelian.
Since $\Lambda^3 H$ is  also a free abelian group,
assertion (i) follows from  the short exact sequence (\ref{eq:Habiro_sec_bis}).

We already know that (\ref{eq:Habiro_sec_bis}) is a central extension. To prove assertion (ii),
it thus remains to show that any central element $M$ of $\cyl/Y_3$ belongs to $\Jcob/Y_3$.
Since $M$ is central, 
$t:=\tau_1(M)\in \Lambda^3 H$ satisfies $[t,\cdot]=0$ for the bracket introduced in Theorem \ref{thm:2cocycle}.
By composing this bracket with 
$$
\xymatrix{
\jacobi_{2}^{<,c}(H) \subset \jacobi_{2}^{<,c}(H_\Q)  \ar[r]^-{\chi^{-1}}_-\simeq & \jacobi_{2}^{c}(H_\Q)  
\ar@{->>}[r]^-{p_{2,2}} &  \jacobi_{2,2}^{c}(H_\Q),
}
$$
we get a skew-symmetric bilinear form $b_\omega: \Lambda^3 H_\Q \times \Lambda^3 H_\Q \to \jacobi_{2,2}^{c}(H_\Q)$
and, again, we have $b_\omega(t,\cdot)=0$.
A direct computation shows that, for all $x_1,x_2,x_3,y_1,y_2,y_3 \in H_\Q$,
$$
b_\omega(x_1\wedge x_2 \wedge x_3, y_1\wedge y_2 \wedge y_3)=
- \frac{1}{4} \left| {\scriptsize \begin{array}{ccc} \omega(x_1, y_1) & \omega(x_1, y_2) & \omega(x_1 , y_3)\\
\omega(x_2 , y_1) & \omega(x_2, y_2) & \omega(x_2 , y_3)\\
\omega(x_3, y_1) & \omega(x_3 , y_2) & \omega(x_3 , y_3)
\end{array} }\right|\cdot \thetagraph.
$$
(See  \cite[Lemma 5.4]{HM_SJD}.) 
By considering a symplectic basis of $(H_\Q,\omega)$, 
it can be seen that $b_\omega$ is itself a symplectic form on $\Lambda^3 H_\Q$. 
We deduce that $t=0$, \ie $M$ belongs to $\Jcob/Y_3$. 

We now prove assertion (iii). 
The inclusion $\Gamma_n\!\left(\cyl/Y_k\right) \subset Y_n\cyl/Y_k$ holds true
for any integers $k\geq n \geq1$ by results of Goussarov \cite{Goussarov} and Habiro \cite{Habiro}.
For $k=3$ and $n=2$, this inclusion is strict for the following reasons:
in the case $g=0$, the group $\cyl/Y_3$ is abelian but $Y_2\cyl/Y_3$ is sent isomorphically to $2\Z$ 
by the Casson invariant \cite{Habiro}; in the case $g>0$, 
we know  by Proposition \ref{prop:alpha_properties} that the group homomorphism
$\alpha: \cyl/Y_3 \to S^2H$ is not trivial on $Y_2\cyl/Y_3$, but it is on $\Gamma_2\!\left(\cyl/Y_3\right)$
since $S^2H$ is abelian. It now remains to show that $\Gamma_2\!\left(\cyl/Y_3\right)$ is contained
in $\mcyl(\Gamma_2 \Torelli)/Y_3$ (the converse inclusion being trivially true).
For this, we consider a finite product of commutators 
$$
p:=\prod_{i=1}^r [M_i,N_i]
$$
in the group $\cyl/Y_3$. If we replace one of the $M_i$'s 
by its product $M_i \cdot K$ with a $K \in \Jcob/Y_3$, 
the product $p$ remains unchanged since $K$ belongs to the center of $\cyl/Y_3$.
But, any $M\in \cyl/Y_3$ can be decomposed in the form $M=K\cdot \mcyl(h)$ where $K \in \Jcob/Y_3$
and $h\in \Torelli$ since we have the short exact sequence  (\ref{eq:Habiro_sec_bis}) 
and $\tau_1: \Torelli \to \Lambda^3 H$ is surjective. Thus we can assume that
each $M_i$ and each $N_i$ in $p$ belongs to $\mcyl(\Torelli)/Y_3$.
\end{proof}

\section{Characterization of the $J_2$-equivalence and the $J_3$-equivalence} \label{sec:J2_J3}

In this section we characterize the $J_2$-equivalence  and the $J_3$-equivalence relations
(Theorem~B and Theorem~C, respectively).

\subsection{Proof of Theorem~B}

The following lemma with $k=2$ proves the necessary condition in Theorem~B.

\begin{lemma}
\label{lem:rho_k}
Let $M,M'\in \cyl$. 
If $M$ is $J_k$-equivalent to $M'$, then we have
$$
\rho_k(M)=\rho_k\left(M'\right) \ \in \Aut(\pi/\Gamma_{k+1}\pi).
$$
\end{lemma}

\begin{proof}
By assumption, there exist a surface $S \subset \interior(M)$ and an  $s\in \mcg(S)[k]$ such that $M'=M_s$.
Let $E$ be the closure of $M\setminus (S\times [-1,1])$ 
where $S\times [-1,1]$ is a regular neighborhood of $S$ in $M$.
Thus $M'$ is obtained by gluing to $E$ the mapping cylinder of $s$.
The van Kampen theorem shows that there exists a unique isomorphism
between $\pi_1(M)/\Gamma_{k+1} \pi_1(M)$ and $\pi_1(M')/\Gamma_{k+1} \pi_1(M')$ 
such that the following diagram commutes:
$$
\xymatrix @!0 @R=1cm @C=3cm {
{\frac{\pi_1(M)}{\Gamma_{k+1} \pi_1(M)}} \ar@{-->}[rr]_-\simeq^-{\exists !}
& & {\frac{\pi_1(M')}{\Gamma_{k+1} \pi_1(M')}} .\\
&{\frac{\pi_1\left(E\right)}{\Gamma_{k+1}\pi_1\left(E\right)}} 
\ar@{->>}[lu] \ar@{->>}[ru] & 
}
$$
This triangular diagram can be ``expanded''  as follows:
$$
\xymatrix{
&\frac{\pi_1\left(\Sigma\right)}{\Gamma_{k+1}\pi_1\left(\Sigma\right)}
\ar[dl]^-\simeq_-{m_{+,*}}  \ar[dr]_-\simeq^-{m_{+,*}'} \ar[dd]&\\
\frac{\pi_1\left(M\right)}{\Gamma_{k+1}\pi_1\left(M\right)}\ar[rr] |!{[ur];[dr]}\hole^<<<<<<<<{\simeq}
&&\frac{\pi_1\left(M'\right)}{\Gamma_{k+1}\pi_1\left(M'\right)}\\
&\frac{\pi_1\left(E\right)}{\Gamma_{k+1}\pi_1\left(E\right)}\ar[lu]\ar[ru]&\\
&\frac{\pi_1\left(\Sigma\right)}{\Gamma_{k+1}\pi_1\left(\Sigma\right)} 
\ar@/^1pc/[uul]_-\simeq^-{m_{-,*}} \ar@/_1pc/[uur]^-\simeq_-{m_{-,*}'} \ar[u]&
}
$$
The front faces of this bipyramidal diagram commute.
Therefore its back faces are commutative too, and we deduce that $\rho_k(M)=\rho_k(M')$.
\end{proof}

To prove the sufficient condition in Theorem~B, assume that $M,M' \in \cyl$ are such that $\rho_2(M)=\rho_2(M')$.
Since $\cyl/Y_2$ is a group \cite{Goussarov,Habiro}, 
there is a $D\in \cyl$ such that $M$ is $Y_2$-equivalent to $D\circ M'$. 
We deduce that $\rho_2(D)=1$ or, equivalently, that $\tau_1(D)=0$.
As in the proof of Lemma \ref{lem:d^2beta}, we can use \cite{MM}
to find some $Y$-graphs with special leaves $G_1, \dots, G_m$ in $(\Sigma \times I)$ such that 
$D$ is $Y_2$-equivalent to $\prod_{i=1}^m (\Sigma\times I)_{G_i}$.
Since surgery along a $Y$-graph with  special leaf is equivalent 
to a Dehn twist along a bounding simple closed curve of genus $1$,
each cobordism $(\Sigma\times I)_{G_i}$ is $J_2$-equivalent to $(\Sigma \times I)$.
Using the fact that ``$Y_2 \Rightarrow J_2$'', we conclude that
$$
M\stackrel{J_2}{\sim} D \circ M' \stackrel{J_2}{\sim}  \prod_{i=1}^m (\Sigma\times I)_{G_i} \circ M'
\stackrel{J_2}{\sim}  (\Sigma \times I) \circ M' = M'.
$$ 

\subsection{Proof of Theorem~C}

The necessary condition in Theorem~C follows from Lemma \ref{lem:rho_k} 
and the next result (with $k=3$).

\begin{lemma}
\label{lem:tau_mod_Ik}
Let $M,M'\in \cyl$. 
If $M$ is $J_k$-equivalent to $M'$, then we have
$$
\tau(M',\partial_-M';\xi_0)- \tau(M,\partial_-M;\xi_0) \in I^k.
$$
\end{lemma}

\noindent
The analogous statement for closed $3$-manifolds is proved in \cite[Lemma 4.14]{Massuyeau_torsion}.
The proof is easily adapted to  homology cylinders: 
see the proof of Theorem \ref{thm:torsion_finiteness}.

To prove the sufficient condition in Theorem C, we need the following result of Morita.

\begin{proposition}[Morita \cite{Morita_Casson1}]\label{prop:Morita}
There exists a homology $3$-sphere $P$ that is $J_3$-equivalent to $S^3$
and whose Casson invariant is $+1$.
\end{proposition}

\noindent
This is proved in \cite[Proposition 6.4]{Morita_Casson1}.
Since this will play a crucial role in the sequel,
we would like to develop a little bit Morita's argument.

\begin{proof}[Proof of Proposition \ref{prop:Morita}]
Let $R$ be a compact connected oriented surface of genus $2$ with one boundary component,
and let $j:R\to S^3$ be a Heegaard embedding as explained in \S \ref{subsec:Casson}.
Morita shows that there exists an element $\psi\in \mcg(R)[3]$ such that $\lambda_j(\psi)=1$.
To prove this, he uses his decomposition formula for the Casson invariant 
(which is recalled in \S \ref{subsec:review_Morita})
and he considers the following element of $\left(\Lambda^2 H \otimes \Lambda^2H\right)^{\mathfrak{S}_2}$:
$$ 
 s_1:= (\alpha_1\wedge \beta_1)\leftrightarrow (\alpha_2\wedge \beta_2) 
 -(\alpha_1\wedge \alpha_2)\leftrightarrow (\beta_1\wedge \beta_2)
+ (\alpha_1\wedge \beta_2)\leftrightarrow (\beta_1\wedge \alpha_2). 
$$
Note that $s_1\in \Lambda^4H\subset \left(\Lambda^2 H \otimes \Lambda^2H\right)^{\mathfrak{S}_2}$.  
Morita claims that there exists a family of elements 
$u_1,v_1,\dots,u_r,v_r$ of $H$ such that $\omega(u_i,v_i)=1$ for each $i\in \{1,\dots, r\}$, and  
\begin{equation}\label{eq:s1}
s_1 = 3\cdot (\alpha_1\wedge \beta_1 + \alpha_2\wedge \beta_2)^{\otimes 2} 
+ \sum_{i=1}^r \pm  (u_i\wedge v_i)^{\otimes 2}.
\end{equation}
For example, it can be checked that the following equality holds
\begin{eqnarray*}
  s_1 \quad =&  & 3\cdot (\alpha_1\wedge \beta_1 + \alpha_2\wedge \beta_2)^{\otimes 2} \\
  & + & \big((2\alpha_1 + \beta_2)\wedge (\beta_1 + \alpha_2)\big)^{\otimes 2} 
            - \big((\alpha_1 + \beta_1 + \alpha_2)\wedge (\alpha_1 + \beta_1 + \beta_2)\big)^{\otimes 2} \\
  & + & \big((\alpha_1 - \beta_2)\wedge \beta_1\big)^{\otimes 2} 
  + \big( \alpha_1\wedge (\beta_1 - \alpha_2)\big)^{\otimes 2} \\
  & - & 2\cdot\big( \alpha_1\wedge (\beta_1 + \alpha_2)\big)^{\otimes 2} 
  - 2\cdot \big(\alpha_2\wedge (\alpha_1 + \beta_2)\big)^{\otimes 2} \\
  & - & \big( \alpha_1\wedge (\beta_1 + \alpha_2 + \beta_2)\big)^{\otimes 2} 
  + \big( \alpha_2\wedge (\alpha_1 + \beta_1 + \beta_2)\big)^{\otimes 2} \\
  & + & \big((\alpha_1 + \beta_1 + \alpha_2)\wedge \beta_2\big)^{\otimes 2} 
  - \big((\alpha_1 + \alpha_2 + \beta_2)\wedge \beta_1\big)^{\otimes 2} \\
  & + & \big( \alpha_1\wedge (\beta_1 + \beta_2)\big)^{\otimes 2} 
  - \big((\beta_1 + \alpha_2)\wedge \beta_2\big)^{\otimes 2} 
           + \big((\alpha_1 + \alpha_2)\wedge \beta_1\big)^{\otimes 2} \\
  & - & 7\cdot(\alpha_1\wedge \beta_1)^{\otimes 2} - 2\cdot(\alpha_2\wedge \beta_2)^{\otimes 2},  
\end{eqnarray*}
which proves the existence of such a family.  

Now, for each pair of elements $(u_i,v_i)$ such that $\omega(u_i,v_i)=1$, 
there exists some $\phi_i\in \Sp(H)$ such that $\phi_i(u_i)=\alpha_1$ and $\phi_i(v_i)=\beta_1$.  
Therefore, there exists a genus one bounding simple closed curve $\gamma_i$ 
such that the Dehn twist $T_{\gamma_i}$ along $\gamma_i$ satisfies 
$\tau_2(T_{\gamma_i})=(u_i\wedge v_i)^{\otimes 2}$.  
Let also $\gamma_0$ be a genus two bounding simple closed curve 
such that $\tau_2(T_{\gamma_0})=(\alpha_1\wedge \beta_1 + \alpha_2\wedge \beta_2)^{\otimes 2}$. 
(Here, we have used Morita's formula \cite[Proposition 1.1]{Morita_Casson1} for $\tau_2$: see (\ref{eq:tau2BSCC}) below.)
Then  
$$ 
\psi := T_{\gamma_0}^{-3} \circ \prod_{i=1}^rT_{\gamma_i}^{\mp 1} \ \in \mcg(R)[2]
$$
has the desired properties. Indeed we have
$$
\tau_2(\psi)=\{-s_1\}=0 \ \in \left(\Lambda^2 H \otimes \Lambda^2H\right)^{\mathfrak{S}_2}/\Lambda^4H,
$$
so that $\psi$ actually belongs to $\mcg(R)[3]$, and
$$
\lambda_j(\psi)=-\frac{1}{24}d(\psi) = -\frac{1}{24}\left( -3\cdot 8+ 0\right)=1. 
$$
Here, $d:\mcg(R)[2]\to \Z$ denotes Morita's core of the Casson invariant 
and two of his formulas are used: 
see  (\ref{eq:decomposition_d}) and (\ref{eq:core_bscc}) below.  
This completes the proof.  
\end{proof}

We now prove the sufficient condition in Theorem~C. 
Let $M,M' \in \cyl$ be such that $\rho_3(M)=\rho_3(M')$ and $\alpha(M)=\alpha(M')$.
Since $\cyl/Y_3$ is a group \cite{Goussarov,Habiro}, 
there is a $D\in \cyl$ such that $M$ is $Y_3$-equivalent to $D\circ M'$.
This $D$ satisfies $\rho_3(D)=1$ and $\alpha(D)=0$. 
Let $P$ be the homology $3$-sphere from Proposition \ref{prop:Morita} and set
$$
 D' := \left(\Sigma \times I\right) \sharp P^{\sharp \lambda_j(D)}.
$$
Then, $D$ and $D'$ share the same invariants $\rho_3,\alpha$ and $\lambda_j$
so that they are $Y_3$-equivalent by Theorem~A. 
Using the fact that ``$Y_3 \Rightarrow J_3$'', we conclude that
$$
M\stackrel{J_3}{\sim}  D \circ M' \stackrel{J_3}{\sim}  D' \circ M' = 
 M' \sharp P^{\sharp \lambda_j(D)} \stackrel{J_3}{\sim}  M'.
$$

\begin{remark}
According to Habiro \cite{Habiro}, the $Y_4$-equivalence for homology $3$-spheres is also classified by the Casson invariant.
One can wonder whether there exists a homology $3$-sphere that is $J_4$-equivalent to $S^3$ and whose Casson invariant is $+1$.
It would follow from an affirmative answer to this question, and the same argument as above, 
that any  homology $3$-sphere is $J_4$-equivalent to $S^3$, thus improving Pitsch's result \cite{Pitsch}.  
\end{remark}

\section{Core of the Casson invariant for homology cylinders}\label{sec:core_Casson}

In this section, we extend Morita's definition of the core of the Casson invariant
\cite{Morita_Casson1,Morita_Casson2} to the monoid $\Jcob=\cob[2]$ (Theorem~D).
At this point, it is important to emphasize how our sign conventions and notation differ from Morita's. 
The $k$-th Johnson homomorphism $\tau_k: \mcg[k] \to H \otimes\mathcal{L}_{k+1}(H)$ 
defined in \S \ref{subsec:Johnson} corresponds to $-\tau_{k+1}$ in Morita's papers. 
(Note the shift of index and the minus sign: 
we have identified $H$ with $H^*$ by $h\mapsto \omega(h,\cdot)$
while Morita uses the map $h\mapsto \omega(\cdot,h)$ in his papers.)
Besides, for a Heegaard embedding $j: \Sigma \to S^3$,
our map $\lambda_j\circ \mcyl: \Torelli \to \Z$ defined in \S \ref{subsec:Casson}
corresponds to $-\lambda_j^*$ in Morita's papers.

\subsection{A quick review of Morita's work}\label{subsec:review_Morita}

We summarize some of the results obtained by Morita in \cite{Morita_Casson1,Morita_Casson2}.
Let $j:\Sigma \to S^3$ be a Heegaard embedding as in \S \ref{subsec:Casson}. 
Morita proved that the restriction of $\lambda_j$ to $\Johnson=\mcg[2]$ is a group homomorphism.
This restricted map ``suffices'' for the understanding of the Casson invariant,
since any homology $3$-sphere is $J_2$-equivalent to $S^3$.
Furthermore, Morita showed that $\lambda_j:\Johnson \to \Z$ decomposes 
as a sum of two homomorphisms, one being completely determined by the second Johnson homomorphism $\tau_2$, 
and the other one -- which he calls the \emph{core} of the Casson invariant -- 
being independent of the embedding $j$.
Let us recall this decomposition in more detail.

To define the first of the two homomorphisms in Morita's decomposition, let $\mathcal{A}$ 
be the algebra over $\Z$ generated by elements $l(u,v)$, for each pair of elements $u$, $v$ of $H$, 
and subject to the relations 
 $$ 
 l(n\cdot u+n'\cdot u',v)=n\cdot l(u,v) + n' \cdot l(u',v)\quad \textrm{and}\quad   l(v,u) = l(u,v) +\omega(u,v)
 $$
for all $u,u',v \in H$ and $n,n'\in \Z$. 
The embedding $j:\Sigma \to S^3$  defines an algebra homomorphism
$\varepsilon_j:\mathcal{A}\rightarrow \Z$ by setting 
$$ 
\varepsilon_j\left( l(u,v) \right) := \operatorname{lk}(u,v^+), 
$$
where $v^+$ is a push-off of $v$ 
in the positive normal direction of 
$F_g= j(\Sigma)\subset S^3$. 
Let $\theta: \left(\Lambda^2 H \otimes \Lambda^2H\right)^{\mathfrak{S}_2}\rightarrow \mathcal{A}$ be the group  homomorphism defined by
$$
\left\{\begin{array}{lcl}
\theta\left( (u\wedge v)\otimes (u\wedge v)\right) &:= &l(u,u)l(v,v) - l(u,v)l(v,u)\\
\theta\left( (a\wedge b)\leftrightarrow (c\wedge d) \right) &:= & l(a,c)l(b,d) - l(a,d)l(b,c) - l(d,a)l(c,b) + l(c,a)l(d,b).
\end{array}\right. 
$$
Using Casson's formula relating the variation of his invariant under surgery along a $(\pm 1)$-framed knot
to the Alexander polynomial of that knot, Morita was able to compute the value of $\lambda_j$ 
on a twist $T_\gamma$ along a bounding simple closed curve $\gamma \subset \Sigma$. 
He found that
\begin{equation}
\label{eq:lambda_bscc}
\lambda_j(T_\gamma) = \varepsilon_j\circ \theta(\omega_\gamma\otimes \omega_\gamma)
\end{equation}
where $\omega_\gamma$ is the symplectic form of the subsurface of $\Sigma$ bounded by $\gamma$.
(More precisely, we have $\omega_\gamma=\sum_{i=1}^h u_i\wedge v_i$ if the subsurface has genus $h$ 
and if $(u_i,v_i)_{i=1}^h$ is any symplectic basis of its first homology group.)
Let also $\overline{d}:  \left(\Lambda^2 H \otimes \Lambda^2H\right)^{\mathfrak{S}_2}\rightarrow \Z$ be the homomorphism defined by
$$
\left\{\begin{array}{lcl}
\overline{d}\left( (u\wedge v)\otimes (u\wedge v) \right)& :=  & 0\\
\overline{d}\left( (a\wedge b)\leftrightarrow (c\wedge d) \right) &:= &
\omega(a,b)\omega(c,d) - \omega(a,c)\omega(b,d) + \omega(a,d)\omega(b,c).\\
\end{array}\right. 
$$
It turns out that 
 $
 \overline{q_j}:=\varepsilon_j\circ \theta + \frac{1}{3}\overline{d} 
 $
vanishes on $\Lambda^4 H\subset \left(\Lambda^2 H \otimes \Lambda^2H\right)^{\mathfrak{S}_2}$, 
so that we can see it as a homomorphism $\overline{q_j}:\DD_2(H) \to \Z$.
(Recall that the target $\DD_2(H)$ of $\tau_2$ is identified
with $\left(\Lambda^2 H \otimes \Lambda^2H\right)^{\mathfrak{S}_2}/\Lambda^4H$
by Proposition \ref{prop:Morita-Levine}.) Hence we obtain a homomorphism 
$$
q_j := -\overline{q_j}\circ \tau_2:\Johnson \longrightarrow \Q.  
$$
Another description of $q_j$ is given by Perron in \cite{Perron} using Fox's differential calculus.

The definition of the second homomorphism in Morita's decomposition of $\lambda_j$ is more delicate.
Let $k: \mcg \to H$ be a crossed homomorphism 
whose homology class generates $H^1(\mcg;H) \simeq \Z$, see \cite{Morita_families1}, 
and which is invariant under stabilization of the surface $\Sigma$.
There is a $2$-cocycle $c_k: \mcg \times \mcg \to \Z$ associated to $k$ by the formula
$c_k(\phi,\psi) := \omega(k(\phi^{-1}),k(\psi))$. 
This cocycle represents the first characteristic class of surface bundles
$e_1 \in H^2(\mcg)$ introduced by Morita in \cite{Morita_characteristic1,Morita_characteristic2}. 
Meyer's signature $2$-cocycle on $\Sp(H)$ \cite{Meyer} composed with $\rho_1:\mcg \to \Sp(H)$
gives another $2$-cocycle $\tau$  on $\mcg$ such that $[-3\tau]=e_1$.
Since $\mcg$ is perfect (in genus $g\geq 3$), there is a unique $1$-cochain
$d_k:\mcg \to \Z$ whose coboundary is $c_k+3\tau$ and which is preserved by stabilization of the surface $\Sigma$. 
Because the $1$-cocycle $k$ vanishes on $\Johnson$,
the restriction of $d_k$ to $\Johnson$ is a group homomorphism
$$
d:\Johnson \longrightarrow \Z
$$
which does not depend on the choice of $k$.

The group $\Johnson$ is according to Johnson \cite{Johnson_subgroup}
generated by Dehn twists along bounding simple closed curves, and Morita proves that
\begin{equation}
\label{eq:core_bscc}
d(T_\gamma)=4h(h-1)
\end{equation}
for any bounding simple closed curve $\gamma\subset \Sigma$ of genus $h$.
He deduces from (\ref{eq:lambda_bscc}) and (\ref{eq:core_bscc}) 
the following decomposition formula for $\lambda_j$:
\begin{equation}
\label{eq:decomposition_d}
-\lambda_j = \frac{1}{24} d +q_j: \Johnson \longrightarrow \Z.
\end{equation}
Recall that any homology $3$-sphere is $J_3$-equivalent to $S^3$, as expected by Morita \cite{Morita_Casson1}
and proved by Pitsch \cite{Pitsch}. Thus the homomorphism $d:\Johnson \to \Z$,
and more precisely its restriction to the subgroup $\mcg[3]$,
contains all the topological information on homology $3$-spheres carried by the Casson invariant:
Morita calls $d$ the \emph{core of the Casson invariant}. 
Observe that $d$ takes values in $8 \Z$ according to (\ref{eq:core_bscc}),
and that it is obviously trivial in genus $g=0,1$.

\subsection{Proof of the existence in Theorem~D}\label{subsec:existence}

We now go back to homology cylinders and we consider the submonoid $\Jcob=\cob[2]$ of $\cyl$.
We shall prove that, for any genus $g\geq 0$,
the group homomorphism $d:\Johnson \to 8\Z$ can be extended 
to a $Y_3$-invariant and 
$\mcg$-conjugacy invariant 
monoid homomorphism $d:\Jcob \to 8\Z$.
We start by considering the map
$$ 
d'': \Jcob \longrightarrow \Q 
$$
defined by $d'':=p_{2,2}\circ Z_2$. 
In other words, $d''(M)$ is  the coefficient of $\thetagraph$ in $Z(M)$.  

\begin{lemma}\label{lem:d''}
The map  $d''$ is a monoid homomorphism and has the following properties:
\begin{itemize}
\item[(i)] It is canonical, \ie it does not depend on the choices which are needed 
in the construction of the LMO homomorphism $Z$;
\item[(ii)] It is invariant under $Y_3$-equivalence and conjugation by $\mcg$;
\item[(iii)] It takes  values in $\frac{1}{8}\Z$.
\end{itemize}
\end{lemma}

\begin {proof}
For any $M,M'\in \cyl$, we have $Z(M\circ M')= Z(M)\star Z(M')$
so that the coefficient of $\thetagraph$ in $Z(M\circ M')$ 
is the sum of the same coefficients in $Z(M)$ and $Z(M')$, 
plus a contribution of $Z_1(M)\star Z_1(M')$. 
By (\ref{eq:tau1_LMO}) 
the latter vanishes if $M,M'$ belong to $\Jcob$, 
so that $d''$ is a monoid homomorphism.  

To prove (i), observe that
for every $M \in \Jcob$ the square of $\{M\}\in \Jcob/Y_3$ belongs to the subgroup $Y_2\cyl/Y_3$ 
(as follows from \cite{MM} or from \S \ref{subsec:description1}).
Thus, the formula 
\begin{equation}
\label{eq:half_d''}
d''(M)=\frac{1}{2} d''(M\circ  M)
\end{equation}
shows that $d''$ is determined by its restriction to $Y_2\cyl$.
Since the inverse of  $Z_2: (Y_2\cyl / Y_3)\otimes \Q \rightarrow \jacobi^c_2(H_\Q)$ 
is the surgery map $\psi_2 \circ \chi_2$ by (\ref{eq:univLMO}),
it does not depend on the choices involved in the construction of $Z$.
We deduce assertion (i).

The invariance of $d''$ under $Y_3$-equivalence  is inherited from $Z_2:\cyl \to \jacobi_2(H_\Q)$.
Since $Z_2:Y_2\cyl \to  \jacobi_2^c(H_\Q)$ is $\mcg$-equivariant \cite{HM_SJD}, we have
$$
d''\left(\mcyl(f)\circ M \circ \mcyl(f^{-1})\right)
\stackrel{(\ref{eq:half_d''})}{=} 
\frac{1}{2}d''\left(\mcyl(f)\circ M^2\circ  \mcyl(f^{-1})\right)
= \frac{1}{2}d''(M^2) \stackrel{(\ref{eq:half_d''})}{=} d''(M)
$$
for all $M\in \Jcob$ and $f\in \mcg$. This proves assertion (ii).

We  prove (iii).
For any graph clasper  $G$ of degree $2$ in $(\Sigma\times I)$, 
we have $\widetilde{Z}^Y_2\left( (\Sigma\times I)_{G} \right)=\pm D$ 
where the Jacobi diagram $D$ has the same shape as $G$ \cite{CHM}. 
Thus we have 
$$
\widetilde{Z}^Y_2\left( Y_2\cyl  \right) \subset \jacobi^c_2(L^\pm) \subset \jacobi^c_2(L^\pm_\Q).
$$  
As seen in \S \ref{subsec:Jacobi}, we also have 
$s\circ \varphi\left( \jacobi^c_2(L^\pm) \right) = \jacobi^{<,c}_2(H)$, 
and it follows from the formula for $\chi^{-1}$ given in \cite{HM_SJD} that 
$\chi^{-1}\left( \jacobi^{<,c}_2(H) \right)$ is contained in $\frac{1}{4} \jacobi^{c}_2(H)$. 
Thus, we have
$$
Z_2(Y_2 \cyl)\subset \frac{1}{4} \jacobi^{c}_2(H) \subset \jacobi^{c}_2(H_\Q)
$$
so that $d''(Y_2 \cyl)$ is contained in $\frac{1}{4}\Z$.
We conclude thanks to (\ref{eq:half_d''}) that $d''(\Jcob)\subset \frac{1}{8}\Z$.
\end{proof}

We are now going to express $d''$ in terms of the classical invariants $\lambda_j$, $\alpha$ and $\tau_2$.
Here the Heegaard embedding $j:\Sigma \to S^3$ is chosen compatibly with the system of meridians and parallels $(\alpha,\beta)$ 
as explained in the paragraph preceding Lemma \ref{lem:Casson_LMO}.
We shall denote by 
$\langle \cdot , \cdot \rangle : L^\pm\times L^\pm \rightarrow \Z$ the symmetric bilinear form defined by 
$$
\forall i,j\in \{1,\dots,g\}, \quad
 \langle i^+ , j^- \rangle := \delta_{i,j}\quad,\quad \langle i^+ , j^+ \rangle := 0 \quad,\quad \langle i^- , j^- \rangle := 0 
$$
and, for all $a,b,c,d\in L^\pm$, we set 
\begin{eqnarray*}
 \operatorname{H}_\Phi\big(\hn{a}{b}{c}{d}\big) & := &  
 \langle a,d \rangle\phin{b}{c} + \langle b,c \rangle\phin{a}{d} -\langle a,c \rangle\phin{b}{d}  -\langle b,d \rangle\phin{a}{c}, \\
 \operatorname{H}_\Theta\big(\hn{a}{b}{c}{d}\big) & := &  
 \big( \langle a,d \rangle\langle b,c \rangle -\langle a,c \rangle\langle b,d \rangle\big) \thetagraph, \\
 \Phi_\Theta\big(\phin{a}{b}\big) & := &  \langle a,b \rangle \thetagraph.  
\end{eqnarray*}

\begin{lemma}\label{lem:formula_d''}
For all $M\in \Jcob$, we have
\begin{equation}\label{eq:formula_d''}
d''(M) = -\frac{\lambda_j(M)}{2} - \frac{1}{4}\Phi_\Theta(\alpha(M)) - \frac{1}{4}\operatorname{H}_\Theta(\tau_2(M)).
\end{equation}
\end{lemma}

\noindent
Here, and in the sequel, a tacit identification between $\jacobi^c(L^\pm_\Q)$ and $\jacobi^c(H_\Q)$ 
is always through the ``obvious'' isomorphism that transforms $L^\pm_\Q$-colored diagrams  
into $H_\Q$-colored diagrams by the rules (\ref{eq:L+-_H}).
Thus the second Johnson homomorphism
$$
\tau_2(M) \in  
\frac{\left(\Lambda^2 H \otimes \Lambda^2H\right)^{\mathfrak{S}_2}}{\Lambda^4H} \subset 
\frac{\left(\Lambda^2 H_\Q \otimes \Lambda^2H_\Q\right)^{\mathfrak{S}_2}}{\Lambda^4H_\Q} 
\simeq \frac{S^2 \Lambda^2H_\Q}{\Lambda^4H_\Q} 
\simeq \jacobi^c_{2,0}(H_\Q)
$$ 
is seen as an element of $\jacobi^c_{2,0}(L^\pm_\Q)$, 
while the quadratic part of the relative RT torsion
$$
\alpha(M) \in S^2H \simeq  \jacobi^c_{2,1}(H)
$$ 
is interpreted as an element of $\jacobi^c_{2,1}(L^\pm)$.

\begin{proof}[Proof of Lemma \ref{lem:formula_d''}]
The ``non-obvious'' isomorphism $\kappa:\jacobi^c(L^\pm_\Q)\rightarrow \jacobi^c(H_\Q)$ defined by (\ref{eq:kappa_formula})
is given in degree $2$ by the formulas
\begin{eqnarray*}
 \kappa \big(\hn{a}{b}{c}{d}\big) & = & 
 -\hn{a}{b}{c}{d} -\frac{1}{2} \operatorname{H}_\Phi\big(\hn{a}{b}{c}{d}\big) - \frac{1}{4} \operatorname{H}_\Theta\big(\hn{a}{b}{c}{d}\big),\\
 \kappa \big(\phin{a}{b}\big) & = & \phin{a}{b} + \frac{1}{2} \Phi_\Theta\big(\phin{a}{b}\big), \\
 \kappa (\thetagraph ) & = & -\thetagraph, 
\end{eqnarray*}
where the labels on the right-hand side of these equalities are understood as elements of $H_\Q$ by the rules  (\ref{eq:L+-_H}).
Using these formulas, we obtain that $d''=p_{2,2}\circ \kappa\circ \widetilde{Z}^Y_2$ is given, for any $M\in \Jcob$, by 
\begin{equation}
\label{eq:d''}
 d''(M) = - p_{2,2}\circ \widetilde{Z}^Y_2(M)  + \frac{1}{2}\Phi_\Theta\big(p_{2,1}\circ \widetilde{Z}^Y_2(M)\big) 
          - \frac{1}{4}\operatorname{H}_\Theta\big(p_{2,0}\circ \widetilde{Z}^Y_2(M)\big).
\end{equation}
By Lemma \ref{lem:tau2_LMO}, we know that 
$$
p_{2,0}\circ \widetilde{Z}^Y_2(M)=-p_{2,0}\circ Z_2(M) = -\tau_2(M),
$$ 
and by Lemma \ref{lem:alpha_LMO}, we have that
$$ 
\alpha(M) = -2 p_{2,1}\circ Z_2(M) = -2 \left( p_{2,1}\circ \widetilde{Z}^Y_2(M) -\frac{1}{2}\operatorname{H}_\Phi\big(p_{2,0}\circ \widetilde{Z}^Y_2(M)\big) \right). 
$$
It follows that 
$$ 
p_{2,1}\circ \widetilde{Z}^Y_2(M) = -\frac{1}{2} \alpha(M) + \frac{1}{2} \operatorname{H}_\Phi(p_{2,0}\circ \widetilde{Z}^Y_2(M)\big)
= -\frac{1}{2} \alpha(M) - \frac{1}{2} \operatorname{H}_\Phi(\tau_2(M)). 
$$ 
Finally, recall from (\ref{eq:cassonLMO}) that the coefficient of $p_{2,2}\circ \widetilde{Z}^Y_2(M)$ is $\frac{1}{2}\lambda_j(M)$.  
The result follows from  these interpretations of $p_{2,i}\circ \widetilde{Z}^Y_2(M)$ for $i=0,1,2$ 
and equation (\ref{eq:d''}).
\end{proof}

Recall from Proposition \ref{prop:alpha_properties} that $\alpha\circ \mcyl: \Torelli \to S^2H$ is trivial.
Therefore, equation (\ref{eq:formula_d''}) gives another decomposition of $\lambda_j$ on the subgroup $\Johnson$:
\begin{equation}\label{eq:decomposition_d''}
-\lambda_j=  2d''   + \frac{1}{2}\operatorname{H}_\Theta\circ \tau_2: \Johnson \longrightarrow \Z.
\end{equation}
This identity should be compared to Morita's decomposition (\ref{eq:decomposition_d}).
Whereas $d$ depends quadratically on the genus of  bounding simple closed curves (\ref{eq:core_bscc}),
the next lemma shows that  $d''$ depends linearly on the genus. 
Thus the decomposition (\ref{eq:decomposition_d''}) is essentially Auclair's formula \cite[Theorem 4.4.6]{Auclair}.

\begin{lemma}\label{lem:d''_bscc}
Let $\gamma\subset \Sigma$ be a bounding simple closed curve of genus $h$,
and let $T_\gamma$ denote the Dehn twist along $\gamma$. 
Then we have $d''(T_\gamma)=-h/8$. 
\end{lemma}

\begin{proof}
Since $d''$ is  $\mcg$-conjugacy invariant,  
we can assume without loss of generality that the curve $j(\gamma)$ 
bounds a disk in the lower handlebody of the genus $g$ Heegaard splitting of $S^3$.
Thus the $3$-manifold $S^3(\mcyl(T_\gamma),j)$ is obtained from $S^3$ by surgery
along a $(+1)$-framed trivial knot, so that it is homeomorphic to $S^3$.  
Hence we have $\lambda_j(T_\gamma)=0$, and we deduce from Lemma \ref{lem:formula_d''} that
$$
d''(T_\gamma) = -\frac{1}{4}\operatorname{H}_\Theta(\tau_2(T_\gamma)).
$$
By \cite[Proposition 1.1]{Morita_Casson1}, we have 
(taking into account the difference in sign conventions) 
\begin{equation}\label{eq:tau2BSCC}
\tau_2(T_\gamma) = \frac{1}{2}\sum_{i=1}^h\hn{\alpha_i}{\beta_i}{\alpha_i}{\beta_i} 
+ \sum_{1\leq i<j\leq h} \hn{\alpha_i}{\beta_i}{\alpha_j}{\beta_j}. 
\end{equation}
The result follows from the observations that
$$
\operatorname{H}_\Theta\big( \hn{\alpha_i}{\beta_i}{\alpha_i}{\beta_i} \big) 
= \thetagraph \quad\textrm{and}\quad \operatorname{H}_\Theta\big( \hn{\alpha_i}{\beta_i}{\alpha_j}{\beta_j} \big) = 0
$$
for all $i\neq j$.
\end{proof}

We  also need the homomorphism
$\overline{d'}:  \DD_2(H) \simeq \left(\Lambda^2 H \otimes \Lambda^2H\right)^{\mathfrak{S}_2}/\Lambda^4H\rightarrow \Z$ 
defined by Morita \cite{Morita_Casson2} in the following way:
$$
\left\{\begin{array}{lcl}
 \overline{d'}\left( (a\wedge b)\otimes (a\wedge b)\right) & := & -3 \omega(a,b)^2\\
\overline{d'}\left( (a\wedge b)\leftrightarrow (c\wedge d) \right) &:= &
-4 \omega(a,b)\omega(c,d) - 2 \omega(a,c)\omega(b,d) + 2 \omega(a,d)\omega(b,c).
\end{array}\right.
$$
We get a monoid homomorphism
$$ 
d': \Jcob \longrightarrow \Z 
$$
by  setting $d':=-\overline{d'}\circ \tau_2$. 
Observe that $d'$ shares the same properties as $d''$: 
it is canonical and it is invariant under $Y_3$-equivalence 
as well as under the action of $\mcg$ by conjugation. 
A simple computation based on (\ref{eq:tau2BSCC}) gives
\begin{equation}
\label{eq:d'_bscc}
d'(T_\gamma)=h(2h+1),
\end{equation}
for any bounding simple closed curve $\gamma\subset \Sigma$ of genus $h$ \cite{Morita_Casson2}.
Morita proved that, in genus $g\geq 2$, any  $\mcg$-conjugacy invariant 
group homomorphism $\Johnson \to \Z$ can be written in a unique way as a linear combination  
(with rational coefficients) of $d$ and $d'|_\Johnson$ \cite{Morita_Casson2}.
The next lemma expresses 
$8d''|_\Johnson$ in this way.

\begin{lemma}\label{lem:d_d'_d''}
We have
\begin{equation}
\label{eq:d_d'_d''}
8d''|_{\Johnson} = \frac{1}{6}d - \frac{1}{3}d'|_{\Johnson}. 
\end{equation}
\end{lemma}

\begin{proof}
Equation (\ref{eq:d_d'_d''}) can be deduced from (\ref{eq:decomposition_d}),
(\ref{eq:decomposition_d''}) and the definition of $d'$ by a direct computation.
Alternatively, we can use the fact that $\Johnson$ is generated by Dehn twists along 
bounding simple closed curves \cite{Johnson_subgroup}. 
Let $\gamma \subset \Sigma$ be a bounding simple closed curve of genus $h$.
Equations (\ref{eq:core_bscc}) and (\ref{eq:d'_bscc}) give
$$
\frac{1}{6}d(T_\gamma) - \frac{1}{3}d'(T_\gamma)
= \frac{1}{6} \cdot 4h(h-1) - \frac{1}{3}\cdot h(2h+1)=-h
$$
and we conclude thanks to Lemma \ref{lem:d''_bscc}.
\end{proof}

To prove the existence in Theorem~D, we define the monoid map $d:\Jcob \to \Z$ by
\begin{equation}\label{eq:definition_extended_core}
d:= 2d'+48 d''.
\end{equation}
According to Lemma \ref{lem:d''}, this map $d$ is $\mcg$-conjugacy invariant 
as well as $Y_3$-invariant. 
According to Lemma \ref{lem:d_d'_d''}, it extends Morita's map $d:\Johnson \to \Z$ through $\mcyl$.
The invariant $d:\Jcob \to \Z$ can be written explicitly in terms of $\lambda_j$, $\alpha$ and $\tau_2$ 
using Lemma \ref{lem:formula_d''}:
\begin{equation}
\label{eq:formula_extended_core}
\forall M \in \Jcob, \quad d(M) = 
- 24\lambda_j(M) - 12 \Phi_\Theta(\alpha(M)) -12 \operatorname{H}_\Theta(\tau_2(M)) -2 \overline{d'}(\tau_2(M)).
\end{equation}
Besides, it can be written explicitly in terms of $Z$ using Lemma \ref{lem:tau2_LMO}:
\begin{equation}
\label{eq:formula_extended_core_bis}
\forall M \in \Jcob, \quad d(M) = -2 \overline{d'}\circ p_{2,0}\circ Z_2(M) + 48 p_{2,2}\circ Z_2(M).
\end{equation}
It remains to prove that $d(\Jcob)$ is contained in $8\Z$.
For this, we recall from  \S \ref{subsec:description1} that we have an isomorphism
$
\psi_{[2]}: \jacobi_{[2]}^{<,c}(H)\to \Jcob/Y_{3}
$
and we shall actually compute  $d\circ \psi_{[2]}(D)$ for each generator $D$ of $\jacobi^{<,c}_{[2]}(H)$.

\begin{proposition}\label{prop:values_d}
The monoid homomorphism $d:\Jcob \to \Z$ takes the following values on the generators of the group $\Jcob/Y_{3}$:
\begin{eqnarray}
\label{eq:line1}  d\circ \psi_{[2]} \left(\thetagraph \right) & = & 48, \\ 
\label{eq:line2}  d\circ \psi_{[2]} \left(\phio{a}{b} \right) & = & 24\omega(a,b), \\
\label{eq:line3}  \quad \quad  \quad d\circ \psi_{[2]} \left(\hob{a}{b}{c}{d} \right) 
& = & 16\omega(a,b)\omega(c,d) - 16 \omega(a,c)\omega(b,d) - 8\omega(a,d)\omega(b,c),\\
\label{eq:line4}  d\circ \psi_{[2]}(h,h')&=&12\omega(h,h')(\omega(h,h')-1),\\
\label{eq:line5} d\circ \psi_{[2]}(h)&=&0,\\
\label{eq:line6} d\circ \psi_{[2]}(1)&=&-24, 
\end{eqnarray}
\end{proposition}

\begin{proof}
Since $ Z_2 \circ \psi_2 = \chi^{-1}$, we obtain that
\begin{eqnarray*}
Z_2\circ \psi_2 \left(\thetagraph \right) & = & \thetagraph \\
Z_2 \circ \psi_2 \left(\phio{a}{b} \right) & = & \phin{a}{b} + \frac{1}{2}\omega(a,b)\thetagraph \\
Z_2\circ \psi_2 \left(\hob{a}{b}{c}{d} \right) & = & \hn{b}{c}{d}{a} + \frac{1}{4}\left( \omega(a,b)\omega(c,d) - \omega(a,c)\omega(b,d) \right)\thetagraph\\
         & &+  \frac{1}{2}\left(\begin{array}{c} \omega(a,b) \phin{c}{d} - \omega(a,c) \phin{b}{d}\\
               + \omega(c,d) \phin{a}{b} - \omega(b,d) \phin{a}{c} \end{array}\right).
\end{eqnarray*} 
We  deduce from (\ref{eq:formula_extended_core_bis}) the formulas (\ref{eq:line1}), (\ref{eq:line2}) and (\ref{eq:line3}).
Next,   since $d$ is additive, formula (\ref{eq:line4}) is derived from equations (\ref{eq:line2}) and (\ref{eq:line3}) using relation ($G_1$).
Then (\ref{eq:line4}) and relation $(G_0)$ imply  (\ref{eq:line5}).
Finally,  (\ref{eq:line1}) and relation $(G_3)$ give (\ref{eq:line6}). 
\end{proof}

\subsection{Proof of the unicity in Theorem~D}

We need some representation theory of the symplectic group  $\Sp(H_\Q)\simeq \Sp(2g;\Q)$.
In particular, we need the following facts. 

\begin{lemma}\label{lem:Z_Q}
Let $V$ be a finite-dimensional rational $\Sp(2g;\Q)$-module.
\begin{itemize}
\item[(i)] If $L$ is an abelian subgroup of $V$ stable by the action of $\Sp(2g;\Z)$,
then $L\otimes \Q \subset V$ is stable by the action of $\Sp(2g;\Q)$.
\item[(ii)] If $L$ is a lattice of $V$ stable by the action of $\Sp(2g;\Z)$, 
then the action of $\Sp(2g;\Q)$ on $V=L\otimes \Q$ is determined by the action of $\Sp(2g;\Z)$ on $L$.
\item[(iii)] If $L$ is  a lattice of $V$ stable by the action of $\Sp(2g;\Z)$
and if $f:L\to \Z$ is an $\Sp(2g;\Z)$-invariant group homomorphism,
then $f\otimes \Q:V \to \Q$ is $\Sp(2g;\Q)$-invariant. 
\end{itemize}
\end{lemma}

\noindent
These facts may belong to folklore. 
Statement (i) is proved by Asada and Nakamura in \cite[(2.2.8)]{AN}.
Statements (ii), (iii) can be proved using the same kind of arguments.

We shall assume in the sequel that $g\geq 3$.
We denote
$$
Y_3\Torelli:= 
\ker\left(\xymatrix{ 
\Torelli \ar[r]^-\mcyl & \cyl \ar@{->>}[r] & \cyl/Y_3} \right)
$$
which is a subgroup of $\Torelli$ sitting between $\Gamma_3\Torelli$ and $\Gamma_2 \Torelli$.

\begin{lemma}
\label{lem:Sp}
The action of $\mcg$ by conjugation on $\Jcob$ (respectively on $\Johnson$) 
induces an action of $\Sp(H) \simeq \mcg/\Torelli$ on the abelian group $\Jcob/Y_3$
(respectively on $\Johnson/Y_3 \Torelli$), 
and this extends in a unique way to an action of $\Sp(H_\Q)$ on 
the vector space $(\Jcob/Y_3)\otimes \Q$ 
(respectively on $(\Johnson/Y_3 \Torelli)\otimes \Q$).
Under the assumption that $g\geq 3$,  the mapping cylinder construction induces an isomorphism
$$
\left( \frac{\Johnson}{Y_3\Torelli}\otimes \Q\right)^{\Sp(H_\Q)}
\mathop{\longrightarrow}_\simeq^\mcyl
\left( \frac{\Jcob}{Y_3}\otimes \Q\right)^{\Sp(H_\Q)}
$$
between the $\Sp(H_\Q)$-invariant subspaces.
\end{lemma}

\begin{proof}
Let $f\in \Torelli$ and $M\in \Jcob$.
Since $\Jcob/Y_3$ is contained in the center of the group $\cyl/Y_3$ by Lemma \ref{lem:Habiro_sec},
the cobordism $\mcyl(f)\circ M \circ \mcyl(f^{-1})$ is $Y_3$-equivalent to $M$.
Therefore the action of $\mcg$ on $\Jcob/Y_3$ factorizes to $\mcg/\Torelli\simeq \Sp(H)$. 
Regarding $\Johnson/Y_3\Torelli$ as a subgroup of $\Jcob/Y_3$ via $\mcyl$,
we see that the same is true for the action of $\mcg$ on $\Johnson/Y_3\Torelli$.

It follows from \cite{MM} that the quotient group $(\Jcob/Y_3)/(Y_2\cyl/Y_3)\simeq \Jcob/Y_2$ 
is $2$-torsion, so that the inclusion $Y_2\cyl \subset \Jcob$ induces an $\Sp(H)$-equivariant isomorphism
$$
(Y_2\cyl/Y_3)\otimes \Q \mathop{\longrightarrow}^\simeq (\Jcob/Y_3)\otimes \Q.
$$
Since the action of $\Sp(H)$ on $Y_2\cyl/Y_3$ extends to an action of 
$\Sp(H_\Q)$ on $(Y_2\cyl/Y_3)\otimes \Q$ \cite{HM_SJD}, 
the action of $\Sp(H)$ on $\Jcob/Y_3$ extends to an action 
of $\Sp(H_\Q)$ on $(\Jcob/Y_3)\otimes \Q$ and this extension is unique by (ii) of Lemma \ref{lem:Z_Q}.
The same is true for the action of $\Sp(H)$ on $\Johnson/Y_3\Torelli$ by (i) and (ii) of Lemma \ref{lem:Z_Q}.

To prove the last statement, we use the  commutative diagram of $\Sp(H_\Q)$-modules
$$
\xymatrix{
\frac{\Gamma_2\Torelli}{\Gamma_3\Torelli} \otimes \Q \ar@{->>}[r] \ar@{>->}[rd]_-\mcyl& 
\frac{\Gamma_2\Torelli}{Y_3\Torelli} \otimes \Q \ar@{->}[d]^-\mcyl \ar[r]^-\simeq
& \frac{\Johnson}{Y_3\Torelli} \otimes \Q \ar@{->}[d]^-\mcyl\\
& \frac{Y_2\cyl}{Y_3}\otimes \Q \ar[r]^-\simeq&  \frac{\Jcob}{Y_3}\otimes \Q 
}
$$
where the horizontal maps are induced by inclusions 
and verticals maps are induced by the mapping cylinder construction.
The injectivity of the oblique map is proved in \cite[Corollary 1.6]{HM_SJD} assuming that $g\geq 3$.
The bijectivity of $({\Gamma_2\Torelli}/{Y_3\Torelli})\otimes \Q  \to ({\Johnson}/{Y_3\Torelli}) \otimes \Q$
follows from the fact that $\Johnson/\Gamma_2\Torelli$ is $2$-torsion \cite{Johnson_abelianization}.
Passing to the $\Sp(H_\Q)$-invariant subspaces, 
we obtain the following commutative diagram of vector spaces:
$$
\xymatrix{
\left(\frac{\Gamma_2\Torelli}{\Gamma_3\Torelli} \otimes \Q\right)^{\Sp(H_\Q)}
 \ar@{->}[d]_-\mcyl \ar[r]^-\simeq \ar[r]^-\simeq
& \left(\frac{\Johnson}{Y_3\Torelli} \otimes \Q\right)^{\Sp(H_\Q)} \ar@{->}[d]^-\mcyl\\
 \left(\frac{Y_2\cyl}{Y_3}\otimes \Q\right)^{\Sp(H_\Q)} \ar[r]^-\simeq&  
\left(\frac{\Jcob}{Y_3}\otimes \Q\right)^{\Sp(H_\Q)}
}
$$
The decompositions into irreducible $\Sp(H_\Q)$-modules done in \cite[\S 5]{HM_SJD} 
show that the first vertical map is an isomorphism. The conclusion follows.
\end{proof}

We can now prove the unicity in Theorem~D assuming that $g\geq 3$.
We deduce from Morita's formula (\ref{eq:decomposition_d})
that the group homomorphism $d:\Johnson \to \Z$ vanishes on $Y_3\Torelli$.  
Therefore we have a linear map 
$$
d\otimes\Q: ({\Johnson}/{Y_3\Torelli})\otimes \Q \longrightarrow \Q.
$$
Since $d$ is  $\mcg$-conjugacy invariant,  
we deduce from Lemma \ref{lem:Z_Q} (iii) 
that  $d\otimes \Q$ is $\Sp(H_\Q)$-invariant. We have the commutative diagram
$$
\xymatrix{
\Hom_{\Sp(H_\Q)}\left(\frac{\Jcob}{Y_3}\otimes \Q,\Q\right) \ar[r]^-\simeq \ar[d] & 
 \Hom_\Q\left(\left(\frac{\Jcob}{Y_3}\otimes \Q\right)^{\Sp(H_\Q)},\Q\right) \ar[d] \\
\Hom_{\Sp(H_\Q)}\left(\frac{\Johnson}{Y_3\Torelli}\otimes \Q,\Q\right) \ar[r]^-\simeq & 
\Hom_\Q\left(\left(\frac{\Johnson}{Y_3\Torelli}\otimes \Q\right)^{\Sp(H_\Q)},\Q\right) 
}
$$
where the horizontal maps are restrictions and are isomorphisms according to Schur's lemma;
the vertical maps of that diagram are induced by $\mcyl$.
We deduce from Lemma \ref{lem:Sp} that $d\otimes \Q$ 
extends in a unique way through $\mcyl$ 
to an $\Sp(H_\Q)$-invariant linear map $(\Jcob/Y_3)\otimes \Q \to \Q$.
According to Lemma \ref{lem:Z_Q} (iii), an  $\Sp(H_\Q)$-invariant linear map $(\Jcob/Y_3)\otimes \Q \to \Q$
is the same as an 
$\mcg$-conjugacy invariant
and $Y_3$-invariant monoid map $\Jcob\to \Q$.
This proves the unicity statement in Theorem~D (as well as the existence statement 
if one allows values in $\Q$ instead of $8\Z$).

\subsection{A stable version of Theorem D}

Whatever the genus $g\geq 0$ of $\Sigma$ is,  formula (\ref{eq:formula_extended_core}) defines 
an  $\mcg$-conjugacy invariant 
and $Y_3$-invariant monoid map $d:\Jcob \to 8\Z$ 
which extends Morita's map on the subgroup $\Johnson$. 
The invariants $\lambda_j$, $\alpha$ and $\tau_2$
being  preserved by stabilization of the surface $\Sigma$, 
the homomorphism $d$ is invariant under stabilization.
Thus we can summarize the results of this section in the following way. 

\begin{theorem}
There is a unique way to define, 
for each compact connected oriented surface $\Sigma$ with one boundary component,
a monoid homomorphism $d: \Jcob(\Sigma) \to \Z$ that has the following properties:
\begin{itemize}
\item $d$ is $Y_3$-invariant and $\mcg(\Sigma)$-conjugacy invariant;
\item $d\circ \mcyl:\Johnson(\Sigma) \to \Z$ coincides with Morita's core of the Casson invariant;
\item $d$ is preserved under stabilization of the surface $\Sigma$ as shown in Figure \ref{fig:stabilization}:
$$
\xymatrix{
\Jcob(\Sigma) \ar@{>->}[r]\ar[rd]_-d  & \Jcob(\Sigma^s)\ar[d]^-d\\
& \Z
}
$$
\end{itemize}
Moreover, this homomorphism $d$ takes values in $8\Z$ and is given by
$$
\forall M \in \Jcob(\Sigma), \quad d(M) = - 24\lambda_j(M) - 12 \Phi_\Theta(\alpha(M)) -12 \operatorname{H}_\Theta(\tau_2(M)) -2 \overline{d'}(\tau_2(M)).
$$
\end{theorem}

\appendix
\section{Some results from calculus of claspers} \label{sec:calc_clasper}

In this appendix we have collected several results of calculus of claspers which are used in the paper.
In particular, \S \ref{subsec:special_calc_clasper} contains a list of  lemmas
dealing with surgeries along $Y$-graphs that have ``special'' leaves.
Several of these technical results were previously established, 
in a stronger form, by Auclair in \cite{Auclair}.   
(Note that his results are stated   at the level  of the Goussarov--Habiro filtration by finite-type invariants, 
rather than at the level of the $Y$-filtration of homology cylinders.
Moreover,  Auclair adopts in \cite{Auclair} Goussarov's convention 
for claspers  \cite{Goussarov_graphs,GGP} instead of Habiro's convention \cite{Habiro} used here.)

The definition of claspers is recalled  in \S \ref{subsec:clasper},
but it will be convenient to use in this appendix a slightly more general definition. 
Specifically, we shall fully use Habiro's original definition, 
where claspers are decomposed into edges, leaves, nodes and \emph{boxes}  \cite{Habiro}.
A box is a disk with three edges attached, one being distinguished from the other two.  This distinction is done by drawing a box 
as a rectangle, as in Figure \ref{fig:box}.  
In the process of replacing a clasper with a disjoint union of basic claspers, as explained in \S \ref{subsec:clasper}, 
each box is replaced with three leaves as depicted in Figure \ref{fig:box}.
Connected claspers without boxes are called \emph{graph claspers} and are presented in \S \ref{subsec:clasper}.
The drawing convention for claspers are those of \cite[Figure 7]{Habiro}, except for the following: 
a $\oplus$ (respectively a $\ominus$) on an edge represents a positive (respectively negative) half-twist. 
(This replaces the convention of a circled $S$ (respectively $S^{-1}$) used in \cite{Habiro}.)

\begin{figure}[!h]
\includegraphics{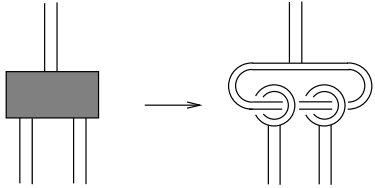}
\caption{How to replace a box with three leaves. }
\label{fig:box}
\end{figure}
\subsection{Standard calculus of claspers}\label{subsec:standard_calc_clasper}
In \cite[Proposition 2.7]{Habiro}, Habiro gives a list of 12 moves on claspers which give \emph{equivalent} claspers,
that is claspers with homeomorphic surgery effect.  
We will freely use \emph{Habiro's moves} (which are derived from Kirby calculus) 
by referring to their numbering in Habiro's paper.  

We start by recalling some basic lemmas of calculus of claspers.   
The proofs are omitted, as they use the same techniques as in \cite[\S 4]{Habiro} 
and \cite{Goussarov_graphs,GGP} (where similar statements appear).  
See also Appendix E of \cite{Ohtsuki_book}.  

\begin{lemma}\label{lem:edgeslide}
Let $G$ be a degree $k$ graph clasper in a $3$-manifold $M$, 
and let $K$ be some framed knot in $M$, disjoint from $G$.  
Let $G'$ be obtained from $G$ by a connected sum, along some band $b$, of an edge with the knot $K$.  
Then we have
$$ 
M_G\stackrel{Y_{k+2}}{\sim} M_{G'\cup H}, 
$$
where $H$ is a 
copy, disjoint from $G'$, of the degree $k+1$ graph clasper $G\cup b\cup K$.  
\end{lemma}

\begin{lemma} \label{lem:crossingchange}
Let $T_1\cup T_2$ be a disjoint union of two graph claspers of degree $k_1$ and $k_2$ respectively in a $3$-manifold $M$.     
Let $T'_1\cup T'_2$ be obtained by a crossing change of a leaf $f_1$ of $T_1$ with a leaf $f_2$ of $T_2$.  
Then we have
$$ 
M_{T_1\cup T_2}\stackrel{Y_{k_1+k_2+1}}{\sim}  M_{T'_1\cup T'_2\cup T}, 
$$ 
where $T$ is a copy, disjoint from $T'_1\cup T'_2$, of some graph clasper 
of degree $k_1+k_2$ obtained from $T_1\cup T_2$ by connecting the edges incident to $f_1$ and $f_2$.  
\end{lemma}

It is convenient to state the next two lemmas in the case of graph claspers in $(\Sigma\times I)$.  

\begin{lemma} \label{lem:split}
Let $G$ be a degree $k$ graph clasper in $(\Sigma\times I)$. 
Let $f_1$ and $f_2$ be the two framed knots obtained by splitting a leaf $f$ of $G$ 
along an arc $\alpha$, \ie we have $f\cup \alpha = f_1\cup f_2$ 
and $f_1\cap f_2=\alpha$ (see e.g. Figure \ref{fig:scind}). 
Then we have
$$ 
(\Sigma\times I)_G \stackrel{Y_{k+1}}{\sim} (\Sigma\times I)_{G_1}\circ (\Sigma\times I)_{G_2}, 
$$ 
where $G_i$ denotes the graph clasper obtained from $G$ by replacing $f$ by $f_i$.  
\end{lemma}

\begin{lemma} \label{lem:twist_as}
Let $G$ be a degree $k$ graph clasper in $(\Sigma\times I)$. 
Let $G'$ be a clasper which differs from $G$ only by a half-twist on an edge, 
and let $\overline{G}$ be obtained from $G$ by reversing the cyclic order on the attaching regions of the three edges at one node.  
Then we have
$$ 
(\Sigma\times I)_G\circ (\Sigma\times I)_{G'} \stackrel{Y_{k+1}}{\sim} (\Sigma\times I)_G\circ (\Sigma\times I)_{\overline{G}} 
\stackrel{Y_{k+1}}{\sim}(\Sigma\times I).  
$$
\end{lemma}

\noindent
The latter part of Lemma \ref{lem:twist_as} is a version of the AS relation for claspers.  
Claspers also satisfy relations analogous to the IHX and STU relations, 
which we shall not recall here.  

\subsection{Refined calculus of claspers for $Y$-graphs with special leaves}\label{subsec:special_calc_clasper}

Let $G$ be a clasper in a $3$-manifold $M$ and let $m\in \Z$.
An \emph{$m$-special leaf of $G$} is a leaf $f$ of $G$ that bounds a disk in $M$ 
with respect to which it is $m$-framed, and such that this disk is disjoint from $G\setminus f$.
Claspers with $m$-special leaves have been studied in \cite{Meilhan}.  
A $0$-special leaf of $G$ is usually called a \emph{trivial leaf}.   
If a graph clasper
$G$ in $M$ contains a trivial leaf, then $M_G$ is homeomorphic to $M$ \cite{Habiro,GGP}.
A $(-1)$-special leaf is simply called a \emph{special leaf}.
Recall the following result from \cite{Meilhan}, see also \cite{Auclair,GGP}.

\begin{lemma}\label{lem:special}
Let $G$ be a graph clasper of degree $k\ge 2$ with an $m$-special leaf in a $3$-manifold $M$ (where $m\in \Z$).  
Then we have
$M_G\stackrel{Y_{k+1}}{\sim} M.$  
\end{lemma}

\noindent
This result, however, is not true in degree $k=1$ for $m$ odd, and in particular for $Y$-graphs with special leaves.  
The purpose of this subsection is to provide refinements of the results of \S \ref{subsec:standard_calc_clasper} 
for $Y$-graphs with special leaves.  Before we do so, let us prepare a few auxiliary results.

In the next statement, the figure represents claspers in a given $3$-manifold 
which are identical outside a $3$-ball, where they are as depicted.  

\begin{lemma} \label{lem:slide}
The moves $(a)$ and $(b)$ of Figure \ref{fig:move1} produce equivalent claspers.  
\begin{figure}[!h]
{\labellist \small \hair 0pt 
\pinlabel {$(a)$} [l] at 52 26
\pinlabel {$(b)$} [l] at 221 28
\endlabellist}
\includegraphics[scale=1]{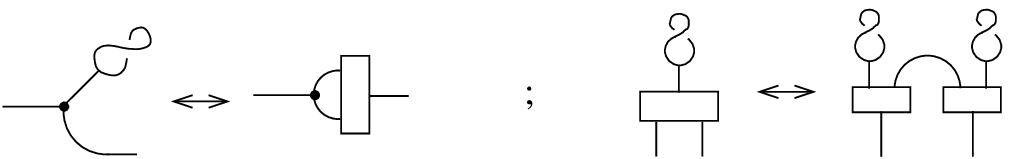}
\caption{The moves $(a)$ and $(b)$. } \label{fig:move1}
\end{figure}
\end{lemma}

\begin{proof}
Move $(a)$ is an immediate consequence of \cite[Theorem 3.1]{GGP} 
(taking into account the fact that the convention used in \cite{GGP} for 
the definition of the surgery along a clasper is the opposite 
of the one used in \cite{Habiro} and  the present paper).

Now we consider the clasper on the left-hand side
of the figure illustrating move $(b)$.
By Habiro's move 5 and move 7, this clasper is equivalent to
$$
{\labellist \small \hair 0pt 
\pinlabel {$=$}  at 32 25
\pinlabel {$\sim$}  at 106 25
\pinlabel {.}  [l] at 167 25
\endlabellist}
\includegraphics[scale=1]{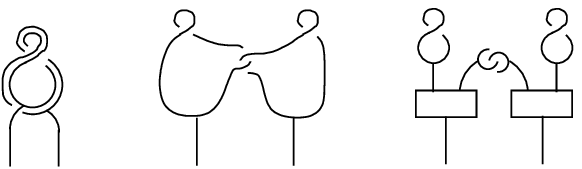}
$$
The proof then follows from Habiro's move 2.  
\end{proof}

The next lemma is in some sense 
a continuation of move $(b)$ for $Y$-graphs.

\begin{lemma}\label{lem:movec}
Let $G$ be a clasper which ``consists of'' two $Y$-graphs in $(\Sigma\times I)$, 
sharing a special leaf via a box as depicted in Figure \ref{fig:movec}. 
Then we have
$$ 
(\Sigma\times I)_G \stackrel{Y_{3}}{\sim}  (\Sigma\times I)_{Y'}\circ  (\Sigma\times I)_{Y''}\circ  (\Sigma\times I)_{H}, 
$$ 
where  $Y'$, $Y''$ and $H$ are the claspers represented in Figure \ref{fig:movec}.  
\begin{figure}[!h]
    {\labellist \small \hair 0pt 
    \pinlabel {$G$} [l] at 51 40
    \pinlabel {$Y'$} [l] at 102 40
    \pinlabel {$Y''$} [l] at 156 40
    \pinlabel {$H$} [l] at 215 40
    \pinlabel {$G'$} [l] at 366 40
    \endlabellist}
  \includegraphics{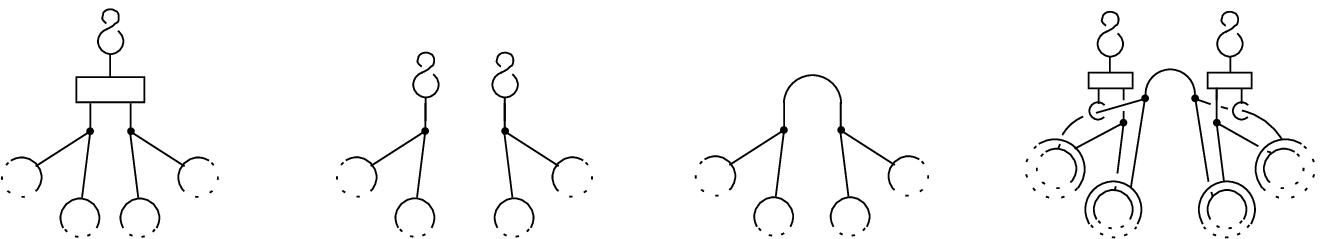} 
  \caption{ }\label{fig:movec}
 \end{figure}
\end{lemma}

\begin{rem}
In Figure \ref{fig:movec}, it is implicit that the leaves of $Y'\cup Y''$ (respectively $H$ and $G'$) 
are parallel copies of the leaves of $G$.  
In particular, a choice of orientation on the leaves of $Y'\cup Y''$ induces one for $H$ and $G'$. 
Similar comments apply to subsequent figures.   
\end{rem}

\begin{proof}[Proof of Lemma \ref{lem:movec}]
By using move $(b)$, followed by several applications of Habiro's moves 11 and 5, we have that $G$ 
is equivalent to the clasper $G'$ represented on the right-hand side of Figure \ref{fig:movec}.
Observe that one of the three components of $G'$ is a graph clasper of degree $2$,
which we denote by $H'$.
On one hand, $H'$ can be homotoped, by a sequence of 
crossing changes in $(\Sigma\times I)$, to a ``horizontal layer'' $\Sigma\times [-1,-1+\varepsilon]$ 
which is disjoint from the rest of $G'$.
Thus, by applying Lemma \ref{lem:edgeslide} and Habiro's move 3 first,
and by using Lemma \ref{lem:crossingchange} subsequently, we obtain that
$(\Sigma\times I)_G \stackrel{Y_{3}}{\sim}  (\Sigma\times I)_{Y'\cup Y''}\circ  (\Sigma\times I)_{H}$. 
On the other hand, it follows from Lemma \ref{lem:crossingchange} and Lemma \ref{lem:special} that 
$(\Sigma\times I)_{Y'\cup Y''}  \stackrel{Y_{3}}{\sim}  (\Sigma\times I)_{Y'}\circ  (\Sigma\times I)_{Y''}$. 
This concludes the proof.
\end{proof}

Finally, we will need the following.

\begin{lemma}\label{lem:general}
Let $G$ be a disjoint union of two $Y$-graphs in $(\Sigma\times I)$ which only differ by a positive half-twist on an edge, 
as represented in Figure \ref{fig:general}.  Then we have
$$
 (\Sigma\times I)_G\stackrel{Y_3}{\sim} (\Sigma\times I)_{C}, 
$$
where $C$ is a degree $2$ graph  clasper in $(\Sigma\times I)$ as represented in Figure \ref{fig:general}.
  \begin{figure}[!h]
    {\labellist \small \hair 0pt 
    \pinlabel {$G$} [l] at 36 32
    \pinlabel {$C$} [l] at 125 32
    \endlabellist}
    \includegraphics{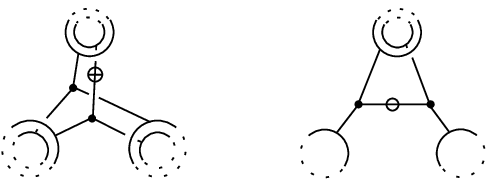}
    \caption{} \label{fig:general}
  \end{figure}
\end{lemma}

\begin{proof}
We have the following equivalences of claspers:
$$
    {\labellist \small \hair 0pt 
    \pinlabel {$\varnothing \sim$} [r] at 0 28
    \pinlabel {$\sim$} [l] at 50 28
    \pinlabel {$\sim$} [l] at 133 28
    \pinlabel {$U$} [l] at 202 35
    \endlabellist}
    \includegraphics{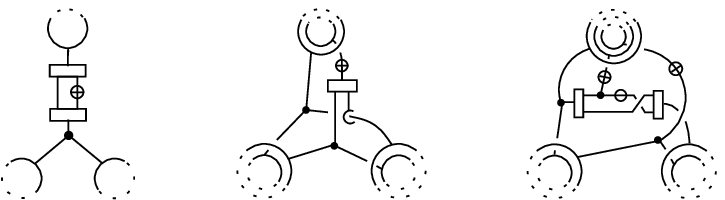}
$$
The first equivalence follows from Habiro's move 4,
and the second one is obtained by applying Habiro's moves 11 and 5.
The third equivalence follows from Habiro's moves 12  and 5: the resulting  clasper is denoted by $U$.
Next, we apply Habiro's move 11 and move 5 several times and, 
using Lemma \ref{lem:edgeslide} and Lemma \ref{lem:crossingchange} in a  way similar to the proof of Lemma \ref{lem:movec}, 
we obtain that 
$$ 
 (\Sigma\times I)_{U}\stackrel{Y_3}{\sim} (\Sigma\times I)_{G}\circ (\Sigma\times I)_{C'}.
$$
Here $C'$ is a degree $2$ graph clasper which only differs from $C$ by three half-twists on edges.  
The result then follows from Lemma \ref{lem:twist_as}.  
\end{proof}

We can now prove the main results of this appendix.  

\begin{lemma}[Edge Sliding]\label{lem:slide_special}
 Let $Y$ be a $Y$-graph with special leaf in a $3$-manifold $M$. 
 Let $Y'$ be obtained from $Y$ by a connected sum, along some band $b$, of an edge $e$ 
with some framed knot in $M$ disjoint from $Y$.  Then we have
  $
  M_Y\stackrel{Y_{3}}{\sim} M_{Y'}. 
  $
\end{lemma}

\begin{proof}
By Lemma \ref{lem:edgeslide}, we have that 
$M_Y\stackrel{Y_3}{\sim} M_{Y'\cup H}$
where $H$ is a degree $2$ graph clasper, disjoint from $Y'$, and with a special leaf.  
By Lemma \ref{lem:crossingchange}, we can  assume that the disk of this special leaf is disjoint from $Y'$.
We conclude thanks to Lemma \ref{lem:special}. 
\end{proof}

\begin{lemma}[Leaf Splitting]\label{lem:split_special}
Let $Y$ be a $Y$-graph with one special leaf in $(\Sigma\times I)$.  
Let $f_1$ and $f_2$ be two framed knots obtained by splitting a leaf $f$ of $Y$ along an arc $\alpha$ as shown 
in Figure \ref{fig:scind}.  
Then we have
$$ 
(\Sigma\times I)_Y 
\stackrel{Y_{3}}{\sim}  (\Sigma\times I)_{Y'}\circ  (\Sigma\times I)_{Y''}\circ  (\Sigma\times I)_{H}, 
$$ 
where $Y_i$ denotes the $Y$-graph obtained from $Y$ by replacing $f$ by $f_i$ ($i=1,2$), 
and where $H$ is the degree 2 graph clasper shown in Figure \ref{fig:scind}.   
\end{lemma}
\begin{figure}[!h]
{\labellist \small \hair 0pt 
\pinlabel {$Y$} [l] at 24 24
\pinlabel {$f$} [tr] at 3 39
\pinlabel {$f'$} [tr] at 50 15
\pinlabel {$\alpha$} [l] at 20 45
\pinlabel {$f_1$} [r] at 78 49
\pinlabel {$f_2$} [l] at 120 49
\pinlabel {$Y''$} [l] at 110 30
\pinlabel {$Y'$} [t] at 96 16
\pinlabel {$f_1$} [r] at 153 49
\pinlabel {$f_2$} [l] at 194 49
\pinlabel {$H$} [l] at 186 29
\pinlabel {$f_1$} [r] at 284 49
\pinlabel {$f_2$} [l] at 325 49
\pinlabel {$G$} [l] at 314 29
\endlabellist}
\includegraphics[scale=1]{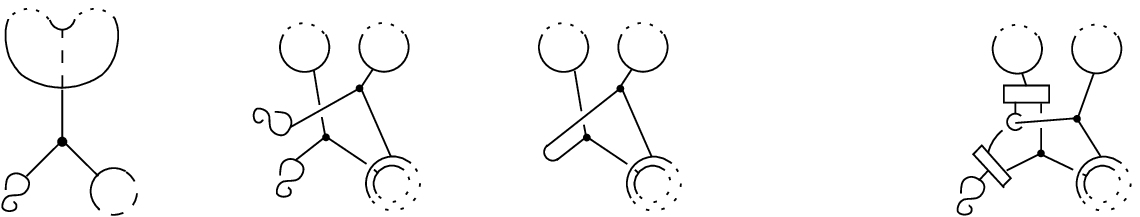}
\caption{ }\label{fig:scind}
\end{figure}
\begin{proof}
Using Habiro's move 7, followed by moves 11 and 5, 
we see that $Y$ is equivalent to the clasper $G$ shown in Figure \ref{fig:scind}. 
Thanks to Habiro's moves 5 and 1,  $G$ can be transformed into a clasper
to which we can apply Lemma \ref{lem:movec}. The H-graph resulting from this application
can be transformed into $H$ thanks to Lemma \ref{lem:edgeslide}.
\end{proof}

We shall need the following three consequences of Lemma \ref{lem:split_special}. 

\begin{corollary}\label{cor:split_2spe}
Suppose that, in the statement of Lemma \ref{lem:split_special} and Figure \ref{fig:scind}, the leaf $f'$ is special.  
Then we have
$$ (\Sigma\times I)_Y \stackrel{Y_{3}}{\sim}  (\Sigma\times I)_{Y'}\circ  (\Sigma\times I)_{Y''}\circ  (\Sigma\times I)_{P}, $$ 
where $P$ is obtained from $H$ by deleting the two parallel copies of $f'$ and connecting the edges by an untwisted band. 
\end{corollary}

\begin{proof}
It suffices to show that, in this situation, we have 
$(\Sigma\times I)_H \stackrel{Y_{3}}{\sim}  (\Sigma\times I)_P$. 
By Habiro's move 2, we can freely insert a pair of small Hopf-linked leaves in an edge of $H$  
so as to obtain a union of two $Y$-graphs with two parallel special leaves.  
The result then follows from Lemma \ref{lem:movec}   
\end{proof}

\begin{corollary}\label{cor:surf}
Suppose that, in the statement of Lemma \ref{lem:split_special}, 
the leaf $f_2$ bounds a genus $1$ surface disjoint from $Y$ 
with respect to which $f_2$ is $0$-framed.  Then we have 
$$ 
(\Sigma\times I)_Y \stackrel{Y_{3}}{\sim}  (\Sigma\times I)_{Y'}. 
$$ 
\end{corollary}

\begin{proof}
By Habiro's move 9, the $Y$-graph $Y''$ is equivalent to a degree $2$ graph clasper with a special leaf, 
and the H-graph $H$ is equivalent to a degree $3$ graph clasper.  
The result then follows from Lemma \ref{lem:special}.   
\end{proof}

\begin{corollary}\label{cor:+1}
Let $Y$ be a $Y$-graph with one special leaf in $(\Sigma\times I)$, and let $Y_+$ be obtained from $Y$ 
by replacing the special leaf with a $(+1)$-special leaf.  Then we have
$$ (\Sigma\times I)_Y \circ  (\Sigma\times I)_{Y_+}\stackrel{Y_{3}}{\sim}  (\Sigma\times I). $$ 
\end{corollary}

\begin{proof}
It suffices to observe that a $0$-framed unknot decomposes as a band sum of a $(-1)$-framed and a $(+1)$-framed unknot.  
Since a graph clasper with a $0$-special leaf is equivalent to the empty clasper, we have by Lemma \ref{lem:split_special} that
$$ 
(\Sigma\times I) \stackrel{Y_{3}}{\sim}  (\Sigma\times I)_{Y}\circ  (\Sigma\times I)_{Y_+}\circ  (\Sigma\times I)_{H}, 
$$ 
where $H$ is an H-graph with a special leaf.  The result then follows from Lemma \ref{lem:special}.  
\end{proof}

We now continue to give the main results of this appendix.

\begin{lemma}[Edge Twisting]\label{lem:twist_special}
 Let $Y$ be a $Y$-graph with one special leaf in $(\Sigma\times I)$, and let $Y'$ be obtained from $Y$ 
by inserting a half-twist in the $\ast$-marked edge of $Y$, see Figure \ref{fig:twist}.   
Then we have 
$$
 (\Sigma\times I)_{Y}\circ  (\Sigma\times I)_{Y'}\stackrel{Y_3}{\sim} (\Sigma\times I)_{H}, 
$$
where $H$ is a degree $2$ graph clasper in $(\Sigma\times I)$ as represented in Figure \ref{fig:twist}. 
  \begin{figure}[!h]
    {\labellist \small \hair 0pt 
    \pinlabel {$Y$} [l] at 29 27
    \pinlabel {$\ast$} [l] at 11 19
    \pinlabel {$H$} [l] at 110 27
    \pinlabel {$G$} [l] at 195 27
    \endlabellist}
    \includegraphics{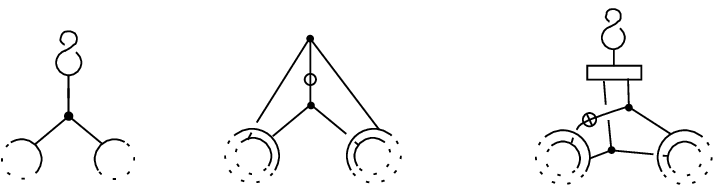}
    \caption{} \label{fig:twist}
  \end{figure}
\end{lemma}

\begin{proof}
By sliding an edge along a $(-1)$-framed unknot, 
a positive half-twist on that edge becomes a negative half-twist.
So, thanks to Lemma  \ref{lem:slide_special},
we can assume that the half-twist given to the $\ast$-marked edge of $Y$ is positive.
Then consider the clasper $G$ on the right-hand side of Figure \ref{fig:twist}.  
On one hand, we have by Lemma \ref{lem:movec}
$$ 
(\Sigma\times I)_{G}
\stackrel{Y_{3}}{\sim}  (\Sigma\times I)_{Y}\circ  (\Sigma\times I)_{Y'}\circ  (\Sigma\times I)_{H'}, 
$$
where $H'$ is a certain H-graph. By Lemma \ref{lem:twist_as}, 
we can replace $H'$  by another H-graph  which only differs from $H$ by a half-twist on an edge.
So we deduce that 
$$
(\Sigma\times I)_{G} \circ (\Sigma\times I)_{H} 
\stackrel{Y_{3}}{\sim}  (\Sigma\times I)_{Y}\circ  (\Sigma\times I)_{Y'}.
$$
On the other hand, Lemma \ref{lem:general} tells that
$(\Sigma\times I)_{G}$ is $Y_{3}$-equivalent to $(\Sigma\times I)_{C}$,
where $C$ is a degree two graph clasper with a special leaf.
So we deduce from Lemma \ref{lem:special} that $(\Sigma\times I)_{G}$ is $Y_3$-equivalent to $(\Sigma\times I)$.
The conclusion follows.  
\end{proof}

\begin{lemma}[Symmetry]\label{lem:as_special}
 Let $Y$ be a $Y$-graph with a special leaf in $(\Sigma\times I)$, and let $\overline{Y}$ be obtained from $Y$ 
 by reversing the cyclic order on the attaching regions of the three edges at its node, 
 as illustrated below. 
    $$
    \labellist \small \hair 0pt 
    \pinlabel {$Y$} [l] at 5 38
    \pinlabel {$\overline{Y}$} [l] at 100 38
    \endlabellist
    \includegraphics{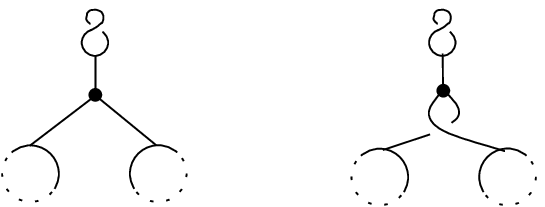}
    $$
 Then we have
 $
 (\Sigma\times I)_Y\stackrel{Y_3}{\sim} (\Sigma\times I)_{\overline{Y}}. 
 $
\end{lemma}

\begin{proof}
Observe that $\overline{Y}$ is isotopic to a copy of $Y$ with a positive half-twist inserted on each edge
that is incident to a non-special leaf (see \cite[page 92]{GGP} for example). 
Choose one of the two non-special leaves of $Y$, 
and denote by $Y'$ the $Y$-graph obtained from $Y$ by inserting a positive half-twist in the corresponding edge. 
Then by a double application of Lemma \ref{lem:twist_special} and
by using Lemma \ref{lem:twist_as} together with Lemma \ref{lem:edgeslide}, we obtain that 
$$
(\Sigma\times I)_{Y}\circ  (\Sigma\times I)_{Y'}\stackrel{Y_3}{\sim} (\Sigma\times I)_{\overline{Y}}\circ  (\Sigma\times I)_{Y'}.
$$ 
The result follows since the quotient monoid $\cyl/Y_3$ is a group.  
\end{proof}

\begin{lemma}[Doubling]\label{lem:doubling}
Let $Y$ be a $Y$-graph with one special leaf in $(\Sigma\times I)$.  Then we have
 $$ 
 \left((\Sigma\times I)_Y\right)^2 \stackrel{Y_3}{\sim} (\Sigma\times I)_H\circ (\Sigma\times I)_P
 $$
where $H$ and $P$ are the graph claspers represented in Figure \ref{fig:box3}. 
  \begin{figure}[!h]
   \begin{center}
    {\labellist \small \hair 0pt 
    \pinlabel {$Y$} [l] at 24 10
    \pinlabel {$H$} [l] at 98 10
    \pinlabel {$P$} [l] at 178 10
    \pinlabel {$G$} [l] at 297 36
    \pinlabel {$C$} [l] at 365 36
    \endlabellist}
    \includegraphics{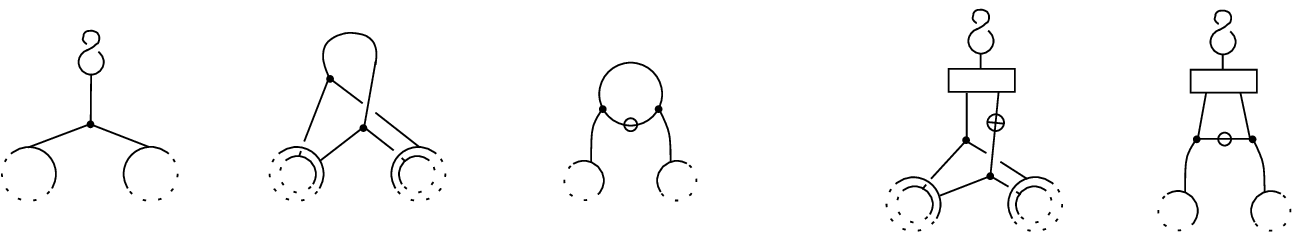}
    \caption{ } \label{fig:box3}
   \end{center}
  \end{figure}
\end{lemma}

\begin{proof}
Consider the claspers $G$ and $C$ shown on the right-hand side of Figure \ref{fig:box3}. 
On one hand, we have   by Lemma \ref{lem:general} that
$(\Sigma\times I)_{G}\stackrel{Y_{3}}{\sim}  (\Sigma\times I)_{C}.$
By Habiro's move 2, we can freely insert a pair of small Hopf-linked leaves in the half-twisted edge of $C$, 
so that we can apply Lemma \ref{lem:movec}, which implies that 
$(\Sigma\times I)_{C}\stackrel{Y_{3}}{\sim} (\Sigma\times I)_{P}$.   
On the other hand, a second application of Lemma \ref{lem:movec} shows that 
$$
 (\Sigma\times I)_{G}
 \stackrel{Y_{3}}{\sim}  (\Sigma\times I)_{Y}\circ  (\Sigma\times I)_{Y'}\circ  (\Sigma\times I)_{H'}, 
$$
where $Y'$ (respectively $H'$) only differs from $Y$  (respectively from $H$)
by a positive half-twist on the edge attached to the special leaf
(respectively on the edge connecting the two nodes).  
Since $Y'$ is isotopic to $Y$ and
since $(\Sigma\times I)_{H'}$ is the inverse of $(\Sigma\times I)_{H}$ modulo $Y_3$ by Lemma \ref{lem:twist_as},
we  conclude  that
$$
(\Sigma\times I)_{P} \circ (\Sigma\times I)_{H} \stackrel{Y_{3}}{\sim} (\Sigma\times I)_{G} \circ (\Sigma\times I)_{H}
\stackrel{Y_{3}}{\sim} \left((\Sigma\times I)_{Y}\right)^2.
$$\\[-1cm]
\end{proof}

We conclude with two particular cases of Lemma \ref{lem:doubling}.
These are easily derived by making use, again, of  Lemma \ref{lem:movec}.

\begin{corollary}\label{cor:2spe}
(1) If $Y$ is a $Y$-graph in $(\Sigma\times I)$ with two special leaves, then we have
       $$
       \left((\Sigma\times I)_Y\right)^2 \stackrel{Y_3}{\sim} (\Sigma\times I)_{\Phi},
       $$ 
       where $\Phi$ is the graph clasper represented in Figure \ref{fig:spe2}.\\
(2) If $Y_s$ is a $Y$-graph in $(\Sigma\times I)$ with three special leaves, then we have 
       $$
       \left((\Sigma\times I)_{Y_s}\right)^2 \stackrel{Y_3}{\sim} (\Sigma\times I)_{\Theta},
       $$ 
       where $\Theta$ is the graph clasper represented in Figure \ref{fig:spe2}.
       \begin{figure}[!h]
    {\labellist \small \hair 0pt 
    \pinlabel {(1)} [r] at -10 23
    \pinlabel {$Y$} [l] at 27 23
    \pinlabel {$\Phi$} [l] at 94 23
    \pinlabel {(2)} [r] at 190 23
    \pinlabel {$\Theta$} [l] at 247 23
    \endlabellist}
	\includegraphics[scale=1]{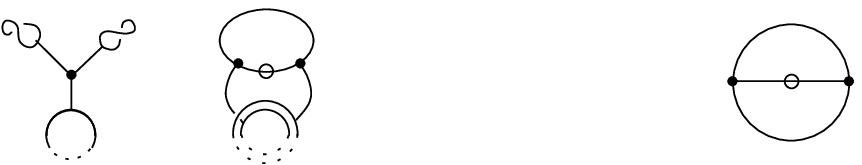}
	\caption{ } \label{fig:spe2}
	\end{figure}
\end{corollary}

\section{Linking numbers in a homology cylinder}\label{subsec:framing_numbers}

In this shorter appendix, we define framing and linking numbers in a homology cylinder over $\Sigma$
and, as a particular case, in the thickened surface $(\Sigma \times I)$.
We also refer to  \cite{CT} where linking numbers are defined in a more general context. 
In the sequel,  $M$ is a homology cylinder over $\Sigma$,
and the inverse of the isomorphism $m_{\pm,*}: H_1(\Sigma) \to H_1(M)$ is denoted by $p:H_1(M) \to H$. 

Let $K$ and $L$ be two disjoint oriented knots in $M$.
We denote by $\N(K)$ an open regular neighborhood of $K$ disjoint from $L$,
and by $i: M \setminus \N(K) \to M$ the inclusion.
The long exact sequence in homology for the pair $\big(M, M \setminus \N(K) \big)$
yields the short exact sequence of abelian groups
\begin{equation}
\label{eq:exterior_K}
\xymatrix{
0 \ar[r] & \Z \ar[r] & H_1\big(M \setminus \N(K) \big) \ar[r]^-{p \circ i_*} & H \ar[r] & 0
}
\end{equation}
where $1\in \Z$ is sent to the homology class of the oriented meridian $\mu(K)$ of $K$.
We still denote by  $m_\pm: \Sigma \to  M  \setminus \N(K)$ the corestriction of $m_\pm:\Sigma \to M$.
The homomorphism $m_{+,*}: H \to H_1(M \setminus \N(K))$ is a section of (\ref{eq:exterior_K})
which, in general, is different from $m_{-,*}: H \to H_1(M \setminus \N(K))$.
Thus there are two ``signed'' versions of the linking number of $L$ and $K$, which are denoted by
$$
\Lk_+(L,K) \in \Z  \quad \hbox{and} \quad \Lk_-(L,K) \in \Z
$$ 
and are defined by
\begin{equation}
\label{eq:Lk_pm}
[L] - m_{\pm,*} \big(p \circ i_*([L]) \big) = \Lk_\pm(L,K) \cdot \left[\mu(K) \right]
\ \in H_1\big( M \setminus \N(K) \big).
\end{equation}
In other words, $\Lk_\pm(L,K)$ is the homological obstruction to ``push'' $L \subset M \setminus K$ 
in a collar neighborhood of $\partial_\pm M$.

\begin{definition}
Let $K$ and $L$ be two disjoint oriented knots in $M$.
The \emph{linking number} of $K$ and $L$ is
$$
\Lk(L,K) := \frac{\Lk_+(L,K)+ \Lk_-(L,K)}{2} \ \in \frac{1}{2}\Z.
$$
\end{definition}

\noindent
When $\Sigma$ is a disk, this definition agrees with the usual notion of linking number in a homology $3$-sphere.

\begin{lemma}\label{lem:linking_numbers}
Linking numbers in $M$ have the following properties:
\begin{enumerate}
\item The two signed versions  of linking numbers are related by
\begin{eqnarray*}
\Lk_-(L,K)- \Lk_+(K,L) &=&0\\
 \Lk_-(L,K) - \Lk_+(L,K) &=& \omega\big(p([L]),p([K])\big)
\end{eqnarray*}
where $\omega: H \times H \to \Z$ is the intersection pairing of $\Sigma$.
\item When  $M=\Sigma \times I$, linking numbers can be computed
from regular projections on $\Sigma \times \{-1\}$ by the following local formulas:

$$
\Lk_-(L,K) = 
\sharp  \begin{array}{c}
\labellist
\small\hair 2pt
 \pinlabel {$K$} [br] at 1 19
 \pinlabel {$L$} [bl] at 19 20
\endlabellist
\includegraphics[scale=1.0]{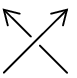}
\end{array}
-  \sharp  \begin{array}{c}
\labellist
\small\hair 2pt
 \pinlabel {$L$} [br] at 1 19
 \pinlabel {$K$} [bl] at 19 20
\endlabellist
\includegraphics[scale=1.0]{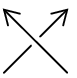}
\end{array}
= \Lk_+(K,L).
$$
\item Linking numbers are preserved by Torelli surgeries.
\end{enumerate}
\end{lemma}

\begin{proof}
Statement (3) means that, if a Torelli surgery $M\leadsto M_s$ (as defined in \S \ref{subsec:definition_relations})
is performed along a surface $S$ disjoint from $K$ and $L$, then we have 
$$
\Lk_\pm(K,L) = \Lk_\pm(K_s,L_s) \ \in \Z
$$
where $K_s,L_s\subset M_s$ denote the images of $K,L\subset M$.
This is easily deduced from (\ref{eq:Lk_pm})
by using the isomorphism $\Phi_s:H_1(M\setminus \N(K)) \to H_1(M_s \setminus \N(K))$ defined at (\ref{eq:iso_homology}).

Statement (2) also follows  from (\ref{eq:Lk_pm}). 
Assertion (1) is deduced from (2) and (3) 
using the fact that any homology cylinder $M$ is obtained from $(\Sigma \times I)$ by a Torelli surgery.
\end{proof}

Finally, we have the following notion of framing number in the homology cylinder $M$.

\begin{definition}
Let $F$ be a framed knot in $M$,
and let $F^\parallel$ be the parallel of $F$ induced by the framing.
The \emph{framing number} of $F$ is 
$$
\Fr(F) := \Lk(F,F^\parallel)\in \Z,
$$
where $F$ and $F^\parallel$ are oriented consistently.
\end{definition}

The following is deduced from the definition of framing number.

\begin{lemma}\label{lem:framed_connect}
Let $K$ and $L$ be two disjoint oriented framed knots in $M$,
and let $K \sharp L$ be a framed connected sum of $K$ and $L$.
Then, we have 
$$
\Fr(K\sharp L) = \Fr(K) + \Fr(L) + 2 \Lk(K,L) \in \Z.
$$
\end{lemma}

\section{The LMO homomorphism up to degree $2$}\label{subsec:LMO_deg2}

The  LMO homomorphism $Z: \cyl \to \jacobi(H_\Q)$ being defined from the LMO functor $\widetilde{Z}$,
one can compute $Z$ up to degree $2$ using some properties of $\widetilde{Z}$ and some of its values gathered in \cite[Table 5.2]{CHM}.
In particular, one can compute in that way the variation of the monoid homomorphism
$$
Z_2 : \Jcob  \longrightarrow \jacobi_2(H_\Q)
$$
under surgery along \emph{any} H-graph, or, along \emph{any} $Y$-graph with one special leaf.
In this appendix, we state such variation formulas
which involve the linking and framing numbers introduced in Appendix \ref{subsec:framing_numbers}.
Since we do not need these formulas in the paper, we omit the detail of the computations.
(Note that special cases of these formulas are needed in the proof of Proposition \ref{prop:values_d},
but they are derived there from the universal property of the LMO homomorphism.)

\begin{proposition}\label{prop:Z2_special}
Let $M\in \cyl$ and let  $Y$ be a $Y$-graph in $M$ with one special leaf and two oriented leaves $K,L$ as shown below:\\[-0.4cm]
$$
\labellist
\small\hair 2pt
 \pinlabel {$K$} [r] at 0 71
 \pinlabel {$L$} [l] at 78 72
\endlabellist
\includegraphics[scale=0.8]{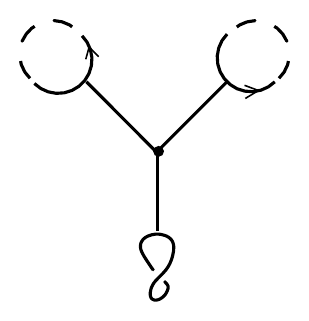}\\[-0.5cm]
$$
Then we have
\begin{eqnarray*}
Z_2\left(M_Y\right) - Z_2(M) &=& 
\frac{1}{2} \hn{k}{l}{k}{l} + \frac{1}{2}\left(\Lk(K,L)^2 + \Lk(K,L) - \Fr(K) \cdot \Fr(L)\right) \cdot \thetagraph \\
& & + \frac{\Fr(L)}{2} \phin{k}{k} + \frac{\Fr(K)}{2} \phin{l}{l}
- \left(\frac{1}{2}+ \Lk(K,L) \right) \cdot \phin{k}{l}
\end{eqnarray*}
where $k,l \in H$ denote the homology classes of $K,L$ respectively.
\end{proposition}

\begin{proposition}\label{prop:Z2_H}
Let $M\in \cyl$ and let $G$ be a  H-graph in $M$ 
whose leaves $L_1, \dots, L_4$ are oriented as shown below:\\[-0.4cm]
$$
\labellist
\small\hair 2pt
 \pinlabel {$L_4$} [l] at 70 9
 \pinlabel {$L_3$} [l] at 69 43
 \pinlabel {$L_2$} [r] at 0 41
 \pinlabel {$L_1$} [r] at 0 9
\endlabellist
\includegraphics[scale=1.3]{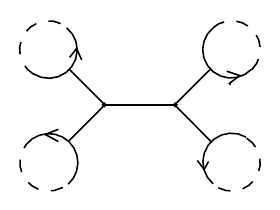}\\[-0.5cm]
$$
Then we have
\begin{eqnarray*}
Z_2\left(M_G\right) - Z_2(M) &=& 
\hn{l_1}{l_2}{l_3}{l_4} + \Lk(L_1,L_3) \phin{l_2}{l_4} + \Lk(L_2,L_4) \phin{l_1}{l_3} \\
&&- \Lk(L_1,L_4) \phin{l_2}{l_3}-\Lk(L_2,L_3) \phin{l_1}{l_4}\\
&& + \left(\Lk(L_1,L_4)\Lk(L_2,L_3) - \Lk(L_1,L_3) \Lk(L_2,L_4)\right)\cdot \thetagraph
\end{eqnarray*}
where $l_1,\dots,l_4\in H$ denote the homology classes of $L_1,\dots,L_4$ respectively.
\end{proposition}

Proposition  \ref{prop:Z2_special} and Proposition \ref{prop:Z2_H} can be used to give independent proofs
of the relations that are shown in Appendix \ref{sec:calc_clasper} by means of clasper calculus.
Besides,  one can deduce from them the following formulas for the core of the Casson invariant.

\begin{corollary}
With the notation of Proposition \ref{prop:Z2_special}, we have
\begin{eqnarray*}
d(M_Y)-d(M) &=&
 - 24\cdot  \Fr(K)  \Fr(L) \\ 
&&+ 12 \left(\Lk_+(K,L) + \Lk_+(K,L)^2  \right) + 12 \left(\Lk_-(K,L) + \Lk_-(K,L)^2  \right).
\end{eqnarray*}
\end{corollary}

\begin{corollary}
With the notation of Proposition \ref{prop:Z2_H}, we have
\begin{eqnarray*}
d(M_G)-d(M) &=&
8 \cdot \omega(l_1,l_2) \omega(l_3,l_4)\\ 
&&- 8 \cdot \big(\Lk_+(L_1,L_3)\Lk_+(L_2,L_4)+\Lk_-(L_1,L_3)\Lk_-(L_2,L_4)\big)\\
&& + 8 \cdot \big(\Lk_+(L_1,L_4)\Lk_+(L_2,L_3)+\Lk_-(L_1,L_4)\Lk_-(L_2,L_3)\big)\\
&&- 16 \cdot \big(\Lk_+(L_1,L_3)\Lk_-(L_2,L_4)+\Lk_-(L_1,L_3)\Lk_+(L_2,L_4)\big)\\
&& +16 \cdot \big(\Lk_+(L_1,L_4)\Lk_-(L_2,L_3)+\Lk_-(L_1,L_4)\Lk_+(L_2,L_3)\big).
\end{eqnarray*}
\end{corollary}

\noindent
(These formulas generalize those of Proposition \ref{prop:values_d}.)
Observe that the variation of $d$ under surgery along a $Y$-graph with one special leaf 
belongs to $24 \Z$, while the variation for an H-graph belongs to $8 \Z$ in general.
We deduce that the group $\Jcob/Y_3$ is not generated by the surgeries on $(\Sigma \times I)$
 along $Y$-graphs with one special leaf.

%
%
%
%
%
%

\bibliographystyle{abbrv}

\bibliography{Y3}

\end{document}